\documentclass{amsart}
\usepackage{indentfirst,mathrsfs}
\usepackage{amsfonts,amsmath,amssymb,amsthm,enumerate}
\usepackage{latexsym,amscd,multirow}
\usepackage{amsbsy}
\usepackage{color}
\newtheorem{thm}{Theorem}
\newtheorem{lem}{Lemma}
\newtheorem{cor}{Corollary}

\newtheorem{prop}{Proposition}

\newtheorem{defn}{Definition}

\newtheorem{cla}{Claim}

\begin{document}
\title{Self-adjoint boundary-value problems of automorphic forms}
\author{Adil Ali}
\email{alix0114@umn.edu}
\address{127 Vincent Hall, 206 Church Street Southeast Minneapolis,MN,55414}
\begin{abstract} 
We apply some ideas of Bombieri and Garrett to construct natural self-adjoint operators on spaces of automorphic forms whose only possible discrete spectrum is $\lambda_{s}$ for $s$ in a subset of on-line zeros of an $L$-function, appearing as a compact period of cuspidal-data Eisenstein series on $GL_{4}$.  These ideas have their origins in results of Hejhal and Colin de Verdi\'ere.  In parallel with the $GL(2)$ case, the corresponding pair-correlation and triple-correlation results limit the fraction of on-the-line zeros that can appear in this fashion.
\end{abstract}
\maketitle

\tableofcontents

\section{Introduction}
We apply the spectral theory of automorphic forms to the study of zeros of $L$-functions. A refined version of the spectral theory of automorphic forms plausibly has bearing on zeros of automorphic $L$-functions and other periods. This is powerfully illustrated by the following example, which is a much simpler analogue of our present result.  In 1977, H. Haas [Haas 1977] attempted to numerically compute eigenvalues $\lambda$ of the invariant Laplacian
$$\Delta\;=\;y^{2}(\frac{\partial^{2}}{\partial x^{2}}+\frac{\partial^{2}}{\partial y^{2}})$$
on $SL_{2}(\mathbb{Z})\backslash\mathfrak{H}$, parametrized as $\lambda_{w}=w(w-1)$.  Haas listed the $w$-values, intending to solve the differential equation $$(\Delta-\lambda_{w})u=0$$  H. Stark and D. Hejhal [Hejhal 1981] observed zeros of $\zeta$ and of an $L$-function on the list.  This suggested an approach to the Riemann Hypothesis, hoping that  zeros $w$ of $\zeta$ would be in bijection with eigenvalues $\lambda_{w}=w(w-1)$ of $\Delta$.  Since a suitable version of $\Delta$ is a self-adjoint, non-positive operator, these eigenvalues would necessarily be non-positive also, forcing either $\text{Re}(w)=\frac{1}{2}$ or $w\in [0,1]$.  Hejhal attempted to reproduce Haas' list with more careful computations, but the zeros failed to appear on Hejhal's list.
Hejhal realized that Haas had solved the inhomogeneous equation $$(\Delta-\lambda_{w})u=\delta_{\omega}^{\text{afc}}$$ allowing a multiple of an automorphic Dirac $\delta_{\omega}^{\text{afc}}$ on the right hand side. Here $\omega$ is a cube root of unity, and $\delta_{\omega}^{\text{afc}}(f)=f(\omega)$ for an $SL_{2}(\mathbb{Z})$-automorphic waveform $f$. However, since solutions $u_{w}$ of $(\Delta-\lambda)u=\delta_{\omega}^{\text{afc}}$ are not genuine eigenfunctions of the Laplacian, this no longer implied non-positivity of the eigenvalues.

The natural question was whether the Laplacian could be modified so as to exhibit a fundamental solution as a legitimate eigenfunction for the perturbed operator.  That is, one would want a variant $\Delta'$ for which
$$(\Delta'-\lambda_{w})u_{w}\;=\;0\iff (\Delta-\lambda_{w})u_{w}\;=\;C\cdot\delta_{\omega}^{\text{afc}}$$
Because of Y. Colin de Verdi\`{e}re's  argument for meromorphic continuation of Eisenstein series [CdV 1981], it was anticipated that $\Delta'=\Delta^{\text{Fr}}$ would be a fruitful choice for the Friedrichs extension of a suitably chosen restriction. $\Delta^{\text{Fr}}$ is self-adjoint, and therefore symmetric.  This gave glimpses of progress toward the Riemann hypothesis.

Friedrichs extensions have the desired properties and they played an essential role in another story, namely Colin de Verdi\`{e}re's meromorphic continuation of Eisenstein series, though there, the distribution that appeared was the evaluation of constant term at height $y=a$. There, the spaces of interest were the orthogonal complements $L^{2}(\Gamma\backslash\mathfrak{H})_{a}$ to the spaces of pseudo-Eisenstein series with test function data supported on $[a,\infty)$.  $\Delta_{a}$ was $\Delta$ with domain $C_{c}^{\infty}(\Gamma\backslash\mathfrak{H})$ and constant term vanishing above height $y=a$. $\Delta^{\text{Fr}}$ was the Friedrichs extension of $\Delta_{a}$ to a self-adjoint operator on $L^{2}(\Gamma\backslash\mathfrak{H})_{a}$.
In this way, a Friedrichs extension attached to the distribution on $\Gamma\backslash\mathfrak{H}$ given by
$$T_{a}(f)\;=\;(c_{P}f)(ia)$$
has all eigenfunctions inside a $+1$-index  global automorphic Sobolev space, defined as the completion of $C_{c}^{\infty}(\Gamma\backslash\mathfrak{H})$ with respect to the $+1$-Sobolev norm $$|f|_{H^{1}}=\langle (1-\Delta)f,f\rangle^{\frac{1}{2}}$$ The Dirac $\delta$ on a two-dimensional manifold lies in a global Sobolev space $H^{-1-\epsilon}$ with index $-1-\epsilon$ for all $\epsilon>0$, but not in $H^{-1}$, so by elliptic regularity, a fundamental solution lies in the $+1-\epsilon$-Sobolev space.  This implies that a fundamental solution could not be an eigenfunction for any Friedrichs extension of a restriction of $\Delta$ described by boundary conditions.

The automorphic Dirac $\delta_{\omega}^{\text{afc}}$ is an example of a \textit{period functional}.  Periods of automorphic forms have been studied extensively: after all, Mellin transforms of cuspforms are noncompact periods.  Hecke and Maass were aware of Eisenstein series periods: in effect, Hecke treated  finite sums over Heegner points attached to \textit{negative} fundamental discriminants, and Maass treated compact geodesic periods attached to \textit{positive} fundamental discriminants.  A simple example is given by 
$$E_{s}(i)\;=\;\frac{\zeta_{\mathbb{Q}(i)}(s)}{\zeta_{\mathbb{Q}}(2s)}$$
More generally, let $\ell$ a quadratic field extension of a global field $k$ of characteristic not $2$. Let $G=GL_{2}(k)$, and let $H$ be a copy of $\ell^{\times}$ inside $G$.  The period of an Eisenstein series $E_{s}\;=\;\sum_{\gamma\in P_{k}\backslash G_{k}}\varphi(\gamma g)$ along $H$ is defined by the compactly-supported integral
$$\text{period of}\;E_{s}\;\text{along}\;H\;=\;\int_{Z_{\mathbb{A}}H_{k}\backslash H_{\mathbb{A}}}E_{s}$$
Via Iwasawa-Tate integrals, 
$$\int_{Z_{\mathbb{A}}H_{k}\backslash H_{\mathbb{A}}}E_{s}\;=\;\frac{\xi_{\ell}(s)}{\xi_{k}(2s)}$$
Noncompact periods have been studied extensively.  Let $G$ be a reductive group over a number field $F$, and let $H\subset G$ be a subgroup obtained as the fixed point set of an involution $\theta$. [Jacquet-Lapid-Rogowski 1997] studied the period integral
$$\Pi^{H}(\varphi)\;=\;\int_{H(F)\backslash H(\mathbb{A})}\varphi(h)\;dh$$
The authors use a regularization procedure and a relative trace formula to obtain an Euler product for $\Pi(E)$, where $E$ is an Eisenstein series.

This paper examines the discrete spectrum of a Friedrichs extension $\widetilde{\Delta}_{\theta}$ associated to a compactly-supported $GL_{4}(\mathbb{Z})$-invariant distribution $\tilde{\theta}$ on $G=GL(4)$, whose projection $\theta$ to the subspace of $L^{2}(GL_{4}(\mathbb{Z})\backslash GL_{4}(\mathbb{R})/O_{4}(\mathbb{R}))$ spanned by $2,2$ pseudo-Eisenstein series with fixed cuspidal data $f$ and $\overline{f}$ and the residue of this Eisenstein series, a Speh form. This distribution lies in the $-1$ index Sobolev space.  We prove that the parameters $w$ of the discrete spectrum $\lambda_{w}=w(w-1)$, if any, of $\widetilde{\Delta}_{\theta}$ interlace with the zeros of the constant term of the $2,2$ Eisenstein series $E_{f,\overline{f},s}^{P}$ where $f$ is a $GL(2)$ cuspform.  Such spacing is too regular to be compatible with the corresponding pair-correlation and triple-correlation conjectures, and this powerfully constrains the number of zeros $w$ of $\theta E_{1-w}$ appearing in the discrete spectrum of $\widetilde{\Delta}_{\theta}$.  In particular, the discrete spectrum is presumably sparse.

\section{Spectral Theory}
We follow [Langlands 1976], [MW 1990], [MW 1989], and [Garrett 2012]. Fix, once and for all, $K_{\infty}=O_{4}(\mathbb{R})$, and $K_{v}=GL_{4}(\mathbb{Z}_{v})$ for non-archimedean places $v$.  Let $\mathfrak{z}$ be the center of the enveloping algebra of $G_{\infty}=GL_{4}(\mathbb{R})$. 
\begin{defn}Given a parabolic $P$ in $G=GL_{4}$ and a function $f$ on $Z_{\mathbb{A}}G_{k}\backslash G_{\mathbb{A}}$, the constant term of $f$ along $P$ is
$$c_{P}f(g)=\int_{N_{k}\backslash N_{\mathbb{A}}}f(ng)\;dn$$
where $N$ is the unipotent radical of $P$.\end{defn}  
We will let $k=\mathbb{Q}$ throughout. An automorphic form is a \textit{cuspform} if, for all parabolics $P$, the constant term along $P$ is zero. This is the Gelfand condition (in the weak sense). Since the right $G_{\mathbb{A}}$-action commutes with taking constant terms, the space of functions $L^{2}_{\text{cusp}}(Z_{\mathbb{A}}G_{k}\backslash G_{\mathbb{A}})$ satisfying the Gelfand condition is $G_{\mathbb{A}}$-stable, and so is a sub-representation of $L^{2}(Z_{\mathbb{A}}G_{k}\backslash G_{\mathbb{A}})$. We note that there are non-$K_{v}$-finite vectors in $L^{2}(Z_{\mathbb{A}}G_{k}\backslash G_{\mathbb{A}})$.  R. Godement, A. Selberg, I. Gelfand and I. I. Piatetski-Shapiro showed that integral operators attached to test functions on $L^{2}_{\text{cusp}}(Z_{\mathbb{A}}G_{k}\backslash G_{\mathbb{A}})$ are compact.  Specifically, for $\varphi\in C_{c}^{\infty}(G_{\mathbb{A}})$ which is right $K$-invariant, the operator $$f\rightarrow\varphi\cdot f$$ gives a compact operator from $L^{2}_{\text{cusp}}(Z_{\mathbb{A}}G_{k}\backslash G_{\mathbb{A}})$ to itself. Here $$(\varphi\cdot f)(y)\;=\;\int\limits_{Z_{\mathbb{A}}G_{k}\backslash G_{\mathbb{A}}}\varphi(x)\cdot f(yx)\;dx$$  By the spectral theorem for compact operators, this sub-representation decomposes into a direct sum of irreducibles, each with finite multiplicity.  The remainder of $L^{2}$ is decomposed as follows.

We classify non-cuspidal automorphic forms according to their cuspidal support, i.e. the smallest parabolics on which they have non-zero constant term.  In $GL(4)$ there are four associate classes of proper parabolic subgroups.  There is $P^{4}=GL_{4}$, $P^{2,1,1}$, $P^{1,2,1}$, $P^{1,1,2}$, the maximal proper parabolic subgroups $P^{3,1}$, $P^{1,3}$ and $P^{2,2}$, and the standard minimal parabolic subgroup $P^{1,1,1,1}$. 
\begin{defn} A pseudo-Eisenstein series is a function of the form
$$\Psi_{\varphi}(g)\;=\;\sum_{\gamma\in P_{k}\backslash G_{k}}\varphi(\gamma\cdot g)$$
where $\varphi$ is a continuous function on $Z_{\mathbb{A}}N_{\mathbb{A}}M_{k}\backslash G_{\mathbb{A}}$ with cuspidal data on the Levi component.\end{defn}  For example, given the $2,2$ parabolic, the function out of which the pseudo-Eisenstein series is constructed is
$$\varphi_{\phi,f_{1}\otimes f_{2}}(\left( \begin{array}{cc}
A & * \\
0& D\end{array} \right))\;=\;\phi(\left|\frac{\text{det}A}{\text{det}D}\right|^{2})\cdot f_{1}(A)\cdot f_{2}(D)$$
where $\phi$ is a compactly-supported, smooth function on $\mathbb{R}$ and $f_{1}$ and $f_{2}$ are cuspforms on $GL_{2}$ with trivial central character.  For the $3,1$ parabolic, consider the function 
$$\varphi_{\phi,f_{1}\otimes f_{2}}(\left( \begin{array}{cc}
A & * \\
0& d\end{array} \right))\;=\;\phi(\left|\frac{\text{det}A}{d^{3}}\right|)\cdot f_{1}(A)$$
where $A\in GL_{3}$ and $f_{1}$ is a cuspform on $GL_{3}$. For the $2,1,1$ parabolic, let
$$\varphi_{f,\phi_{1},\phi_{2}}(\left( \begin{array}{ccc}
A & 0&0 \\
0& b&0\\
0&0&c\end{array} \right))\;=\;f(A)\cdot\phi_{1}(\frac{\text{det}A}{b^{2}})\cdot \phi_{2}(\frac{\text{det}A}{c^{2}})$$
The $1,1,1,1$-pseudo-Eisenstein series is discussed later.

\begin{prop} In the following, abbreviate $\varphi_{\phi,f_{1}\otimes f_{2}}$ by $\varphi$. For any square-integrable automorphic form $f$ and any pseudo-Eisenstein series $\Psi_{\varphi}^{P}$, with $P$ a parabolic subgroup
$$\langle f,\Psi_{\varphi}^{P}\rangle_{Z_{\mathbb{A}}G_{k}\backslash G_{\mathbb{A}}}\;=\;\langle c_{P}f,\varphi\rangle_{Z_{\mathbb{A}}N_{\mathbb{A}}^{P}M_{k}^{P}\backslash G_{\mathbb{A}}}$$
\end{prop}

\begin{proof}  The proof involves a standard unwinding argument.  Let $N^{P}$ and $M^{P}$ denote the unipotent radical and Levi component of $P$, respectively. Observe that
$$\langle f,\Psi_{\varphi}^{P}\rangle_{Z_{\mathbb{A}}G_{k}\backslash G_{\mathbb{A}}}\;=\;\int\limits_{Z_{\mathbb{A}}G_{k}\backslash G_{\mathbb{A}}}f(g)\cdot\overline{\Psi_{\varphi}^{P}(g)}\;dg\;=\int\limits_{Z_{\mathbb{A}}G_{k}\backslash G_{\mathbb{A}}}f(g)(\sum_{\gamma\in P_{k}\backslash G_{k}}\overline{\varphi(\gamma\cdot g)})\;dg$$
This is
\begin{align*}&=\int\limits_{Z_{\mathbb{A}}P_{k}\backslash G_{\mathbb{A}}}f(g)\overline{\varphi(g)}\;dg=\int\limits_{Z_{\mathbb{A}}N_{k}M_{k}\backslash G_{\mathbb{A}}}f(g)\overline{\varphi(g)}\;dg\;\\&=\;\int\limits_{Z_{\mathbb{A}}N_{\mathbb{A}}M_{k}\backslash G_{\mathbb{A}}}\int\limits_{N_{k}\backslash N_{\mathbb{A}}}f(ng)\overline{\varphi(ng)}\;dn\;dg\\
&=\int\limits_{Z_{\mathbb{A}}N_{\mathbb{A}}M_{k}\backslash G_{\mathbb{A}}}(\int\limits_{N_{k}\backslash N_{\mathbb{A}}}f(ng)\;dn)\overline{\varphi(g)}\;dg\\
&=\langle c_{P}f,\varphi\rangle_{Z_{\mathbb{A}}N_{\mathbb{A}}^{P}M_{k}^{P}\backslash G_{\mathbb{A}}}
\end{align*}
\end{proof}
From this adjointness relation, we have the following
\begin{cor} A square-integrable automorphic form is a cuspform if and only if it is orthogonal to all pseudo-Eisenstein series.\end{cor} Since the critical issues arise at the archimedean place, we consider the real Lie group.  To this end, let $G=PGL_{4}(\mathbb{R})$, $\Gamma=PGL_{4}(\mathbb{Z})$.  \begin{defn} The standard minimal parabolic $B$ is defined as the subgroup $$B=P^{1,1,1,1}$$ of upper-triangular matrices, with standard Levi component $A$, unipotent radical $N$, and Weyl group $W$, the latter represented by permutation matrices.\end{defn}  Let $A^{+}$ be the image in $G$ of positive diagonal matrices.  Consider characters on $B$ of the form
$$\chi\;=\;\chi_{s}: (\left( \begin{array}{cccc}
a_{1} & * & *&* \\
0&a_{2}& *&*\\
0&0&a_{3}&*\\
0&0&0&a_{4}\end{array} \right))\;=\;|a_{1}|^{s_{1}}\cdot |a_{2}|^{s_{2}}\cdot |a_{3}|^{s_{3}}\cdot |a_{4}|^{s_{4}}$$
For the character to descend to $PGL_{n}$, necessarily $s_{1}+s_{2}+s_{3}+s_{4}=0$.  \begin{defn}The standard spherical vector is
$$\varphi_{s}^{\text{sph}}(pk)\;=\;\chi_{s}(p)$$
and the spherical Eisenstein series is
$$E_{s}(g)\;=\;\sum_{\gamma\in B\cap\Gamma\backslash\Gamma}\varphi_{s}^{\text{sph}}(\gamma\cdot g)$$\end{defn}
The spherical Eisenstein series is convergent for $\text{Re}(s)\gg 1$ and meromorphically continued to an entire function of $s$ as in [Langlands 544, Appendix 1].
The function $f\rightarrow c_{B}f(g)$ is left $N(B\cap\Gamma)$-invariant.  

Recall that for  $\varphi\in C_{c}^{\infty}(N(B\cap\Gamma)\backslash G)^{K}\approx C_{c}^{\infty}(A^{+})$, letting $\langle,\rangle_{X}$ be the pairing of distributions and test functions on a space $X$, the pseudo-Eisenstein series $\Psi_{\varphi}(g)$ enters the adjunction relation
$$\langle c_{B}f,\varphi\rangle_{N(B\cap\Gamma)\backslash G}\;=\;\langle f,\Psi_{\varphi}\rangle_{\Gamma\backslash G}$$
That is, $\varphi\rightarrow \Psi_{\varphi}$ is adjoint to $f\rightarrow c_{B}f$.  Then $c_{B}f=0$ is equivalent to $$\langle f,\Psi_{\varphi}\rangle_{\Gamma\backslash G}=0$$
for all $\varphi$.

\begin{prop} The pseudo-Eisenstein series $\Psi_{\varphi}$ admits a $W$-symmetric expansion as an integral of Eisenstein series.  That is,
$$\Psi_{\varphi}\;=\;\frac{1}{|W|}\frac{1}{(2\pi i)^{\text{dim}\mathfrak{a}}}\int_{\rho+i\mathfrak{a}^{*}}E_{s}\cdot \langle \Psi_{\varphi}, E_{2\rho-s}\rangle_{\Gamma\backslash G}\;ds$$
\end{prop}
\begin{proof} To decompose the pseudo-Eisenstein series $\Psi_{\varphi}$ as an integral of minimal-parabolic Eisenstein series, begin with Fourier transform on the Lie algebra $\mathfrak{a}\approx \mathbb{R}^{n-1}$ of $A^{+}$.  Let $\langle,\rangle:\mathfrak{a}^{*}\times\mathfrak{a}\rightarrow\mathbb{R}$ be the $\mathbb{R}$-bilinear pairing of $\mathfrak{a}$ with its $\mathbb{R}$-linear dual $\mathfrak{a}^{*}$.  For $f\in C_{c}^{\infty}(\mathfrak{a})$, the Fourier transform is
$$\widehat{f}(\xi)\;=\;\int_{\mathfrak{a}}e^{-i\langle x,\xi\rangle}f(x)\;dx$$
Fourier inversion is
$$f(x)\;=\;\frac{1}{(2\pi)^{\text{dim}\mathfrak{a}}}\int_{\mathfrak{a^{*}}}e^{i\langle x,\xi\rangle}\widehat{f}(\xi)\;d\xi$$
Let $\text{exp}:\mathfrak{a}\rightarrow A^{+}$ be the Lie algebra exponential, and $\text{log}:A^{+}\rightarrow \mathfrak{a}$ the inverse.  Given $\varphi\in C_{c}^{\infty}(A^{+})$, let $f=\varphi\circ \text{exp}$ be the corresponding function in $C_{c}^{\infty}(\mathfrak{a})$.  The (multiple) Mellin transform $M\varphi$ of $\varphi$ is the Fourier transform of $f$:
$$\mathscr{M}\varphi(i\xi)\;=\;\widehat{f}(\xi)$$
Mellin inversion is Fourier inversion in these coordinates:
$$\varphi(\text{exp} x)\;=\;f(x)\;=\;\frac{1}{(2\pi)^{\text{dim}\mathfrak{a}}}\int_{\mathfrak{a}^{*}}e^{i\langle\xi,x\rangle}\widehat{f}(\xi)\;d\xi\;=\;\frac{1}{(2\pi)^{\text{dim}\mathfrak{a}}}\int_{\mathfrak{a}^{*}}e^{i\langle \xi,x\rangle} \mathscr{M}\varphi(i\xi)\;d\xi$$
Extend the pairing $\langle,\rangle$ on $\mathfrak{a}^{*}\times\mathfrak{a}$ to a $\mathbb{C}$-bilinear pairing on the complexification.  Use the convention
$$(\text{exp})^{i\xi}\;=\;e^{i\langle \xi,x\rangle}\;=\;e^{\langle i\xi,x\rangle}$$
With $a=\text{exp}x\in A^{+}$, Mellin inversion is
$$\varphi(a)\;=\;\frac{1}{(2\pi)^{\text{dim}\mathfrak{a}}}\int_{\mathfrak{a}^{*}}a^{i\xi}\mathscr{M}\varphi(i\xi)\;d\xi\;=\;\frac{1}{(2\pi i)^{\text{dim}\mathfrak{a}}}\int_{i\mathfrak{a}^{*}}a^{s}\mathscr{M}\varphi(s)\;ds$$
With this notation, the Mellin transform itself is
$$\mathscr{M}\varphi(s)\;=\;\int_{A^{+}}a^{-s}\varphi(a)\;da$$
Since $\varphi$ is a test function, its Fourier-Mellin transform is entire on $\mathfrak{a}^{*}\otimes_{\mathbb{R}}\mathbb{C}$.  Thus, for any $\sigma\in\mathfrak{a}^{*}$, Mellin inversion can be written
$$\varphi(a)\;=\;\frac{1}{(2\pi i)^{\text{dim}\mathfrak{a}}}\int_{\sigma +i\mathfrak{a}^{*}}a^{s}\mathscr{M}\varphi(s)\;ds$$
Identifying $N(B\cap\Gamma)\backslash G/K\approx A^{+}$, let $g\rightarrow a(g)$ be the function that picks out the $A^{+}$ component in an Iwasawa decomposition $G=NA^{+}K$.  For $\sigma\in \mathfrak{a}^{+}$ suitable for convergence, the following rearrangement is legitimate,
\begin{align*}&\Psi_{\varphi}(g)\;=\;\sum_{\gamma\in (B\cap\Gamma)\backslash\Gamma}\varphi(a(\gamma\circ g))\;=\;\sum_{\gamma\in B\cap\Gamma\backslash\Gamma}\frac{1}{(2\pi i)^{\text{dim}\mathfrak{a}}}\int_{\sigma+i\mathfrak{a}^{*}}a(\gamma g)^{s}\mathscr{M}\varphi(s)\;ds\\
&=\;\frac{1}{(2\pi i)^{\text{dim}\mathfrak{a}}}\int_{\sigma+i\mathfrak{a}^{*}}\big(\sum_{\gamma\in B\cap\Gamma\backslash\Gamma}a(\gamma g)^{s}\big)\mathscr{M}\varphi(s)\;ds\;=\;\frac{1}{(2\pi i)^{\text{dim}\mathfrak{a}}}\int_{\sigma+i\mathfrak{a}^{*}}E_{s}(g)\mathscr{M}\varphi(s)\;ds\\
\end{align*}
This does express the pseudo-Eisenstein series as a superposition of Eisenstein series, as desired.  However, the coefficients $\mathscr{M}\varphi$ are not expressed in terms of $\Psi_{\varphi}$ itself.  This is rectified as follows.  Letting $\rho$ denote the half-sum of positive roots,
\begin{align*}&\langle f,E_{s}\rangle_{\Gamma\backslash G}\;=\;\int_{\Gamma\backslash G}f(g)E_{s}(g)\;=\;\int_{B\cap\Gamma\backslash G}f(g)a(g)^{s}\;dg\\
&=\;\int_{N(B\cap\Gamma)\backslash G}\int_{N\cap\Gamma\backslash N}f(ng)a(ng)^{s}\;dg\;=\;\int_{N(B\cap\Gamma)\backslash G}c_{B}f(g)a(g)^{s}\;dg\\&=\;\int_{A^{+}}c_{B}f(a)a^{s}\frac{da}{a^{2\rho}}\;
=\;\int_{A^{+}}c_{B}f(a)a^{-(2\rho-s)}\;da\;=\;\mathscr{M}c_{B}f(2\rho-s)\\
\end{align*}
That is, with $f=\Psi_{\varphi}$,
$$\langle \Psi_{\varphi},E_{s}\rangle_{\Gamma\backslash G}\;=\;\mathscr{M}c_{B}\Psi_{\varphi}(2\rho-s)$$
On the other hand, a similar unwinding of the pseudo-Eisenstein series, and the recollection of the constant term $c_{B}E_{s}$, gives
\begin{align*}&\langle \Psi_{\varphi},E_{s}\rangle_{\Gamma\backslash G}\;=\;\int_{B\cap\Gamma\backslash G}\varphi(g)E_{s}(g)\;dg\;=\;\int_{N(B\cap\Gamma)\backslash G}\int_{N\cap\Gamma\backslash N}\varphi(ng)E_{s}(ng)\;dg\\
&\;=\;\int_{N(B\cap\Gamma)\backslash G}\varphi(g)c_{B}E_{s}(g)\;dg\;=\;\int_{A^{+}}\varphi(a)c_{B}E_{s}(a)\;\frac{da}{a^{2\rho}}\\&\;=\;\int_{A^{+}}\varphi(a)\sum_{w}c_{w}(s)a^{w\cdot s}\;\frac{da}{a^{2\rho}}\\
&\;=\;\sum_{w}c_{w}(s)\int_{A^{+}}\varphi(a)a^{-(2\rho-w\cdot s)}\;da\;=\;\sum_{w}c_{w}(s)\mathscr{M}\varphi(2\rho-w\cdot s)\\
\end{align*}
Combining these,
$$\mathscr{M}c_{B}\Psi_{\varphi}(2\rho-s)\;=\;\langle \Psi_{\varphi},E_{s}\rangle_{\Gamma\backslash G}\;=\;\sum_{w}c_{w}(s)\mathscr{M}\varphi(2\rho-w\cdot s)$$
Replacing $s$ by $2\rho-s$, noting that $2\rho-w\cdot(2\rho-s)=w\cdot s$,
$$\mathscr{M}c_{B}\Psi_{\varphi}(s)\;=\;\sum_{w}c_{w}(2\rho-s)\mathscr{M}\varphi(w\cdot s)$$
To convert the expression
$$\Psi_{\varphi}(g)\;=\;\frac{1}{(2\pi i)^{\text{dim}\mathfrak{a}}}\int_{\sigma+i\mathfrak{a}^{*}}E_{s}(g)\mathscr{M}\varphi(s)\;ds$$
into a $W$-symmetric expression, to obtain an expression in terms of $c_{B}\Psi_{\varphi}$, we must use the functional equations of $E_{s}$.  However, $\sigma+i\mathfrak{a}^{*}$ is $W$-stable only for $\sigma=\rho$.  Thus, the integral over $\sigma+i\mathfrak{a}^{*}$ must be viewed as an iterated contour integral, and moved to $\rho+i\mathfrak{a}^{*}$.  
\begin{align*}\Psi_{\varphi}&\;=\;\frac{1}{|W|}\sum_{w}\frac{1}{(2\pi i)^{\text{dim}\mathfrak{a}^{*}}}\int_{\rho+i\mathfrak{a}^{*}}E_{w\cdot s}\mathscr{M}\varphi(w\cdot s)\;ds\\&\;=\frac{1}{|W|}\frac{1}{(2\pi i)^{\text{dim}\mathfrak{a}}}\int_{\rho+i\mathfrak{a}^{*}}E_{s}\big(\sum_{w}\frac{1}{c_{w}(s)}\mathscr{M}\varphi(w\cdot s)\big)\;ds\end{align*}
On $\rho+i\mathfrak{a}^{*}$, we have $\frac{1}{c_{w}(s)}=c_{w}(2\rho-s)$.  Therefore,
$$\sum_{w}\frac{1}{c_{w}(s)}\mathscr{M}\varphi(w\cdot s)\;=\;\sum_{w}c_{w}(2\rho-s)\mathscr{M}\varphi(w\cdot s)\;=\;\mathscr{M}c_{B}\Psi_{\varphi}(s)$$
This gives the desired spectral decomposition,
\begin{align*}&\Psi_{\varphi}\;=\;\frac{1}{|W|}\frac{1}{(2\pi i)^{\text{dim}\mathfrak{a}}}\int_{\rho+i\mathfrak{a}^{*}}E_{s}\cdot \mathscr{M}\Psi_{\varphi}(s)\;ds\\&=\;\frac{1}{|W|}\frac{1}{(2\pi i)^{\text{dim}\mathfrak{a}}}\int_{\rho+i\mathfrak{a}^{*}}E_{s}\cdot\langle \Psi_{\varphi},E_{2\rho-s}\rangle_{\Gamma\backslash G}\;ds\end{align*}
\end{proof}
\begin{prop} The map $f\rightarrow(s\rightarrow \langle f,E_{s}\rangle)$ is an inner-product-preserving map from the Hilbert-space span of the pseudo-Eisenstein series to its image in $L^{2}(\rho+i\mathfrak{a})$.\end{prop}
\begin{proof}Let $f\in C_{c}^{\infty}(\Gamma\backslash G)$, $\varphi\in C_{c}^{\infty}(N\backslash G)$, and assume $\Psi_{\varphi}$ is orthogonal to residues of $E_{s}$ above $\rho$.  Using the expression for $\Psi_{\varphi}$ in terms of Eisenstein series,
\begin{align*}&\langle \Psi_{\varphi},f\rangle\;=\;\langle \frac{1}{|W|}\frac{1}{(2\pi i)^{\text{dim}\mathfrak{a}}}\int_{\rho+i\mathfrak{a}^{*}}\langle \Psi_{\varphi},E_{2\rho-s}\rangle\cdot E_{s}ds,f\rangle\\&=\;\frac{1}{|W|}\frac{1}{(2\pi i)^{\text{dim}\mathfrak{a}}}\int_{\rho+i\mathfrak{a}^{*}}\langle \Psi_{\varphi},E_{2\rho-s}\rangle\cdot\langle E_{s},f\rangle \;ds\end{align*}
\end{proof}

The map $$\Psi_{\varphi}\rightarrow \langle \Psi_{\varphi},E_{2\rho-s}\rangle$$ with $s=\rho+it$ and $t\in\mathfrak{a}^{*}$, produces functions $$u(t)=\langle \Psi_{\varphi},E_{\rho-it}\rangle$$ satisfying
\begin{align*}&u(wt)\;=\;\langle \Psi_{\varphi},E_{2\rho-w\cdot s}\rangle \;=\;\langle \Psi_{\varphi},E_{w\cdot (2\rho-s)}\rangle\;=\;\langle \Psi_{\varphi},\frac{E_{2\rho-s}}{c_{w}(2\rho-s)}\rangle\\&=\;c_{w}(s)\cdot u(t)\;\;\;\;\text{for all}\;w\in W\end{align*}
since $$c_{w}(2\rho-s)\;=\;\overline{c_{w}(s)}=\frac{1}{c_{w}(s)}$$ on $\rho+i\mathfrak{a}^{*}$.

\begin{prop} Any $u\in L^{2}(\rho+i\mathfrak{a}^{*})$ satisfying $u(wt)=c_{w}(s)\cdot u(t)$ for all $w\in W$ is in the image.\end{prop}
\begin{proof}  First, for compactly-supported $u$ meeting this condition, we claim
$$\Psi_{u}\;=\;\frac{1}{|W|}\frac{1}{(2\pi i)^{\text{dim}\mathfrak{a}}}\int_{\rho+i\mathfrak{a}^{*}}u(t)\cdot E_{\rho+it}\;dt\neq 0$$
It suffices to show $c_{B}\Psi_{u}$ is not $0$.  With $s=\rho+it$, the relation implies $u(t)E_{2\rho-s}$ is invariant by $W$.  Let 
$$C\;=\;\{t\in\mathfrak{a}^{*}\;:\;\langle t,\alpha\rangle >0\;\text{for all simple}\;\alpha>0\}$$
be the positive Weyl chamber in $\mathfrak{a}^{*}$, where $\langle,\rangle$ is the Killing form transported to $\mathfrak{a}^{*}$ by duality.  Then
$$\Psi_{u}\;=\;\frac{1}{|W|}\frac{1}{(2\pi i)^{\text{dim}\mathfrak{a}}}\int_{\rho+i\mathfrak{a}^{*}}u(t)\cdot E_{s}\;dt\;=\;\frac{1}{(2\pi i)^{\text{dim}\mathfrak{a}}}\int_{\rho+iC}u(t)\cdot E_{s}\;dt$$
Since $u(tw)=u(t)\cdot c_{w}(\rho+it)$, the constant term of $\Psi_{u}$ is
$$c_{B}\Psi_{u}\;=\;\frac{1}{(2\pi i)^{\text{dim}\mathfrak{a}}}\int_{\rho+i\mathfrak{a}^{*}}u(t)\cdot a^{s}\;dt$$
This Fourier transform does not vanish for non-vanishing $u$.\end{proof}

Given $G=GL_{4}(\mathbb{R})$, $\Gamma=GL_{4}(\mathbb{Z})$, and $K=O_{4}(\mathbb{R})$, it is necessary to invoke the complete spectral decomposition of $L^{2}(\Gamma\backslash G/K)$, that cuspforms and cuspidal data Eisenstein series attached to non-minimal parabolic Eisenstein series attached to non-minimal parabolics, and their $L^{2}$ residues, as well as the minimal-parabolic pseudo-Eisenstein series, span $L^{2}(\Gamma\backslash G/K)$.  And we must demonstrate the orthogonality of integrals of minimal-parabolic Eisenstein series to all other spectral components.

We now decompose the pseudo-Eisenstein series with cuspidal data.  We carry this out for the $3,1$ pseudo-Eisenstein series, $2,2$ pseudo-Eisenstein series, and $2,1,1$ pseudo-Eisenstein series with cuspidal data.  This follows a similar pattern as the spectral decomposition.
Let $P=P^{3,1}$.  We decompose $P^{3,1}$ and $P^{1,3}$ pseudo-Eisenstein series with cuspidal support.  The data for a $P$ pseudo-Eisenstein series is smooth, compactly-supported, and left $Z_{\mathbb{A}}M_{k}^{P}N_{\mathbb{A}}^{P}$-invariant.  For now, we assume that the data is spherical, i.e. right $K$-invariant.  This means that the function is determined by its behavior on $Z_{\mathbb{A}}M_{k}^{P}\backslash M_{\mathbb{A}}^{P}$.  In contrast to the minimal parabolic case, this is not a product of copies of $GL_{1}$, so we can not simply use the $GL_{1}$ spectral theory (Mellin inversion) to accomplish the decomposition.  Instead, this quotient is isomorphic to $GL_{3}(k)\backslash GL_{3}(\mathbb{A})$, so we will use the spectral theory for $GL_{3}$.  If $\eta$ is the data for a $P^{3,1}$ pseudo-Eisenstein series $\Psi_{\eta}$, we can write $\eta$ as a tensor product $\eta\;=\;f\otimes\mu$ on
$$Z_{GL_{3}(\mathbb{A})}GL_{3}(k)\backslash GL_{3}(\mathbb{A})\cdot Z_{GL_{3}(k)}\backslash Z_{GL_{3}(\mathbb{A})}$$
Saying that the data is $\textit{cuspidal}$ means that $f$ is a cusp form.  Similarly, the data $\varphi=\varphi_{F,s}$ for a $P^{2,1}$-Eisenstein series is the tensor product of a $GL_{3}$ cusp form $F$ and a character $\chi_{s}=|.|^{s}$ on $GL_{1}$.  We show that $\Psi_{f,\eta}$ is the superposition of Eisenstein series $E_{F,s}$ where $F$ ranges over an orthonormal basis of cusp forms and $s$ is on the critical line.

\newpage
\begin{prop} The pseudo-Eisenstein series $\Psi_{f,\eta}$ admits a spectral decomposition
$$\Psi_{f,\eta}\;=\;\sum_{F}\int_{s} \langle \Psi_{f,\eta},E_{F,s}\rangle\cdot E_{F,s}\;ds$$
where the sum is over spherical cuspforms $F$ on $GL_{3}(k)\backslash GL_{3}(\mathbb{A})$.
\end{prop}
\begin{proof}Using the spectral expansions of $f$ and $\eta$,
$$\eta\;=\;f\otimes\eta\;=\;\big(\sum_{\text{cfms}\;F}\langle f,F\rangle\big)\cdot\big(\int_{s}\langle \mu,\chi_{s}\rangle\cdot\chi_{s}\;ds\big)\;=\;\sum_{\text{cfms}\;F}\int_{s}\langle \eta_{f,\mu},\varphi_{F,s}\rangle\cdot\varphi_{F,s}\;ds$$
So the pseudo-Eisenstein series can be re-expressed as a superposition of Eisenstein series
\begin{align*}\Psi_{f,\eta}(g)\;&=\;\sum_{\gamma\in P_{k}\backslash G_{k}}\eta_{f,\mu}(\gamma g)\\
&=\;\sum_{\gamma\in P_{k}\backslash G_{k}}\sum_{\text{cfms}\;F}\int_{s}\langle \eta_{f,\mu},\varphi_{F,s}\rangle\cdot\varphi_{F,s}(\gamma g)\;ds\\
&\;=\;\sum_{\text{cfms}\;F}\int_{s}\langle \eta_{f,\mu},\varphi_{F,s}\rangle\;\sum_{\gamma\in P_{k}\backslash G_{k}}\varphi_{F,s}(\gamma g)\;ds\\
&=\;\sum_{\text{cfms}\;F}\int_{s}\langle \eta_{f,\mu},\varphi_{F,s}\rangle\cdot E_{F,s}\;ds
\end{align*}
The coefficient $\langle \eta,\varphi\rangle_{GL_{3}}$ is the same as the pairing $\langle \Psi_{\eta},E_{\varphi}\rangle_{GL_{4}}$, since
$$\langle \Psi_{\eta},E_{\varphi}\rangle\;=\;\langle c_{P}(\Psi_{\eta}),\varphi\rangle\;=\;\langle \eta,\varphi\rangle$$
So the spectral decomposition is
$$\Psi_{f,\eta}\;=\;\sum_{\text{cfms}\;F}\int_{s}\langle \Psi_{f,\eta},E_{F,s}\rangle\cdot E_{F,s}\;ds$$
\end{proof}

It now remains to show that pseudo-Eisenstein series for the associate parabolic, $Q=P^{1,3}$ can also be decomposed into superpositions of $P$-Eisenstein series.  Notice that in the decomposition above, when we decomposed $P$-pseudo-Eisenstein series into genuine $P$-Eisenstein series, we did not use the functional equation to fold up the integral, as in the case of minimal parabolic pseudo-Eisenstein series.  For maximal parabolic Eisenstein series, the functional equation does not relate the Eisenstein series to itself, but rather the Eisenstein series of the associate parabolic.  We will use this functional equation to obtain the decomposition of associate parabolic pseudo-Eisenstein series.  The functional equation is
$$E_{F,s}^{Q}\;=\;b_{F,s}\cdot E_{F,1-s}^{P}$$
where $b_{F,s}$ is a meromorphic function that appears in the computation of the constant term along $P$ of the $Q$-Eisenstein series.
\begin{prop} The pseudo-Eisenstein series $\Psi_{f,\mu}^{Q}$ admits a spectral decomposition
$$\Psi_{f,\mu}^{Q}\;=\;\sum_{F}\int_{s} \langle \Psi_{f,\mu}^{Q},E_{F,1-s}^{P}\rangle\cdot |b_{F,1-s}|^{2}\cdot E_{F,1-s}^{P}$$
where $F$ ranges over an orthonormal basis of cuspforms.
\end{prop}

\begin{proof}We consider a $Q$-pseudo-Eisenstein series $\Psi_{f,\mu}^{Q}$ with cuspidal data.  By the same arguments used above to obtain the decomposition of $P$-pseudo-Eisenstein series, we can decompose $\Psi_{f,\mu}^{Q}$ into a superposition of $Q$-Eisenstein series,
$$\Psi_{f,\mu}^{Q}(g)\;=\;\sum_{\text{cfms}\;F}\int_{s}\langle \eta_{f,\mu},\varphi_{F,s}\rangle\cdot E_{F,s}^{Q}(g)$$
Now using the functional equation,
\begin{align*}&\Psi_{f,\mu}^{Q}(g)\;=\;\sum_{\text{cfms}\;F}\int_{s}\langle \Psi_{f,\mu}^{Q},b_{F,s}\cdot E_{F,1-s}^{P}\rangle\cdot b_{F,s}\cdot E_{F,1-s}^{P}\\&\;=\;\sum_{\text{cfms}\;F}\int_{s}\langle \Psi_{f,\mu}^{Q},E_{F,1-s}^{P}\rangle\cdot |b_{F,s}|^{2}\cdot E_{F,1-s}^{P}\\
\end{align*}
giving the proposition.\end{proof}
So we have a decomposition of $Q$-pseudo-Eisenstein series (with cuspidal data) into a $P$-Eisenstein series (with cuspidal data).  In order to use the functional equation we did have to move some contours, but in this case there are no poles, so we did not pick up any residues.
Likewise, if $\eta$ is the data for a $P^{2,1,1}$ pseudo-Eisenstein series $\Psi_{\eta}$, we can write $\eta$ as a tensor product 
$\eta\;=\;f\otimes \mu_{1}\otimes\mu_{2}$
on $$Z_{GL_{4}(\mathbb{A})}\backslash Z_{GL_{2}(\mathbb{A})}\times Z_{GL_{1}(\mathbb{A})}\times Z_{GL_{1}(\mathbb{A})}$$  Similarly, the data $\varphi=\varphi_{F,s_{1},s_{2}}$ for a $P^{2,1,1}$-Eisenstein series is the tensor product of a $GL_{2}$ cuspform and characters $\chi_{s_{1}}$ and $\chi_{s_{2}}$ on $GL_{1}$.  We show that $\Psi_{f,\mu}$ is the superposition of Eisenstein series $E_{F,s_{1},s_{2}}$ where $F$ ranges over an orthonormal basis of cusp forms and $s_{1}$ and $s_{2}$ are on the vertical line.
\begin{prop} The $2,1,1$ pseudo-Eisenstein series $\Psi_{f,\mu_{1},\mu_{2}}$ admits a spectral expansion
$$\Psi_{f,\mu_{1},\mu_{2}}\;=\;\sum_{F}\int_{s_{1}}\int_{s_{2}}\langle \eta_{f,\mu_{1},\mu_{2}},\varphi_{F,s_{1},s_{2}}\rangle\cdot E_{F,s_{1},s_{2}}$$
where $F$ ranges over an orthonormal basis of cuspforms.
\end{prop}
\begin{proof}Using the spectral expansions of $f$ and $\mu$,
$$\eta\;=\;f\otimes\mu_{1}\otimes\mu_{2}\;=\;\big(\sum_{\text{cfms}\;F}\langle f,F\rangle\cdot F\big)\cdot\big(\int_{s_{1}}\langle \mu_{1},\chi_{s_{1}}\rangle\cdot\chi_{s_{1}}\;ds_{1}\big)\cdot\big(\int_{s_{2}}\langle \mu_{2},\chi_{s_{2}}\rangle\cdot\chi_{s_{2}}\;ds_{2}\big)$$
$$=\;\sum_{\text{cfms}\;F}\int_{s_{1}}\int_{s_{2}}\langle \eta_{f,\mu_{1},\mu_{2}},\varphi_{F,s_{1},s_{2}}\rangle\cdot\varphi_{F,s_{1},s_{2}}\;ds_{1}\;ds_{2}$$
Therefore, the pseudo-Eisenstein series can be re-expressed as a (double) superposition of Eisenstein series.
\begin{align*} \Psi_{f,\mu_{1},\mu_{2}}\;&=\;\sum_{\gamma\in P_{k}\backslash G_{k}}\eta_{f,\mu_{1},\mu_{2}}(\gamma g)\\
&=\;\sum_{\gamma\in P_{k}\backslash G_{k}}\sum_{\text{cfms}\; F}\int_{s_{1}}\int_{s_{2}}\langle \eta_{f,\mu_{1},\mu_{2}},\varphi_{F,s_{1},s_{2}}\rangle\cdot\varphi_{F,s_{1},s_{2}}(\gamma g)\;ds_{1}\;ds_{2}\\
&=\;\sum_{\text{cfms}\;F}\int_{s_{1}}\int_{s_{2}}\langle \eta_{f,\mu_{1},\mu_{2}},\varphi_{F,s_{1},s_{2}}\rangle\sum_{\gamma\in P_{k}\backslash G_{k}}\varphi_{F,s_{1},s_{2}}(\gamma g)\;ds_{1}\;ds_{2}\\
&=\;\sum_{\text{cfms}\;F}\int_{s_{1}}\int_{s_{2}}\langle \eta_{f,\mu_{1},\mu_{2}},\varphi_{F,s_{1},s_{2}}\rangle\cdot E_{F,s_{1},s_{2}}(g)\end{align*}
\end{proof}
Finally, if $\eta$ is the data for a $P^{2,2}$ pseudo-Eisenstein series $\Psi_{\eta}$, we can write $$\eta_{f,g,\mu}=f\otimes g\otimes\mu$$
on $$Z_{GL_{4}}(\mathbb{A})/ Z_{GL_{2}}(\mathbb{A})\times Z_{GL_{2}}(\mathbb{A})$$ where $f$ and $g$ are cuspforms, and $\mu$ is a compactly-supported smooth function on $GL(1)$.    Similarly, the data $\varphi=\varphi_{f_{1},f_{2},s}$ for a $P^{2,2}$-Eisenstein series is the tensor product of $GL(2)$ cuspforms $f_{1}$ and $f_{2}$ and a character $\chi_{s}$.
\begin{prop} The $2,2$ pseudo-Eisenstein series $\Psi_{\eta}$ has a spectral expansion in terms of $2,2$ Eisenstein series
$$\Psi_{\eta}\;=\;\sum_{F_{1},F_{2}}\int_{s}\langle \eta_{f,g,\mu},\varphi_{F_{1},F_{2},s}\rangle E_{F_{1},F_{2},s}\;ds$$
where $F_{1}$ and $F_{2}$ are cuspforms on $GL(2)$.
\end{prop}
\begin{proof}Writing
$$\eta\;=\;f\otimes g\otimes\mu\;=\;\big(\sum_{\text{cfms}\;F}\langle f,F\rangle\cdot F\big)\big(\sum_{\text{cfms}\;F}\langle g,F\rangle\cdot F\big)\cdot\big(\int_{s}\langle\mu,\chi_{s}\rangle\cdot\chi_{s}\big)$$
$$=\;\sum_{\text{cfms}F_{1},F_{2}}\int_{s}\langle \eta_{f,g,\mu},\varphi_{F_{1},F_{2},s}\rangle\cdot\varphi_{F_{1},F_{2},s}\;ds$$
As before, the corresponding pseudo-Eisenstein series will unwind
$$\Psi_{\eta}\;=\;\sum_{\gamma\in P_{k}\backslash G_{k}}\eta_{f,g,\mu}(\gamma g)\;=\;\sum_{\text{cfms}F_{1},F_{2}}\int_{s}\langle \eta_{f,g,\mu},\varphi_{F_{1},F_{2},s}\rangle\cdot E_{F_{1},F_{2},s}\;ds$$
\end{proof}
Recall the construction of $2,2$ pseudo-Eisenstein series. Let $\phi\in C_{c}^{\infty}(\mathbb{R})$ and let $f$ be a spherical cuspform on $GL_{2}$ with trivial central character.  
Let $$\varphi( \left( \begin{array}{cc}
A & B\\
0& D\end{array} \right))=\phi(\Big|\frac{\text{det}A}{\text{det}D}\Big|^{2})\cdot f(A)\cdot \overline{f}(D)$$
extending by right $K$-invariance to be made spherical.
Define the $P^{2,2}$ pseudo-Eisenstein series by
$$\Psi_{\varphi}(g)\;=\;\sum_{\gamma\in P_{k}\backslash G_{k}}\varphi(\gamma g)$$
We recall the construction of $2,1,1$ pseudo-Eisenstein series.  Let $f$ be a spherical cuspform on $GL_{2}(k)\backslash GL_{2}(\mathbb{A})$, and let $\phi_{1}, \phi_{2}\in C_{c}^{\infty}(\mathbb{R})$. Let
$$\varphi_{f,\phi_{1},\phi_{2}}(\left( \begin{array}{ccc}
A & 0&0 \\
0& b&0\\
0&0&c\end{array} \right))\;=\;f(A)\cdot\phi_{1}(\frac{\text{det}A}{b^{2}})\cdot \phi_{2}(\frac{\text{det}A}{c^{2}})$$
The $2,1,1$ pseudo-Eisenstein series with this data is
$$\Psi_{\varphi}\;=\;\sum_{\gamma\in P_{k}\backslash G_{k}}\varphi_{f,\phi_{1},\phi_{2}}(\gamma g)$$

\begin{prop} The pseudo-Eisenstein series $\Psi_{\varphi}^{2,2}$ is orthogonal to all other pseudo-Eisenstein series in $\text{Sob}(+1)$.
\end{prop}
\begin{proof}Recall by [MW p.100] that 
$$\langle \Psi_{\varphi}^{2,2},\Psi_{\psi}^{2,1,1}\rangle_{L^{2}}\;=\;0$$

Let us now check that they're also orthogonal in the $+1$-Sobolev space.  Note that
$$\langle \Psi_{\varphi}^{2,2},\Psi_{\psi}^{2,1,1}\rangle_{+1}\;=\;\langle \Psi_{\varphi}^{2,2},\Psi_{\psi}^{2,1,1}\rangle_{L^{2}}+\langle \Delta\Psi_{\varphi}^{2,2},\Psi_{\psi}^{2,1,1}\rangle_{L^{2}}$$
Since the first summand is zero, it suffices to prove that the second is zero. To this end, we rewrite the Casimir operator
$$\Omega\;=\;\Omega_{1}+\Omega_{2}+\Omega_{3}+\Omega_{4}$$
where $$\Omega_{1}\;=\;\frac{1}{2}H_{1,2}^{2}+E_{1,2}E_{2,1}+E_{2,1}E_{1,2}$$
and
$$\Omega_{2}\;=\;\frac{1}{2}H_{3,4}^{2}+E_{3,4}E_{4,3}+E_{4,3}E_{3,4}$$
while
$$\Omega_{3}\;=\;\frac{1}{4}H_{1,2,3,4}^{2}$$
We let $\Omega_{4}$ be the remaining terms appearing in the expression of Casimir.  We prove that application of $\Omega$ to $\Psi_{\varphi}$ produces another function in the span of $2,2$ pseudo-Eisenstein series.  Being in the span of $2,2$ pseudo-Eisenstein series renders $\Omega \Psi_{\varphi}$ orthogonal to all other non-associate pseudo-Eisenstein series. We will prove that when restricted to $G/K$, $\Omega_{1}$ acts as the $SL_{2}$-Laplacian on the cuspform $f$, $\Omega_{2}$ acts as the $SL_{2}$-Laplacian on $\overline{f}$, while $\Omega_{3}$ acts as a second derivative on the test function.
Indeed, let
$$\Omega_{1}\;=\;\frac{1}{2}H_{1,2}^{2}+E_{1,2}E_{2,1}+E_{2,1}E_{1,2}$$
where $H_{1,2}=\text{diag}(1,-1,0,0)$ and $E_{i,j}$ is the matrix with $1$ in the $ij^{th}$ position and $0$'s elsewhere. We check how $H_{1,2}$ acts on smooth functions on $\varphi$.  Let
$$A\;=\;\left( \begin{array}{cc}
a& b \\
c& d\end{array} \right)\;\;\;D\;=\;\left( \begin{array}{cccc}
f & g\\
h&i\end{array} \right)$$
Observe that
$$H_{1,2}\cdot\varphi( \left( \begin{array}{cc}
A & * \\
0& D\end{array} \right))\;=\;\frac{d}{dt}\bigg|_{t=0}\varphi( \left( \begin{array}{cccc}
a & b & 0 & 0\\
c& d&0&0\\
0&0&f&g\\
0&0&h&i\end{array} \right)\cdot\left( \begin{array}{cccc}
e^{t} & 0 & 0 & 0\\
0& e^{-t}&0&0\\
0&0&1&0\\
0&0&0&1\end{array} \right))$$
This is
$$\frac{d}{dt}\bigg|_{t=0}\;\varphi( \left( \begin{array}{cccc}
a e^{t}& be^{t} & 0 & 0\\
ce^{t}& de^{t}&0&0\\
0&0&f&g\\
0&0&h&i\end{array} \right))\;=\;\frac{d}{dt}\bigg|_{t=0}\phi(\Big|\frac{\text{det}A}{\text{det}D}\Big|^{2})\cdot f(\left( \begin{array}{cc}
ae^{t} & be^{-t}\\
ce^{t}& de^{-t}\end{array} \right))\cdot f(D)$$
Use Iwasawa coordinates on the upper left hand $GL(2)$ block of the Levi component, namely
$$n_{x_{1}}\;=\;\left( \begin{array}{cccc}
1& x_{1} & 0 & 0\\
0& 1&0&0\\
0&0&1&0\\
0&0&0&1\end{array} \right)\;\;\;m_{y_{1}}\;=\;\left( \begin{array}{cccc}
\sqrt{y_{1}}& 0 & 0 & 0\\
0& \frac{1}{\sqrt{y_{1}}}&0&0\\
0&0&1&0\\
0&0&0&1\end{array} \right)$$
As in the discussion for $SL_{2}(\mathbb{R})$, 
$$(H_{1,2}f)(n_{x_{1}}m_{y_{1}})\;=\;2y_{1}\frac{\partial}{\partial y_{1}}f(n_{x_{1}}m_{y_{1}})$$
Therefore, letting $\Delta_{1}$ be $\Omega_{1}$ restricted to $G/K$, we see that the effect of $\Delta_{1}$ on the cuspform $f$ is just
$$\Delta_{1}(f)\;=\;y_{1}^{2}(\frac{\partial^{2}}{\partial x_{1}^{2}}+\frac{\partial^{2}}{\partial y_{1}^{2}})f\;=\;\lambda_{f}\cdot f$$
Therefore, 
$$\Delta_{1}(\varphi_{\phi,f,\overline{f}})\;=\;\varphi_{\phi,\lambda_{f} f,\overline{f}}\;=\;\lambda_{f}\cdot\varphi_{\phi,f,\overline{f}}$$
A similar argument which uses $H_{3,4}$, $E_{3,4}$ and $E_{4,3}$ as the standard basis in the lower right $2\times 2$ block, shows that, for $\Delta_{2}$ the restriction of $\Omega_{2}$ to smooth functions on $G/K$,
$$\Delta_{2}(\varphi_{\phi,f,\overline{f}})\;=\;\varphi_{\phi,f,\overline{f}}\;=\;\lambda_{\overline{f}}\varphi_{\phi,\lambda_{f}f,\overline{f}}$$
It remains to check the effect of $\Omega_{3}\;=\;\frac{1}{4}H_{1,2,3,4}^{2}$. Observe that
$$H_{1,2,3,4}\varphi( \left( \begin{array}{cccc}
a &b  & 0 & 0\\
c& d&0&0\\
0&0&f&g\\
0&0&h&i\end{array} \right))\;=\;\frac{d}{dt}\bigg|_{t=0}\varphi(\left( \begin{array}{cccc}
a &b  & 0 & 0\\
c& d&0&0\\
0&0&f&g\\
0&0&h&i\end{array} \right)\cdot \left( \begin{array}{cccc}
e^{t} &0  & 0 & 0\\
0& e^{t}&0&0\\
0&0&e^{-t}&0\\
0&0&0&e^{-t}\end{array} \right))$$
Yet this is just
$$=\;\frac{d}{dt}\bigg|_{t=0}\varphi(\left( \begin{array}{cccc}
ae^{t} &be^{t}  & 0 & 0\\
ce^{t}& de^{t}&0&0\\
0&0&fe^{-t}&ge^{-t}\\
0&0&he^{-t}&ie^{-t}\end{array} \right))$$
Which gives
$$=\;\frac{d}{dt}\bigg|_{t=0}\phi(\frac{e^{t}\text{det}A}{e^{-t}\text{det}D})\cdot f( \left( \begin{array}{cc}
ae^{t} &be^{t} \\
ce^{t}& de^{t}\end{array} \right))\cdot\overline{f}( \left( \begin{array}{cc}
fe^{-t} & ge^{-t} \\
he^{-t}& ie^{-t}\end{array} \right))\;=\;2\cdot\phi'\cdot f(A)\cdot \overline{f}(D)$$
since both $f$ and $\overline{f}$ have trivial central character.  Therefore, the effect of $\frac{1}{4}H_{1,2,3,4}$ as a differential operator on $\varphi_{\phi,f,\overline{f}}$ is
$$\frac{1}{4}H_{1,2,3,4}\cdot\varphi_{\phi,f,\overline{f}}\;=\;\varphi_{\phi'',f,\overline{f}}$$
That is,
$$\Delta_{3}\varphi_{\phi,f,\overline{f}}\;=\;\varphi_{\phi'',f,\overline{f}}$$
Together the effect of the three differential operators is
$$(\Delta_{1}+\Delta_{2}+\Delta_{3})\varphi_{\phi,f,\overline{f}}\;=\;\varphi_{(\lambda_{f}+\lambda_{\overline{f}})\phi+\phi'',f\overline{f}}$$
Therefore,
$$(\Delta_{1}+\Delta_{2}+\Delta_{3})(\Psi_{\varphi_{\phi,f,\overline{f}}})\;=\;\Psi_{\varphi_{(\lambda_{f}+\lambda_{\overline{f}})\phi+\phi'',f,\overline{f}}}$$
The operator $\Delta_{4}$ acts by $0$ on the vector $\varphi_{\phi,f,\overline{f}}$. Therefore,
$$\Delta\Psi_{\varphi_{\phi,f,\overline{f}}}\;=\;\Psi_{\varphi_{(\lambda_{f}+\lambda_{\overline{f}})\phi+\phi'',f,\overline{f}}}$$
The function
$$\Psi_{\varphi_{(\lambda_{f}+\lambda_{\overline{f}})\phi+\phi'',f,\overline{f}}}$$
is another $2,2$ pseudo-Eisenstein series because $(\lambda_{f}+\lambda_{\overline{f}})\phi+\phi''$ is another function in $C_{c}^{\infty}(\mathbb{R})$, so [MW, p.100] applies again to give
$$\langle \Psi_{\varphi_{(\lambda_{f}+\lambda_{\overline{f}})\phi+\phi'',f,\overline{f}}},\Psi_{\psi}^{2,1,1}\rangle_{L^{2}}\;=\;0$$
Therefore,
$$\langle\Delta \Psi_{\varphi,f,\overline{f}},\Psi_{\psi}^{2,1,1}\rangle_{L^{2}}\;=\;0$$
proving that the pseudo-Eisenstein series are orthogonal in the $+1$-index Sobolev space.  An inductive argument shows that they are orthogonal in every Sobolev space.

An analogous argument shows that $2,2$ pseudo-Eisenstein series are orthogonal to $3,1$ pseudo-Eisenstein series, as well as $1,1,1$ pseudo-Eisenstein series.
\end{proof}
We turn our attention to the $3,1$-Eisenstein series. 
\begin{prop} $3,1$ pseudo-Eisenstein series are orthogonal to all other (non-associate) pseudo-Eisenstein series in $\text{Sob}(+1)$.
\end{prop}
\begin{proof} We review the construction of $3,1$ pseudo-Eisenstein series with cuspidal and test function data.  Let $f_{1}$ be a spherical cuspform on $GL_{3}(k)\backslash GL_{3}(\mathbb{A})$ and $\phi\in C_{c}^{\infty}(\mathbb{R})$.  Consider the vector
$$\varphi_{f,\phi}(\left( \begin{array}{cc}
A & * \\
0& d\end{array} \right))\;=\;f(A)\cdot\phi(\frac{\text{det}A}{d^{3}})$$
Working in $GL_{4}$ consider the element
$$H_{1}\;=\;\left( \begin{array}{cccc}
1 & 0&0&0 \\
0& 0&0&0\\
0&0&0&0\\
0&0&0&0\end{array} \right)\in\mathfrak{gl}_{4}(\mathbb{R})$$
We determine the effect of $H_{1}$ as a differential operator on $\varphi_{f,\phi}$. To this end, let
$$n_{x_{1}x_{2}x_{3}}\;=\;\left( \begin{array}{cccc}
1 & x_{1}&x_{2}&0 \\
0& 1&x_{3}&0\\
0&0&1&0\\
0&0&0&1\end{array} \right)\;\;\;\;m_{y_{1}y_{2}y_{3}y_{4}}\;=\;\left( \begin{array}{cccc}
y_{1} & 0&0&0 \\
0& y_{2}&0&0\\
0&0&y_{3}&0\\
0&0&0&y_{4}\end{array} \right)$$
Then $$H_{1}\cdot \varphi_{f,\phi}(n_{x_{1}x_{2}x_{3}}m_{y_{1}y_{2}y_{3}y_{4}})\;=\;\frac{d}{dt}\bigg|_{t=0}\varphi_{f,\phi}(n_{x_{1}x_{2}x_{3}}m_{y_{1}y_{2}y_{3}y_{4}}\left( \begin{array}{cccc}
e^{t} & 0&0&0 \\
0& 1&0&0\\
0&0&1&0\\
0&0&0&1\end{array} \right))$$
This is 
$$\frac{d}{dt}\bigg|_{t=0}\varphi_{f,\phi}(n_{x_{1}x_{2}x_{3}}m_{y_{1}e^{t}y_{2}y_{3}y_{4}})\;=\;y_{1}\frac{\partial}{\partial y_{1}}\varphi_{f,\phi}(n_{x_{1}x_{2}x_{3}}m_{y_{1}y_{2}y_{3}y_{4}})$$
Therefore,
$$H_{1}\cdot \varphi_{f,\phi}(n_{x_{1}x_{2}x_{3}}m_{y_{1}y_{2}y_{3}y_{4}})\;=\;y_{1}\frac{\partial}{\partial y_{1}}\varphi_{f,\phi}(n_{x_{1}x_{2}x_{3}}m_{y_{1} y_{2}y_{3}y_{4}})$$
The effect of $H_{2}$ and $H_{3}$ is computed similarly.  That is
$$H_{2}\cdot \varphi_{f,\phi}(n_{x_{1}x_{2}x_{3}}m_{y_{1}y_{2}y_{3}y_{4}})\;=\;y_{2}\frac{\partial}{\partial y_{2}}\varphi_{f,\phi}(n_{x_{1}x_{2}x_{3}}m_{y_{1} y_{2}y_{3}y_{4}})$$
while
$$H_{3}\cdot \varphi_{f,\phi}(n_{x_{1}x_{2}x_{3}}m_{y_{1}y_{2}y_{3}y_{4}})\;=\;y_{3}\frac{\partial}{\partial y_{3}}\varphi_{f,\phi}(n_{x_{1}x_{2}x_{3}}m_{y_{1} y_{2}y_{3}y_{4}})$$
With notation as before, we determine the effect of $E_{1,2}$ as a differential operator.  Observe that
$$E_{1,2}\cdot \varphi_{f,\phi}(n_{x_{1}x_{2}x_{3}}m_{y_{1}y_{2}y_{3}y_{4}})\;=\;\frac{d}{dt}\bigg|_{t=0}\varphi_{f,\phi}(n_{x_{1}x_{2}x_{3}}m_{y_{1}y_{2}y_{3}y_{4}}\left( \begin{array}{cccc}
1 & t&0&0 \\
0& 1&0&0\\
0&0&1&0\\
0&0&0&1\end{array} \right))$$
This is just
$$\frac{d}{dt}\bigg|_{t=0}\varphi_{f,\phi}(n_{x_{1}+y_{1}t x_{2}x_{3}}m_{y_{1}y_{2}y_{3}y_{4}})\;=\;y_{1}\frac{\partial}{\partial x_{1}}\varphi_{f,\phi}(n_{x_{1}x_{2}x_{3}}m_{y_{1}y_{2}y_{3}y_{4}})$$
Therefore, the effect of $E_{1,2}$ is $y_{1}\frac{\partial}{\partial x_{1}}$, and $E_{1}$ differentiates only the cuspform $f$.  Similar arguments show that the effect of $E_{1,3}$ as a differential operator is
$$E_{1,3}\rightarrow y_{2}\frac{\partial}{\partial x_{2}}$$
and
$$E_{2,3}\rightarrow y_{3}\frac{\partial}{\partial x_{3}}$$
Observe that $E_{1,4}$, $E_{2,4}$, and $E_{3,4}$ act by $0$ on $\varphi_{f,\phi}$. We prove this for $E_{1,4}$, the argument being identical for $E_{2,4}$ and $E_{3,4}$. Note
$$E_{1,4}\cdot \varphi_{f,\phi}(n_{x_{1}x_{2}x_{3}}m_{y_{1}y_{2}y_{3}y_{4}})\;=\;\frac{d}{dt}\bigg|_{t=0}\varphi_{f,\phi}(n_{x_{1}x_{2}x_{3}}m_{y_{1}y_{2}y_{3}y_{4}}\left( \begin{array}{cccc}
1 & 0&0&t \\
0& 1&0&0\\
0&0&1&0\\
0&0&0&1\end{array} \right))$$
This is
$$\frac{d}{dt}\bigg|_{t=0}\varphi_{f,\phi}(\left( \begin{array}{cccc}
1 & x_{1}&x_{2}&* \\
0& 1&x_{3}&0\\
0&0&1&0\\
0&0&0&1\end{array} \right)\cdot\left( \begin{array}{cccc}
y_{1} & 0&0&0 \\
0& y_{2}&0&0\\
0&0&y_{3}&0\\
0&0&0&y_{4}\end{array} \right))\;=\;0$$
Let 
$$H_{4}\;=\;\left( \begin{array}{cccc}
0 & 0&0&0 \\
0& 0&0&0\\
0&0&0&0\\
0&0&0&1\end{array} \right)$$
Then
$$H_{4}\cdot \varphi_{f,\phi}(\left( \begin{array}{cccc}
1 & x_{1}&x_{2}&0 \\
0& 1&x_{3}&0\\
0&0&1&0\\
0&0&0&1\end{array} \right)\cdot\left( \begin{array}{cccc}
y_{1} & 0&0&0 \\
0& y_{2}&0&0\\
0&0&y_{3}&0\\
0&0&0&y_{4}\end{array} \right))$$
Which is $$\frac{d}{dt}\bigg|_{t=0}\varphi_{f,\phi}(\left( \begin{array}{cccc}
1 & x_{1}&x_{2}&0 \\
0& 1&x_{3}&0\\
0&0&1&0\\
0&0&0&1\end{array} \right)\cdot\left( \begin{array}{cccc}
y_{1} & 0&0&0 \\
0& y_{2}&0&0\\
0&0&y_{3}&0\\
0&0&0&y_{4}t\end{array} \right))\;=\;\varphi_{f,\phi'}$$
It is clear to see that this is the only element of the Lie algebra differentiating the test function datum. If $\{X_{i}\}$ is a basis of $\mathfrak{gl}_{4}(\mathbb{R})$ and $\{X_{i}\}$ is the dual basis relative to the trace pairing, define an element $\Omega\in U_{\mathfrak{g}}$ by
$$\Omega\;=\;\sum_{i}X_{i}X_{i}^{*}$$
Let $\Omega_{1}$ be the element of $Z_{\mathfrak{gl_{3}}}$ given by 
$$\Omega_{1}\;=\;\frac{1}{2}H_{1}^{2}+\frac{1}{2}H_{2}^{2}+\frac{1}{2}H_{3}^{2}+E_{1,2}E_{2,1}+E_{1,3}E_{3,1}+E_{2,3}+E_{3,2}$$
As shown above, this element differentiates the cuspidal-data, and does not interact with the test function datum. Since $\Omega_{1}\in Z_{\mathfrak{gl_3}}$, it acts by a scalar $\lambda_{f}$ on the irreducible unramified principal series generated by $f$.  Then,
$$\Omega\;=\;\Omega_{1}+H_{4}+\Omega_{2}$$
where $\Omega_{2}=\Omega-\Omega_{1}-H_{4}$.  Since $\Omega_{2}$ interacts with neither the cuspidal data nor the test function data, its effect as a differential operator on $\varphi_{f,\phi}$ will be $0$.  Note that $\Omega_{1}\cdot \varphi_{f,\phi}=\varphi_{\lambda_{f}f,\phi}$, while $H_{4}\cdot \varphi_{f,\phi}=\varphi_{f,\phi'}$.  Therefore,
$$\Omega \varphi_{f,\phi}\;=\;\varphi_{f,(\lambda_{f}\phi+\phi')}$$
producing another $3,1$ pseudo-Eisenstein series, which is orthogonal to the $2,1,1$ pseudo-Eisenstein series, $1,1,1,1$ pseudo-Eisenstein series, and $2,2$ pseudo-Eisenstein series, by [MW,p.100].
\end{proof}

Finally, we consider $2,1,1$ pseudo-Eisenstein series.  Let $X_{1},X_{2},\dots,X_{n}$ is a basis for $\mathfrak{gl}_{4}(\mathbb{R})$, with dual basis $X_{1}^{*},X_{2}^{*},\dots,X_{n}^{*}$ relative to the trace pairing.  Let $\Omega=\sum_{i} X_{i}\cdot X_{i}^{*}\in Z\mathfrak{g}$, and let $\Delta$ be $\Omega$ descended to $G/K$.  We will show that application of $\Delta$ to a $2,1,1$ pseudo-Eisenstein series made with cuspidal data $f$ and test functions $\phi_{1},\phi_{2}$ produces another $2,1,1$ pseudo-Eisenstein series.  This will prove that $2,1,1$ pseudo-Eisenstein series are orthogonal to all other (non-associate) pseudo-Eisenstein series by [MW, p.100].  We recall the construction of $2,1,1$ pseudo-Eisenstein series.  Let $f$ be a spherical cuspform on $GL_{2}(k)\backslash GL_{2}(\mathbb{A})$, and let $\phi_{1}, \phi_{2}\in C_{c}^{\infty}(\mathbb{R})$. Let
$$\varphi_{f,\phi_{1},\phi_{2}}(\left( \begin{array}{ccc}
A & 0&0 \\
0& b&0\\
0&0&c\end{array} \right))\;=\;f(A)\cdot\phi_{1}(\frac{\text{det}A}{b^{2}})\cdot \phi_{2}(\frac{\text{det}A}{c^{2}})$$
The $2,1,1$ pseudo-Eisenstein series with this data is
$$\Psi_{\varphi}\;=\;\sum_{\gamma\in P_{k}\backslash G_{k}}\varphi_{f,\phi_{1},\phi_{2}}(\gamma g)$$
\begin{prop} The $2,1,1$ pseudo-Eisenstein series $\Psi_{\varphi}$ is orthogonal to all other (non-associate) pseudo-Eisenstein series in $\text{Sob}(+1)$.
\end{prop}
\begin{proof}We consider basis elements of the Lie algebra $\mathfrak{gl}_{4}(\mathbb{R})$.  Let $E_{ij}$ be as before.  Let $H_{i}$ be the matrix with $1$ on the $i^{\text{th}}$ diagonal entry and $0$'s elsewhere.  We consider the effect of the $H_{i}$'s as differential operators on $\varphi_{f,\phi_{1},\phi_{2}}$.  It will be convenient to use an Iwasawa decomposition on the $GL_{2}$ block in the upper left hand corner.  We will be considering right $K$-invariant functions, so $\varphi$ is determined by its effect on $n_{x}m_{y_{1}y_{2}}$ where 
$$n_{x}\;=\;\left( \begin{array}{cccc}
1 & x&0&0 \\
0& 1&0&0\\
0&0&1&0\\
0&0&0&1\end{array} \right)\;\;\;\;\text{and}\;\;\;m_{y_{1}y_{2}}\;=\;\left( \begin{array}{cccc}
y_{1} & 0&0&0 \\
0& y_{2}&0&0\\
0&0&1&0\\
0&0&0&1\end{array} \right)$$
We calculate $H_{1}$'s effect on $\varphi_{f,\phi_{1},\phi_{2}}(n_{x}m_{y_{1} y_{2}})$.  Note that 
$$H_{1}\cdot \varphi(n_{x}m_{y_{1} y_{2}})\;=\;\frac{d}{dt}\bigg|_{t=0}\varphi(n_{x} m_{y_{1}e^{t} y_{2}})\;=\;y_{1}\frac{\partial}{\partial y_{1}}\varphi(n_{x}m_{y_{1} y_{2}})$$
Similarly,
$$H_{2}\cdot \varphi(n_{x}m_{y_{1}y_{2}})\;=\;y_{2}\frac{\partial}{\partial y_{2}}\varphi(n_{x}m_{y_{1}y_{2}})$$
Therefore, $H_{1}$ and $H_{2}$ differentiate the cuspform $f$, and leave the functions $\phi_{1}$ and $\phi_{2}$ as they are.  As before,
$$E_{1,2}\cdot \varphi(n_{x}m_{y_{1}y_{2}})\;=\;y_{1}\frac{\partial}{\partial x}\varphi(n_{x}m_{y_{1} y_{2}})$$
Let us consider the effect of $H_{3}$ as a differential operator on $\varphi$.  Observe that
$$H_{3}\cdot \varphi_{f,\phi_{1},\phi_{2}}(n_{x}m_{y_{1} y_{2} y_{3} y_{4}})\;=\;\frac{d}{dt}\bigg|_{t=0}\varphi_{f,\phi_{1},\phi_{2}}(n_{x}m_{y_{1} y_{2} y_{3} e^{t} y_{4}})$$
This is
$$\frac{d}{dt}\bigg|_{t=0} f(A)\phi_{1}(\frac{\text{det}A}{y_{3}^{2}}e^{-2t})\phi_{2}(\frac{\text{det}A}{y_{4}^{2}})\;=\;-2 f(A)\cdot \phi_{1}'(\frac{\text{det}A}{y_{3}^{2}})\phi_{2}(\frac{\text{det}A}{y_{4}^{2}})$$
Therefore,
$$H_{3}\cdot \varphi_{f,\phi_{1},\phi_{2}}(n_{x}m_{y_{1} y_{2} y_{3} y_{4}})\;=\;\varphi_{f,-2\phi_{1}',\phi_{2}}$$
Similarly,
$$H_{4}\cdot \varphi_{f,\phi_{1},\phi_{2}}(n_{x}m_{y_{1} y_{2} y_{3} y_{4}})\;=\;\varphi_{f,\phi_{1},-2\phi_{2}'}$$
Observe that $E_{1,3}$ acts as $0$ on $\varphi_{f,\phi_{1},\phi_{2}}$.  Indeed,
$$E_{1,3}\cdot\varphi_{f,\phi_{1},\phi_{2}}\;=\;\frac{d}{dt}\bigg|_{t=0}\varphi_{f,\phi_{1},\phi_{2}}(\left( \begin{array}{cccc}
z_{1} &z_{2} &0&0 \\
z_{3}& z_{4}&0&0\\
0&0&b&0\\
0&0&0&c\end{array} \right)\left( \begin{array}{cccc}
1 & 0&t&0 \\
0& 1&0&0\\
0&0&1&0\\
0&0&0&1\end{array} \right))$$
This is just
$$\frac{d}{dt}\bigg|_{t=0}\varphi_{f,\phi_{1},\phi_{2}}(\left( \begin{array}{cccc}
z_{1} &z_{2} &0&0 \\
z_{3}& z_{4}&0&0\\
0&0&b&0\\
0&0&0&c\end{array} \right))\;=\;0$$
The effect of $E_{1,4}$ is computed similarly.  Observe
$$E_{1,4}\cdot \varphi_{f,\phi_{1},\phi_{2}}\;=\;\frac{d}{dt}\bigg|_{t=0}\varphi_{f,\phi_{1},\phi_{2}}(\left( \begin{array}{cccc}
z_{1} &z_{2} &0&0 \\
z_{3}& z_{4}&0&0\\
0&0&b&0\\
0&0&0&c\end{array} \right)\left( \begin{array}{cccc}
1 & 0&0&t \\
0& 1&0&0\\
0&0&1&0\\
0&0&0&1\end{array} \right))$$
Which is
$$\frac{d}{dt}\bigg|_{t=0}\varphi_{f,\phi_{1},\phi_{2}}(\left( \begin{array}{cccc}
z_{1} &z_{2} &0&0 \\
z_{3}& z_{4}&0&0\\
0&0&b&0\\
0&0&0&c\end{array} \right))\;=\;0$$
The elements $E_{3,1}$, $E_{3,2}$, $E_{4,1}$ and $E_{4,2}$ also act as $0$.  To see that $E_{3,4}$ acts by $0$, note
$$E_{3,4}\cdot \varphi_{f,\phi_{1},\phi_{2}}\;=\;\frac{d}{dt}\bigg|_{t=0}\varphi_{f,\phi_{1},\phi_{2}}(\left( \begin{array}{cccc}
z_{1} &z_{2} &0&0 \\
z_{3}& z_{4}&0&0\\
0&0&b&0\\
0&0&0&c\end{array} \right)\left( \begin{array}{cccc}
1 & 0&0&0 \\
0& 1&0&0\\
0&0&1&t\\
0&0&0&1\end{array} \right))$$
Which is
$$\frac{d}{dt}\bigg|_{t=0}\varphi_{f,\phi_{1},\phi_{2}}(\left( \begin{array}{cccc}
z_{1} &z_{2} &0&0 \\
z_{3}& z_{4}&0&0\\
0&0&b&bt\\
0&0&0&c\end{array} \right))\;=\;0\;=\;\frac{d}{dt}\bigg|_{t=0}f(A)\cdot \phi_{1}(\frac{\text{det}A}{b^{2}})\phi_{2}(\frac{\text{det}A}{c^{2}})\;=\;0$$
Likewise, $E_{4,3}$ acts by $0$ as a differential operator. The terms which contribute non-trivially to the effect of the $PGL_{4}(\mathbb{R})$-Laplacian are
$$(H_{1}^{2}+H_{2}^{2}+E_{1,2}E_{2,1}+E_{2,1}E_{1,2})+H_{3}^{2}+H_{4}^{2}$$
the parenthetical expression acts by a scalar $\lambda_{f}$ on the cuspform $f$.  That is,
$$(H_{1}^{2}+H_{2}^{2}+E_{1,2}E_{2,1}+E_{2,1}E_{1,2})\varphi_{f,\phi_{1},\phi_{2}}\;=\;\varphi_{\lambda_{f}f,\phi_{1},\phi_{2}}$$
since $H_{1}^{2}+H_{2}^{2}+E_{1,2}E_{2,1}+E_{2,1}E_{1,2}$ is the Laplacian on $PGL_{2}(\mathbb{R})$.  The remaining two terms in expression act as follows:
$$H_{3}^{2}\varphi_{f,\phi_{1},\phi_{2}}\;=\;\varphi_{f,4\phi''_{1},\phi_{2}}$$
Therefore,
$$(H_{1}^{2}+H_{2}^{2}+E_{1,2}E_{2,1}+E_{2,1}E_{1,2}+H_{3}^{2}+H_{4}^{2})\varphi_{f,\phi_{1},\phi_{2}}\;=\varphi_{\lambda_{f}f,\phi_{1},\phi_{2}}+\varphi_{f,4\phi''_{1},\phi_{2}}+\varphi_{f,\phi_{1},4\phi''_{2}}$$
Therefore, with $\Delta$ the $PGL_{4}(\mathbb{R})$-Laplacian,
$$\Delta\Psi_{\varphi}\;=\;\Psi_{\varphi_{\lambda_{f},\phi_{1},\phi_{2}}}+\Psi_{\varphi_{f,4\phi''_{1},\phi_{2}}}+\Psi_{\varphi_{f,\phi_{1},4\phi''_{2}}}$$
is again in the vector space spanned by $2,1,1$ pseudo-Eisenstein series, so is orthogonal to all other non-associate pseudo-Eisenstein series in $L^{2}$, as claimed.\end{proof}

We review Maass-Selberg relations and the theory of the constant term for $GL_{4}$,  as in [Harish-Chandra, p.75], [MW, p.100-101] and [Garrett 2011a].  Let $P=P^{2,2}$ be the standard, maximal parabolic subgroup
\[ P^{2,2}=\left( \begin{array}{cc}
GL_{2} & * \\
0& GL_{2}\end{array} \right)\]
 with unipotent radical $N^{P}$ and standard Levi component $M^{P}$.  The parabolic $P$ is self-associate. Let $f$ be an everywhere spherical cuspform on $GL_{2}(k)\backslash GL_{2}(\mathbb{A})$ with trivial central character and let $\varphi$ be the vector
$$\varphi(nmk)\;=\;\varphi_{s,f}(nmk)\;=\;|\text{det}\;m_{1}|^{2s}|\text{det}\;m_{2}|^{-2s}\cdot f(m_{1})\cdot \overline{f}(m_{2})$$
where
\[ m=\left( \begin{array}{cc}
m_{1} & * \\
0& m_{2}\end{array} \right)\]
with $m_{1},m_{2}$ in $GL_{2}$, so that $m$ is in the standard Levi component $M$ of the parabolic subgroup $P$, $n\in N$ its unipotent radical, $k\in K$, and $|\cdot|$ is the idele norm.  
\begin{defn} The spherical Eisenstein series is $$E_{s,f}^{P}(g)\;=\;E_{s,f}(g)\;=\;\sum_{\gamma\in P_{k}\backslash G_{k}}\varphi_{s,f}^{P}(\gamma\cdot g)\;\;\;\;\text{for}\;\text{Re}(s)\gg 1$$
\end{defn}
For $\text{Re}(s)$ sufficiently large, this series converges absolutely and uniformly on compacta.  We define truncation operators.  For a standard maximal proper parabolic $P=P^{2,2}$ as above, for $g=nmk$ with $$m=\left( \begin{array}{cc}
m_{1} & * \\
0& m_{2}\end{array} \right)$$ as above, $n\in N^{P}$ and $k\in O(4)$ define the spherical function
$$h^{P}(g)\;=\;h^{P}(pk)\;=\;\frac{|\text{det}\;m_{1}|^{2}}{|\text{det}\;m_{2}|^{2}}\;=\;\delta^{P}(nm)\;=\;\delta^{P}(m)$$
where $\delta^{P}$ is the modular function on $P$.  For fixed large real $T$, the $T$-tail of the $P$-constant term of a left $N_{k}^{P}$-invariant function $F$
\begin{displaymath}
   c_{P}^{T}F(g) = \left\{
     \begin{array}{lr}
       c_{P}F(g) & : h^{P}(g)\geq T\\
       0 & : h^{P}(g) \leq T
     \end{array}
   \right.
\end{displaymath} 
\begin{defn} The \textit{truncation operator} is
$$\Lambda^{T}E_{\varphi}^{P}\;=\; E_{\varphi}^{P}-E^{P}(c_{P}^{T}E_{\varphi}^{P}) $$
where
$$E^{P}(\varphi)(g)\;=\;\sum_{\gamma\in P_{\mathbb{Z}}\backslash\Gamma}\varphi(\gamma g)$$
\end{defn}
These are square-integrable, by the theory of the constant term([MW, pp.18-40], [Harish-Chandra]).  The Maass-Selberg relations  describe their inner product as follows. The inner product $$\langle \Lambda^{T}E_{\varphi}^{P},\Lambda^{T}E_{\psi}^{P}\rangle$$ of truncations $\Lambda^{T}E_{\varphi}^{P}$ and $\Lambda^{T}E_{\psi}^{P}$ of two Eisenstein series $E_{\varphi}^{P}$ and $E_{\psi}^{P}$ attached to cuspidal-data $\varphi$, $\psi$ on maximal proper parabolics $P=P^{2,2}$ is given as follows.  The term $c_{s}$ refers to the quotient of Rankin-Selberg L-functions appearing in the constant term $c_{P}E_{\varphi}^{P}$.  That is, $$c_{s}=\frac{L(2s-1,\pi\otimes \pi^{'})}{L(2s,\pi\otimes \pi^{'})}$$ as in [Langlands 544,Section 4] where $\pi$ is locally everywhere an unramified principal series isomorphic to the representation generated by the cuspform $f$ locally.
\begin{prop}{\textbf{Maass-Selberg relations}}\begin{align*}\langle \Lambda^{T}E_{g_{1}}^{P},\Lambda^{T}E_{g_{2}}^{P}\rangle=&\langle g_{1},g_{2}\rangle \frac{T^{s+\overline{r}-1}}{s+\overline{r}-1}+\langle g_{1},g_{2}^{w}\rangle \overline{c_{r}^{g_{2}}}\frac{T^{s+(1-\overline{r})-1}}{s+(1-\overline{r})-1}\\+&\langle g_{1}^{w},g_{2}\rangle c_{s}^{g_{1}}\frac{T^{(1-s)+\overline{r}-1}}{(1-s)+\overline{r}-1}+\langle g_{1}^{w},g_{2}^{w}\rangle c_{s}^{g_{1}}\overline{c_{r}^{g_{2}}}\frac{T^{(1-s)+(1-\overline{r})-1}}{(1-s)+(1-\overline{r})-1}\end{align*}\end{prop}

Following [M-W pp.18-40], an important consequence of the Maass-Selberg relations is that for a maximal, proper, self-associate parabolic $P$ in $GL_{n}$, on the half-plane $\text{Re}(s)\geq\frac{1}{2}$ the only possible poles are on the real line, and only occur if $\langle f,f^{w}\rangle\neq 0$.  In that case, any pole is simple, and the residue is square-integrable.  In particular, taking $f=f_{o}\times f_{o}$
$$\langle \text{Res}_{s_{o}}E_{\varphi}^{P},\text{Res}_{s_{o}}E_{\varphi}^{P}\rangle\;=\;\langle f_{o},f_{o}\rangle^{2}\cdot \text{Res}_{s_{o}}c_{s}^{\varphi}$$
as in [Harish-Chandra,p.75].
The group $GL_{4}$ gives the first instance of non-constant, noncuspidal contribution to the discrete spectrum; the residues of the Eisenstein series at its poles give Speh forms.
Recall ([Langlands 544] Section 4, though he uses a different normalization), that the constant term is equal to
$$\big|\frac{\text{det} A}{\text{det} D}\big|^{s}\cdot f(A)\cdot\overline{f}(D) + \big|\frac{\text{det} A}{\text{det} D}\big|^{1-s}\cdot\frac{\Lambda(2s-1,\pi\otimes \pi^{'})}{\Lambda(2s,\pi\otimes \pi^{'})}\cdot \overline{f}(A)\cdot f(D)$$
The $L$-function appearing in the numerator necessarily has a residue at the unique pole in the right half-plane.  This residue of the Eisenstein series at this pole is the Speh form [Jacquet] attached to a $GL(2)$ cuspform $f$, and is in $L^{2}$.

We now compute the $2,2$ constant term of the $2,2$ Eisenstein series with cuspidal data $f$ and $\overline{f}$. Let $P=P^{2,2}$ be the self-associate standard parabolic in $G=GL_{4}$ with Levi component $GL_{2}\times GL_{2}$.  Let $f_{1}$ and $f_{2}$ be spherical cuspforms on $GL_{2}(k)\backslash GL_{2}(\mathbb{A})$.  Define the spherical vector
$$\varphi_{s,f_{1},f_{2}}^{P}(\left( \begin{array}{cc}
A & * \\
0& D\end{array} \right))\;=\;\big|\frac{\text{det}A}{\text{det} D}\big|^{s}\cdot f_{1}(A)\cdot f_{2}(D)$$
and then extending to $G_{\mathbb{A}}$ by right $K_{v}$-invariance and $Z_{v}$-invariance everywhere locally.  Define cuspidal-data Eisenstein series for $\text{Re}(s)\gg 1$ by

$$E_{s,f_{1},f_{2}}^{P}(g)\;=\;\sum_{\gamma\in P_{k}\backslash G_{k}}\varphi_{s,f_{1},f_{2}}^{P}(\gamma g)$$
\begin{prop} The $P$-constant term of the $P$-Eisenstein series $E_{s,f_{1},f_{2}}^{P}(g)$ is given by
$$c_{P}E_{s,f_{1},f_{2}}^{P}(g)\;=\;\big|\frac{\text{det}A}{\text{det} D}\big|^{s}\cdot f_{1}(A)\cdot f_{2}(D)+\big|\frac{\text{det}A}{\text{det} D}\big|^{1-s}\cdot f_{1}(A)\cdot f_{2}(D)\cdot\frac{L(\pi_{1}\otimes \pi_{2},2s-1)}{L(\pi_{1}\otimes \pi_{2},2s)}$$
where $\pi_{1}$ is the $G_{\mathbb{A}}$-representation generated by $f_{1}$ and $\pi_{2}$ is the $G_{\mathbb{A}}$-representation generated by $f_{2}$.
\end{prop}
\begin{proof} The constant term of $E_{s,f_{1},f_{2}}$ along $P$ is given by
$$c_{P}E_{s,f_{1},f_{2}}^{P}(g)\;=\;\int_{N_{k}\backslash N_{\mathbb{A}}}E_{s,f_{1},f_{2}}^{P}(ng)\;dn\;=\;\sum_{\xi\in P_{k}\backslash G_{k}/N_{k}}\int_{\xi^{-1}P_{k}\xi\cap N_{k}\backslash N_{\mathbb{A}}}\varphi_{s,f_{1},f_{2}}(\xi\gamma ng)\;dn$$
The double coset space $P\backslash G/N$ surjects to $W^{P}\backslash W/W^{P}$ which has three double coset representatives, two of which give a nonzero contribution.  The identity coset contributes a volume, which we will compute later.  The nontrivial representative is $\xi=\sigma_{2}\sigma_{3}\sigma_{1}\sigma_{2}$.  Observe that $\xi\cdot P_{k}\cdot \xi^{-1}\cap N_{k}=\{1\}$ so that 
$$c_{P}E_{s,f_{1},f_{2}}^{P}(g)\;=\;\int_{N_{k}\backslash N_{\mathbb{A}}}\varphi_{s,f_{1},f_{2}}(ng)\;dn\;+\;\int_{N_{\mathbb{A}}}\varphi_{s,f_{1},f_{2}}(\xi ng)\;dn$$
To compute the contribution of the integral $$\int_{N_{\mathbb{A}}}\varphi_{s,f_{1},f_{2}}(\xi ng)\;dn$$ we must re-express the Eisenstein series representation-theoretically.  To this end, let $\pi_{f_{1}}=\otimes \pi_{f_{1},v}$ be the representation of $G_{\mathbb{A}}$ generated by $f_{1}$ and let $\pi_{f_{2}}=\otimes \pi_{f_{2},v}$ be the $G_{\mathbb{A}}$-representation generated by $f_{2}$.  For places $v$ outside a finite set $S$, fix isomorphisms $$j_{v}:\text{Ind}\chi_{f_{1},v}\rightarrow \pi_{f_{1},v}$$ and $$l_{v}:\text{Ind}\chi_{f_{2},v}\rightarrow \pi_{f_{2},v}$$  Their tensor product $j_{v}\otimes l_{v}$ is a representation of the Levi  $M=GL_{2}\otimes GL_{2}$.  Extend representations of Levi components trivially to parabolics.  A $\pi_{f}$-valued Eisenstein series is formed by a convergent sum
$$E_{\varphi}^{P}\;=\;\sum_{\gamma\in P_{k}\backslash G_{k}}\varphi\circ\gamma$$
Let $T=\otimes_{v}T_{v}:\varphi\rightarrow \int_{N_{\mathbb{A}}}\varphi(\xi ng)\;dn$.  We have a chain of intertwinings

\newpage$$\bigotimes_{v\in S}\text{Ind}_{P_{v}}^{G_{v}}\big((\pi_{f_{1},v})\otimes\pi_{f_{2},v})\nu_{P_{v}}^{s}\big)\otimes\bigotimes_{v\notin S}\text{Ind}_{B_{v}}^{G_{v}}\big((\chi_{f_{1},v}\otimes\chi_{f_{2},v})\nu_{B_{v}}^{s,s,s,-3s}\big)$$

$$\updownarrow_{\text{iterated induction}}$$

$$\bigotimes \text{Ind}_{P_{v}}^{G_{v}}(\pi_{f_{1},v}\otimes \pi_{f_{2},v})v_{P_{v}}^{s}\otimes\bigotimes\text{Ind}_{P_{v}}^{G_{v}}\big(\text{Ind}_{B_{v}}^{P_{v}}(\chi_{f_{1},v}\otimes \chi_{f_{2},v})\nu_{B_{v}}^{s,s,s,-3s}\big)$$

$$\updownarrow_{1\otimes\big(\otimes\text{Ind}_{P_{v}}^{G_{v}}(j_{v}\otimes l_{v})\big)}$$

$$\bigotimes\text{Ind}_{P_{v}}^{G_{v}}(\pi_{f_{1},v}\otimes \pi_{f_{2},v})\nu_{P_{v}}^{s}\otimes\bigotimes\text{Ind}_{P_{v}}^{G_{v}}\big((\pi_{f_{1},v}\otimes \pi_{f_{2},v})\nu_{P_{v}}^{s}\big)$$

$$\updownarrow_{T=\otimes T_{v}}$$

$$\bigotimes\text{Ind}_{P_{v}}^{G_{v}}\big((\pi_{f_{1},v}\otimes\pi_{f_{2},v})\nu_{P_{v}}^{1-s}\big)\otimes\bigotimes\text{Ind}_{P_{v}}^{G_{v}}\big((\pi_{f_{1},v}\otimes\pi_{f_{2},v}\nu_{P_{v}}^{1-s}\big)$$

$$\updownarrow_{1\otimes\big(\bigotimes\text{Ind}_{P_{v}}^{G_{v}}(j_{v}^{-1}\otimes l_{v}^{-1})\big)}$$

$$\bigotimes\text{Ind}_{P_{v}}^{G_{v}}\big((\pi_{f_{1},v}\otimes\pi_{f_{2},v})\nu_{P_{v}}^{1-s}\big)\otimes\bigotimes\text{Ind}_{P_{v}}^{G_{v}}\big(\text{Ind}_{B_{v}}^{P_{v}}(\chi_{f_{1},v}\otimes\chi_{f_{2},v})\nu_{B_{v}}^{3-3s,s-1,s-1,s-1}\big)$$

$$\updownarrow_{\text{iterated induction}}$$

$$\bigotimes\text{Ind}_{P_{v}}^{G_{v}}\big((\pi_{f_{1},v}\otimes\pi_{f_{2},v})\nu_{P_{v}}^{1-s}\otimes 1\big)\otimes\bigotimes\text{Ind}_{B_{v}}^{G_{v}}(\chi_{f_{1},v}\otimes\chi_{f_{2},v})\nu_{B_{v}}^{3-3s,s-1,s-1,s-1}$$

The advantage of this set-up is that for $v$ outside the finite set S, the minimal parabolic unramified principal series has a canonical spherical vector, namely that spherical vector taking value $1$ at $1\in G_{v}$.  Therefore the isomorphism $T_{v}$ can be completely determined by computing its effect on the canonical spherical vector.  The intertwinings $T_{v}$ among minimal-parabolic principal series can be factored as compositions of similar intertwining operators attached to reflections corresponding to positive simple roots, each of which is completely determined by its effect on the canonical spherical vector in the unramified principal series.   The simple reflection intertwinings' effect on the normalized spherical functions reduce to $GL_{2}$ computations.\\

\noindent Thus, with simple reflections
$$\sigma_{1}\;=\;\left( \begin{array}{cccc}
0 & 1&0&0 \\
1& 0&0&0\\
0 & 0 & 1&0\\
0&0&0&1\end{array} \right)\;\;\;\;\sigma_{2}\;=\;\left( \begin{array}{cccc}
1 & 0&0&0 \\
0& 0&1&0\\
0 & 1 & 0&0\\
0&0&0&1\end{array} \right)\;\;\;\;\;\sigma_{3}\;=\;\left( \begin{array}{cccc}
1 & 0&0&0 \\
0& 1&0&0\\
0 & 0 & 0&1\\
0&0&1&0\end{array} \right)$$
and with corresponding root subgroups
$$N_{\sigma_{1}}\;=\;\left( \begin{array}{cccc}
1 & x&0&0 \\
0& 1&0&0\\
0 & 0 & 1&0\\
0&0&0&1\end{array} \right)\;\;\;\;\;N_{\sigma_{2}}\;=\;\left( \begin{array}{cccc}
1 & 0&0&0 \\
0& 1&y&0\\
0 & 0 & 1&0\\
0&0&0&1\end{array} \right)\;\;\;\;\;N_{\sigma_{3}}\;=\;\left( \begin{array}{cccc}
1 & 0&0&0 \\
0& 1&0&0\\
0 & 0 & 1&z\\
0&0&0&1\end{array} \right)$$
The simple-reflection intertwinings
\begin{align*}&S_{\sigma_{1}}f(g)\;=\;\int_{N_{\sigma_{1}}}f(\sigma_{1}ng)\;dn\;\;\;\;S_{\sigma_{2}}f(g)\;=\;\int_{N_{\sigma_{2}}}f(\sigma_{2}ng)\;dn\\& S_{\sigma_{3}}f(g)\;=\;\int_{N_{\sigma_{3}}}f(\sigma_{3}ng)\;dn\end{align*}
are instrumental because we wish to compute the effect of
$$S_{\sigma_{2}}\circ S_{\sigma_{3}}\circ S_{\sigma_{1}}\circ S_{\sigma_{2}}$$
on the normalized spherical vector in the unramified minimal-parabolic principal series $I(s_{1},s_{2},s_{3},s_{4})$.  Furthermore,
$$S_{\sigma\tau}\;=\;S_{\sigma}\circ S_{\tau}$$
Therefore, we must understand the effect of the individual $S_{\sigma_{i}}$'s.  Recall that
$$S_{\sigma_{2}}:I(s_{1},s_{2},s_{3},s_{4})\rightarrow I(s_{1},s_{3}+1,s_{2}-1,s_{4})$$
Similarly,
$$S_{\sigma_{1}}:I(s_{1},s_{2},s_{3},s_{4})\rightarrow I(s_{2}+1,s_{1}-1,s_{3},s_{4})$$
and
$$S_{\sigma_{3}}:I(s_{1},s_{2},s_{3},s_{4})\rightarrow I(s_{1},s_{2},s_{4}+1,s_{3}-1)$$
The normalized spherical function $f^{0}\in I(s_{1},s_{2},s_{3},s_{4})$ is mapped by $S_{\sigma_{1}}$ to a multiple of the normalized spherical function in $I(s_{2}+1,s_{1}-1,s_{3},s_{4})$.  The constant is
$$S_{\sigma_{1}}f^{0}(1)\;=\;\int f^{0}(\sigma_{1}\left( \begin{array}{cccc}
1 & x&0&0 \\
0& 1&0&0\\
0 & 0 & 1&0\\
0&0&0&1\end{array} \right))\;dx\;=\;\int f^{0}(\left( \begin{array}{cccc}
1 & 0&0&0 \\
x& 1&0&0\\
0 & 0 & 1&0\\
0&0&0&1\end{array} \right))\;dx$$
Using the Iwasawa decomposition for $GL_{2}(k_{v})$, we show that this calculation reduces to a $GL_{2}$ calculation.  Indeed, there is
$\left( \begin{array}{cc}
a & b \\
c& d\end{array} \right)$ in the maximal compact of $GL_{2}(k_{v})$ such that
$$\left( \begin{array}{cc}
1 & 0 \\
x& 1\end{array} \right)\left( \begin{array}{cc}
a & b \\
c& d\end{array} \right)\;=\;\left( \begin{array}{cc}
* & * \\
0& *\end{array} \right)$$
Therefore,
$$\left( \begin{array}{cccc}
1 & 0&0&0 \\
x& 1&0&0\\
0 & 0 & 1&0\\
0&0&0&1\end{array} \right)\left( \begin{array}{cccc}
a & b&0&0 \\
c& d&0&0\\
0 & 0 & 1&0\\
0&0&0&1\end{array} \right)\;=\;\left( \begin{array}{cccc}
* & *&0&0 \\
0& *&0&0\\
0 & 0 & 1&0\\
0&0&0&1\end{array} \right)$$
From this, it follows that the constant $S_{\sigma_{1}}f^{0}(1)$ with $$S_{\sigma_{1}}:I(s_{1},s_{2},s_{3},s_{4})\rightarrow I(s_{2}+1,s_{1}-1,s_{3},s_{4})$$ is the same as the constant in the intertwining from $I(s_{1},s_{2})\rightarrow I(s_{2}+1,s_{1}-1)$ of $GL_{2}$ principal series, namely
$$\varphi^{0}(\left( \begin{array}{cc}
1 & 0 \\
x& 1\end{array} \right))\;dx$$
where $\varphi^{0}$ is the normalized spherical vector in the $GL_{2}$ principal series.  A similar argument applies to the other intertwining operators attached to other simple reflections. We recall the $GL_{2}$ computation below.  At absolutely unramified finite places, $\left( \begin{array}{cc}
1 & 0 \\
x& 1\end{array} \right)\in K_{v}=GL_{2}(\sigma_{v})$ for $x\leq 1$.  For $x>1$, 
$$\left( \begin{array}{cc}
1 & 0 \\
x& 1\end{array} \right)\left( \begin{array}{cc}
1 & -\frac{1}{x} \\
x& 1\end{array} \right)=\left( \begin{array}{cc}
\frac{1}{x} & 1 \\
0& x\end{array} \right)\left( \begin{array}{cc}
0 & -1 \\
1& 0\end{array} \right)$$
Thus, with local parameter $\overline{\omega}$ and residue field cardinality $q$, since the measure of
$$\{x\in k_{v}\;:\;|x|=q^{r}\}$$
is $(q-1)q^{r-1}$, we see that
$$\int_{k_{v}}\varphi^{0}(\left( \begin{array}{cc}
1 & 0 \\
x& 1\end{array} \right))\;dx\;=\;\int_{|x|\leq 1}1\;dx\;+\;\int_{|x|>1}\varphi^{0}(\left( \begin{array}{cc}
\frac{1}{x} & 1 \\
0& 1\end{array} \right))\;dx$$
This is
$$1+(1-q)\sum_{r\geq 1}q^{r(1-s_{1}+s_{2})}\;=\;\frac{\zeta_{v}(s_{1}-s_{2}-1)}{\zeta_{v}(s_{1}-s_{2})}$$
with the Iwasawa-Tate unramified local zeta integral $\zeta_{v}(s)$.\\
\noindent Using this $GL_{2}$ reduction, we see that $$S_{\sigma_{2}}:I(s_{1},s_{2},s_{3},s_{4})\rightarrow I(s_{1},s_{3}+1,s_{2}-1,s_{4})$$ and maps the normalized spherical vector in $I(s_{1},s_{2},s_{3},s_{4})$ to $$\frac{\zeta_{v}(s_{2}-s_{3}-1)}{\zeta_{v}(s_{2}-s_{3})}$$ times the normalized spherical function in $I(s_{1},s_{3}+1,s_{2}-1,s_{4})$.  Then
$$S_{\sigma_{1}}:I(s_{1},s_{3}+1,s_{2}-1,s_{4})\rightarrow I(s_{3}+2,s_{1}-1,s_{2}-1,s_{4})$$
and sends the normalized spherical function in $I(s_{1},s_{3}+1,s_{2}-1,s_{4})$ to $$\frac{\zeta_{v}(s_{1}-s_{3}-2)}{\zeta_{v}(s_{1}-s_{3}-1)}$$ times the normalized spherical function in $I(s_{3}+2,s_{1}-1,s_{2}-1,s_{4})$.  Then $S_{\sigma_{3}}$ maps the normalized spherical vector in $I(s_{3}+2,s_{1}-1,s_{2}-1,s_{4})$ to $$\frac{\zeta_{v}(s_{2}-s_{4}-2)}{\zeta_{v}(s_{2}-s_{4}-1)}$$ times the normalized spherical vector in $I(s_{3}+2,s_{1}-1,s_{4}+1,s_{2}-2)$. Finally, $S_{\sigma_{2}}:I(s_{3}+2,s_{1}-1,s_{4}+1,s_{2}-2)\rightarrow I(s_{3}+2,s_{4}+2,s_{1}-2,s_{2}-2)$ and sends the normalized spherical function in $I(s_{3}+2,s_{1}-1,s_{4}+1,s_{2}-2)$ to $$\frac{\zeta_{v}(s_{1}-s_{4}-3)}{\zeta_{v}(s_{1}-s_{4}-2)}$$ times the normalized spherical function in $I(s_{3}+2,s_{4}+2,s_{1}-2,s_{2}-2)$.  Altogether, $S_{\sigma_{2}}\circ S_{\sigma_{3}}\circ S_{\sigma_{1}}\circ S_{\sigma_{2}}$ maps the normalized spherical vector in $I(s_{1},s_{2},s_{3},s_{4})$ to

$$\frac{\zeta_{v}(s_{2}-s_{3}-1)}{\zeta_{v}(s_{2}-s_{3})}\cdot\frac{\zeta_{v}(s_{1}-s_{3}-2)}{\zeta_{v}(s_{1}-s_{3}-1)}\cdot\frac{\zeta_{v}(s_{3}-s_{4}-2)}{\zeta_{v}(s_{2}-s_{4}-1)}\cdot\frac{\zeta_{v}(s_{1}-s_{4}-3)}{\zeta_{v}(s_{1}-s_{4}-2)}$$
times the normalized spherical vector in the unramified principal series $$I(s_{3}+2,s_{4}+2,s_{1}-2,s_{2}-2)$$
For $(s_{1},s_{2},s_{3},s_{4})=(s+s_{f_{1}},s-s_{f_{1}},-s+s_{f_{2}},-s-s_{f_{2}})$ we get 
\begin{align*}&\frac{\zeta_{v}(s-s_{f_{1}}-(-s+s_{f_{2}})-1)}{\zeta_{v}(s-s_{f_{1}}-(-s+s_{f_{2}}))}\cdot\frac{\zeta_{v}(s+s_{f_{1}}-(-s+s_{f_{2}})-2)}{\zeta_{v}(s+s_{f_{1}}-(-s+s_{f_{2}})-1)}\\
&\cdot\frac{\zeta_{v}(s-s_{f_{1}}-(-s-s_{f_{2}})-2)}{\zeta_{v}(s-s_{f_{1}}-(s-s_{f_{2}})-1)}\cdot\frac{\zeta_{v}(s+s_{f_{1}}-(-s-s_{f_{2}})-3)}{\zeta_{v}(s+s_{f_{1}}-(-s-s_{f_{2}})-2)}\end{align*}
the Rankin Selberg L-function
$$\frac{L(\pi_{1}\otimes\pi_{2},2s-1)}{L(\pi_{1}\otimes\pi_{2},2s)}$$
\end{proof}

\section{Global Automorphic Sobolev Spaces}
We recall basic ideas about global automorphic Sobolev spaces.  For example, see Decelles [2011b], [Grubb], and [Garrett 2010]. Consider the group $G=GL(4)$ defined over a number field $k$.  At each place $v$, let $K_{v}$ be the standard maximal compact subgroup of the $v$-adic points $G_{v}$ of $G$. That is, $K_{v}=GL_{4}(\mathfrak{O}_{v})$ for nonarchimedean places $v$ where $\mathfrak{O}_{v}$ denotes the local ring of integers, and $K_{v}=O_{4}(\mathbb{R})$ for $v$ real and $K=U(n)$ for $v$ complex. Consider the space $C_{c}^{\infty}(Z_{\mathbb{A}}G_{k}\backslash G_{\mathbb{A}},\omega)$ where $\omega$ is a trivial central character. We define positive index global archimedean spherical automorphic Sobolev spaces as right $K_{\infty}$-invariant subspaces of  completions of $C_{c}^{\infty}(Z_{\mathbb{A}}G_{k}\backslash G_{\mathbb{A}},\omega)$ with respect to a topology induced by norms associated to the Casimir operator $\Omega$. The operator $\Omega$ acts on the archimedean component $f\in C_{c}^{\infty}(Z_{\mathbb{A}}G_{k}\backslash G_{\mathbb{A}},\omega)$ by taking derivatives in the archimedean component.  The norm $|.|_{\ell}$ on $C_{c}^{\infty}(Z_{\mathbb{A}}G_{k}\backslash G_{\mathbb{A}},\omega)^{K}$ is
$$|f|_{\ell}=\langle (1-\Omega)^{\ell}f,f\rangle^{\frac{1}{2}}$$ 
where $\langle,\rangle$ gives the norm on $L^{2}(Z_{\mathbb{A}}G_{k}\backslash G_{\mathbb{A}},\omega)$, induces a topology on the space $C_{c}^{\infty}(Z_{\mathbb{A}}G_{k}\backslash G_{\mathbb{A}},\omega)^{K}$. 
\begin{defn}The completion $H^{\ell}(Z_{\mathbb{A}}G_{k}\backslash G_{\mathbb{A}},\omega)$ is the $\ell$-th global automorphic Sobolev space.\end{defn}
$H^{\ell}(Z_{\mathbb{A}}G_{k}\backslash G_{\mathbb{A}},\omega)$is a Hilbert space with respect to this topology. 
\begin{defn}For $\ell>0$, the Sobolev space $H^{-\ell}(Z_{\mathbb{A}}G_{k}\backslash G_{\mathbb{A}},\omega)$ is the Hilbert space dual of $H^{\ell}(Z_{\mathbb{A}}G_{k}\backslash G_{\mathbb{A}},\omega)$.\end{defn}
Since the space of test functions is a dense subspace of $H^{\ell}(Z_{\mathbb{A}}G_{k}\backslash G_{\mathbb{A}},\omega)$ with $\ell>0$, dualizing gives an inclusion of $H^{-\ell}(Z_{\mathbb{A}}G_{k}\backslash G_{\mathbb{A}},\omega)$ into the space of distributions. The adjoints of the dense inclusions $H^{\ell}\rightarrow H^{\ell-1}$ are inclusions $$H^{-\ell+1}(Z_{\mathbb{A}}G_{k}\backslash G_{\mathbb{A}},\omega)\rightarrow H^{-\ell}(Z_{\mathbb{A}}G_{k}\backslash G_{\mathbb{A}},\omega)$$

\section{Pre-trace formula estimates on compact periods}
We give a standard argument.  See, for example, [Iwaniec] and [Garrett 2010]. Set $k=\mathbb{Q}$ throughout. Let $\Theta$ be a $k$-subgroup of $G$.  Let $[\Theta]=(Z_{\mathbb{A}}\cap\Theta)\Theta_{k}\backslash\Theta_{\mathbb{A}}$ and $[G]=Z_{\mathbb{A}}G_{k}\backslash G_{\mathbb{A}}/K_{\infty}$. For smooth $f$ on $Z_{\mathbb{A}}G_{k}\backslash G_{\mathbb{A}}$, define the $[\Theta]x$-period of $f$ to be 
$$f_{\Theta,x}\;=\;\int\limits_{[\Theta]}f(hx)\;dh$$
Similarly, with $\phi$ an automorphic form on $\Theta_{k}\backslash \Theta_{\mathbb{A}}$, the $[\Theta],x,\phi$-period of $f$ is

\[
\langle f,\overline{\phi}\rangle_{\Theta}\;=\;\int\limits_{[\Theta]}\overline{\phi}(h)\cdot f(hx)\;dh
\]
For finite places $v$, fix a compact open subgroup $K_{v}$ of $G_{v}$  such that at almost all places, $K_{v}$ is the standard maximal compact subgroup of $G_{v}$, and let $K_{\text{fin}}=\prod_{v}K_{v}$.  Let $K=K_{\infty}\cdot K_{\text{fin}}$.

\begin{prop} The distribution given by integration along a compact quotient $\Theta_{k}\backslash \Theta_{\mathbb{A}}$ lies in $H^{-s}(Z_{\mathbb{A}}G_{k}\backslash G_{\mathbb{A}})$ for all $$s>\frac{\text{dim}(G_{\infty}/K_{\infty})-\text{dim}(\Theta_{\infty}/K_{\infty}^{\Theta})}{2}$$
\end{prop}
\begin{proof}
Consider smooth $f$ on $Z_{\mathbb{A}}G_{k}\backslash G_{\mathbb{A}}$ generating 
unramified principal series at archimedean places.  
The usual action of compactly-supported measures $\eta$ 
on suitable $f$ on $G_{k}\backslash G_{\mathbb{A}}/K_{\infty}$ is given by 
$$(\eta\cdot f)(x)=\int\limits_{G_{\mathbb{A}}}\eta(g)f(xg)\;dg$$  The 
$\Theta_{k}\backslash \Theta_{\mathbb{A}}x$-period of $\eta\cdot f$ 
admits a useful rearrangement

\begin{align*}
&(\eta\cdot f)_{Z_{\mathbb{A}}\Theta_{k}\backslash \Theta_{\mathbb{A}}x}\;=\;\int\limits_{Z_{\mathbb{A}}\Theta_{k}\backslash \Theta_{\mathbb{A}}}
(\eta\cdot f)(hx)\;dh\;=\;\int\limits_{[\Theta]}\int\limits_{G_{\mathbb{A}}}
\eta(g)f(hxg)\;dg\;dh\\
&=\int\limits_{[\Theta]}\int\limits_{G_{\mathbb{A}}}
\eta(x^{-1}h^{-1}g)\;dg\;dh =\int\limits_{[\Theta]}\int\limits_{[G]}\sum_{\gamma\in G_{k}}\eta(x^{-1}h^{-1}\gamma g)f(g)\;dg\;dh\\ &=\int\limits_{[G]}f(g)\left(\int\limits_{[\Theta]}\sum_{\gamma\in G_{k}}\eta(x^{-1}h^{-1}\gamma g)\;dh\right)\;dg
\end{align*}
Denote the inner sum and integral by $q(g)=q_{\Theta,x}(g)$.  
For $\eta$ a left and right $K_{\text{fin}}$-invariant measure, for $f$ a spherical vector in a copy of a principal series, $\eta\cdot f$ will be $K_{\text{fin}}$-invariant.  Since the spherical vector in an irreducible representation is unique (up to scalar), $\eta\cdot f=\lambda_{f}(\eta)\cdot f$ for some constant $\lambda_{f}(\eta)$. Let $\eta_{\infty}$ be the characteristic function of a shrinking ball $B_{\epsilon}$ in $G_{\infty}/K_{\infty}$ of geodesic radius $\epsilon >0$ and at each finite place $v$, let $\eta_{v}$ be the characteristic function of $K_{v}$. The ball $B_{\epsilon}$ has $v$-adic components in $K_{v}$ for almost all $v$, and archimedean component lying within a ball of radius $\epsilon$. Identify $B_{\epsilon}$ with its pre-image $B_{\epsilon}\cdot K_{v}$ in $G_{v}$. Here, we make use of a $G_{\infty}$-invariant metric $$d(x,y)=\nu(x^{-1}y)$$ on $G_{\infty}/K_{\infty}$ where $${\nu(g)=\text{log sup}(|g|,|g^{-1}|)}$$ Here $|\cdot|$ is the operator norm on the group $G_{v}$ given by $$|T|=\text{sup}_{u\leq 1}||Tu||$$ Let $\eta=\otimes_{v}\eta_{v}$. The action of such $\eta$ changes the period by the eigenvalue.  To see this, observe that $$(\eta_{v}\cdot f)(x)=\int_{K_{v}}\eta_{v}(k)f(gk)\;dk=\int_{K_{v}}f(g)dk=\text{vol}(K_{v})\cdot f(g)$$ Also, $\eta_{\infty}\cdot f$ will be a spherical vector.  Since the spherical vector is unique up to a constant multiple, $\eta_{\infty}\cdot f=\lambda_{\infty}\cdot f$ for some scalar $\lambda_{\infty}$. Therefore,

$$(\eta\cdot f)_{\Theta,x}\;=\;\lambda_{f}(\eta)\cdot\text{vol}(K_{\text{fin}})\cdot f_{\Theta,x}$$
An upper bound for the $L^{2}(Z_{\mathbb{A}}G_{k}\backslash G_{\mathbb{A}},\omega)$ norm of $q$, and a lower bound for $\lambda_{f}(\eta)$ contingent on restrictions on the spectral parameter of $f$, yield, by Bessel's inequality, an upper bound for a sum-and-integral of periods $\langle f,\phi\rangle_{\Theta,x}$ as follows.
Estimate the $L^{2}$ norm of $q$:

\begin{align*}
&\int\limits_{[G]}|q(g)|^{2}\;=\;\int\limits_{[G]}\int\limits_{[\Theta]}\int\limits_{[\Theta]}\sum_{\gamma\in G_{k}} \sum_{\gamma_{2}\in G_{k}}\eta(x^{-1}h^{-1}\gamma g)\overline{\eta}(x^{-1}h_{2}^{-1}\gamma_{2}g)\;dg\;dh_{2}\;dg\\
&=\;\int\limits_{G_{\mathbb{A}}}\int\limits_{[\Theta]}\int\limits_{[\Theta]}\sum_{\gamma\in G_{k}}\eta(x^{-1}h^{-1}\gamma g)\eta(x^{-1}h_{2}^{-1}g)\;dh\;dh_{2}\;dg
\end{align*}
With $C$ a large enough compact subset of $\Theta_{\mathbb{A}}$ to surject to $[\Theta]=(Z_{\mathbb{A}}\cap\Theta)\Theta_{k}\backslash\Theta_{\mathbb{A}}$,

$$\int\limits_{[G]}|q(g)|^{2}\leq \int\limits_{G_{\mathbb{A}}}\int\limits_{C}\int\limits_{C}\sum_{\gamma\in G_{k}}|\eta|(x^{-1}h^{-1}\gamma g)|\eta|(x^{-1}h_{2}^{-1}g)\;dh\;dh_{2}\;dg$$
The set

\begin{align*}
&\Phi\;=\;\Phi_{H,x,\eta}\;\\&=\;\{\gamma\in G_{k}:\eta(x^{-1}h_{2}^{-1}\gamma g)\eta(x^{-1}h_{2}^{-1}g)\neq 0\;\text{for some}\;h,h_{2}\in C\;\text{and}\; g\in G_{\mathbb{A}}\}\\
&=\;\{\gamma\in G_{k}:\gamma\in CxB_{\epsilon}g^{-1}, g\in CxB_{\epsilon}\}\;\subset\;G_{k}\cap CxB_{\epsilon}\cdot (CxB_{\epsilon})^{-1}\}
\end{align*}
the last set in the sequence above is the intersection of a closed, discrete set with a compact set, so is finite, and can only shrink as $\epsilon\rightarrow 0^{+}$.

For $K_{0}$ a compact open subgroup in the finite adele part $G_{0}$ of $G_{\mathbb{A}}$, a ball of archimedean radius $\epsilon$ is the product $B_{\epsilon}\times K_{0}$.  Here $B_{\epsilon}$ is the inverse image in $G_{\infty}$ of the geodesic ball of radius $\epsilon$ in $G_{\infty}/K_{\infty}$. For each $\gamma\in\Phi$, for each $h\in C$, $\eta(x^{-1}h^{-1}\gamma g)\neq 0$ only for $g$ in a ball in $X=G_{\mathbb{A}}/K_{\infty}$ of radius $\epsilon$, with volume dominated by $\epsilon^{\text{dim} X}$.  Thus,
$$\int\limits_{G_{k}\backslash G_{\mathbb{A}}}|q(g)|^{2}\;dg\;\ll\;\int\limits_{C}\epsilon^{\text{dim} X+\text{dim} Y}\;dh\;\ll\;\epsilon^{\text{dim} X+\text{dim} Y}$$
By automorphic Plancherel, with $\eta$ as above,
$$\sum_{\text{cfm} F}|\lambda_{F}(\eta)|^{2}\cdot |\langle \eta\cdot F,\phi\rangle|^{2}+\ldots\;\ll\;\epsilon^{\text{dim} X+\text{dim} Y}$$
Next,we give a bound on the spectral data to give a non-trivial lower bound for $\lambda_{f}(\eta)$. Left and right $K$-invariant $\eta$ necessarily gives $\eta\cdot f=\lambda_{f}(\eta)\cdot f$, since up to scalars $f$ is the unique spherical vector in the irreducible representation $f$ generates.  This is an intrinsic representation-theoretic relation, because an isomorphism of principal series sends a spherical vector in the first representation to a constant multiple of the spherical vector in the second representation.  That is, if $$\varphi:V\rightarrow W$$ is an isomorphism of representations, and $f_{1}\in V$ and $f_{2}\in W$ are the unique spherical vectors, then $$\varphi(f_{1})=c\cdot f_{2}$$ for a constant $c$.  To see this, observe that $$k\cdot\varphi(f_{1})=\varphi(k\cdot f_{1})=\varphi(f_{1})$$  Therefore, $\varphi(f_{1})$ is indeed invariant under the $K$-action, so is the spherical vector in the representation $V_{2}$. Then a calculation gives
$$\lambda_{f_{1}}(\eta)\cdot\varphi(f_{1})=\eta\cdot\varphi(f_{1})=\eta\cdot(c\cdot f_{2})=c\cdot \lambda_{f_{2}}(\eta)=\lambda_{f_{2}}(\eta)\cdot c\cdot f_{2}=\lambda_{f_{2}}(\eta)\cdot\varphi(f_{1})$$
so that $\lambda_{f_{1}}=\lambda_{f_{2}}$, as claimed.
Therefore, the eigenvalue $\lambda_{f}(\eta)$ can be computed in the usual model of the principal series at an archimedean place, as $$\eta\cdot\varphi_{s}^{0}(1)=\lambda_{f}(\eta)$$ for $\varphi_{s}^{0}$ the normalized spherical vector for $s\in \mathfrak{a}^{*}\otimes_{\mathbb{R}}\mathbb{C}$, and $\varphi^{0}(1)=1$. Thus,

$$\lambda_{f}(\eta)\;=\;(\eta\cdot\varphi_{s}^{0}(1))\;=\;\int\limits_{G_{\mathbb{R}}}\eta(g)\cdot\varphi_{s}^{0}(g)\;dg\;=\;\int\limits_{B_{\epsilon}}\varphi_{s}^{0}(g)\;dg$$
Let $P^{+}$ be the connected component of the identity in the standard minimal parabolic.  The Jacobian of the map $P^{+}\times K\rightarrow G_{\mathbb{R}}$ is non-vanishing at $1$, and $\varphi^{0}(1)=1$, so a suitable bound of $\epsilon$ on the spectral parameter $s\in\mathfrak{a}^{*}\otimes_{\mathbb{R}}\mathbb{C}$ will keep $\varphi_{s}^{0}(g)$ near $1$ on $B_{\epsilon}$.  In the example of $GL_{n}(\mathbb{R})$ with $\varphi_{s}^{0}$ the usual spherical vector, bounds of the form $|s_{j}|\ll\frac{1}{\epsilon}$ assure that $\text{Re} \;\varphi_{s}^{0}(g)\geq\frac{1}{2}$ on $B_{\epsilon}$, which prevents cancellation in the real part of $\varphi_{s}^{0}(g)$ for $g\in B_{\epsilon}$, so

$$\left|\lambda_{f}(\eta)\right|\;=\;\left|\int\limits_{B_{\epsilon}}\varphi_{s}^{0}(g)\;dg\right|\;\gg\;\int\limits_{B_{\epsilon}}\text{Re}\varphi_{s}^{0}(g)\;dg\;\gg\;\int\limits_{B_{\epsilon}}\frac{1}{2}\;dg\;\gg\;\epsilon^{\text{dim}X}$$
Combining the upper bound on $|q|_{L^{2}}^{2}$ with its lower bound on eigenvalues $t_{F}(t_{F}-1)$, letting $T=\frac{1}{\epsilon}$,

$$(\epsilon^{\text{dim}X})^{2}\times\left(\sum_{\text{cfm F}\:|t_{F}|\leq T}|F_{\Theta_{k}\backslash \Theta_{\mathbb{A}}x}|^{2}+\ldots\right)\;\ll\;\epsilon^{\text{dim}X+\text{dim}Y}$$
so
$$\sum_{\text{cfm F}:|t_{F}|\ll T}|F_{\Theta_{k}\backslash \Theta_{\mathbb{A}}x}|^{2}+\ldots\;\ll\;T^{\text{dim} X-\text{dim} Y}$$
Similarly,
$$\sum_{\text{cfm F}:|t_{F}|\ll T}|\langle \eta\cdot F,\phi\rangle|^{2}+\ldots\;\ll\;T^{\text{dim}X-\text{dim}Y}$$
\end{proof}

\section{Casimir Eigenvalue}
Let $G=SL_{4}(\mathbb{R})$ and $I(s_{1},s_{2},s_{3},s_{4})$ a minimal-parabolic principal series. Let $$\mathfrak{g}=\mathfrak{sl}_{4}$$ be the Lie algebra of $G$.  For $i\neq j$, let $E_{i,j}$ be the matrix with $1$ in the $(i,j)$-th position and $0$ elsewhere.  Let $H_{i,j}$ be the matrix with $1$ in the $(i,i)$-th position and $-1$ in the $(j,j)$-th position.  Observe that $H_{i,i+1}$ span the Cartan subalgebra $\mathfrak{h}$ and the $E_{i,j}$ for $i\neq j$ span the rest of the Lie algebra.  Assume without loss of generality that $i<j$.  We have the bracket relations
$$[E_{i,j},E_{j,i}]=H_{i,j}$$
As before, the Casimir element is given by
$$\Omega\;=\;\frac{1}{2}H_{1,2}^{2}+\frac{1}{2}H_{2,3}^{2}+\frac{1}{2}H_{3,4}^{2}+(\sum_{j,i}E_{i,j}E_{i,j}+E_{i,j}E_{j,i})$$
Rearranging, this gives
$$\Omega\;=\;\frac{1}{2}H_{1,2}^{2}+\frac{1}{2}H_{2,3}^{2}+\frac{1}{2}H_{3,4}^{2}+(\sum_{i,j}2E_{j,i}E_{i,j}+H_{i,j})$$
The lie algebra $\mathfrak{g}$ acts on $C^{\infty}(G)$ by $$X\cdot f(g)=\frac{d}{dt}|_{t=0}f(ge^{tX})$$ The product $E_{j,i}E_{i,j}$ act by $0$, so Casimir is simply
$$\Omega\;=\;(\frac{1}{2}H_{1,2}^{2}-H_{1,2})+(\frac{1}{2}H_{2,3}^{2}-H_{2,3})+(\frac{1}{2}H_{3,4}^{2}-H_{3,4})+H_{1,4}+H_{1,3}+H_{2,4}$$
\begin{prop}The Casimir operator acts on $I(s_{1},s_{2},s_{3},s_{4})$ by the scalar
\begin{align*}&\frac{1}{2}(s_{1}-s_{2})^{2}-(s_{1}-s_{2})+\frac{1}{4}(s_{1}+s_{2}-s_{3}-s_{4})^{2}-(s_{2}-s_{3})+\frac{1}{2}(s_{3}-s_{4})^{2}-(s_{3}-s_{4})\\&-(s_{1}-s_{4})-(s_{1}-s_{3})-(s_{2}-s_{4})\end{align*}
\end{prop}
\begin{proof}Let us see how $H_{1,2}$ acts on $I(s_{1},s_{2},s_{3},s_{4})$.  Note that 
$$e^{tH_{1,2}}\;=\;\left( \begin{array}{cccc}
e^{t} & 0&0&0 \\
0& e^{-t}&0&0\\
0 & 0 & 1&0\\
0&0&0&1\end{array} \right)$$
Therefore
\begin{align*}&\frac{d}{dt}|_{t=0}f(\left( \begin{array}{cccc}
e^{t} & 0&0&0 \\
0& e^{-t}&0&0\\
0 & 0 & 1&0\\
0&0&0&1\end{array} \right))\;=\;\frac{d}{dt}|_{t=0}\chi(\left( \begin{array}{cccc}
e^{t} & 0&0&0 \\
0& e^{-t}&0&0\\
0 & 0 & 1&0\\
0&0&0&1\end{array} \right))\\&=\;\frac{d}{dt}|_{t=0}e^{ts_{1}}\cdot e^{-ts_{2}}\\
\end{align*}
This is just $(s_{1}-s_{2})$.  Likewise, we see that $H_{i,j}$ will act on $I_{s}$ by $s_{i}-s_{j}$.  Therefore, the Casimir operator will act by
\begin{align*}&\frac{1}{2}(s_{1}-s_{2})^{2}-(s_{1}-s_{2})+\frac{1}{2}(s_{2}-s_{3})^{2}-(s_{2}-s_{3})+\frac{1}{2}(s_{3}-s_{4})^{2}-(s_{3}-s_{4})+(s_{1}-s_{4})\\&+(s_{1}-s_{3})+(s_{2}-s_{4})\end{align*}

Let $G=GL_{4}$ and $I(s_{1},s_{2},s_{3},s_{4})$ a minimal-parabolic principal series. Let $$\mathfrak{g}=\mathfrak{gl}_{4}$$ be the Lie algebra of $G$.  For $i\neq j$, let $E_{ij}$ be the matrix with $1$ in the $(i,j)$-th position and $0$ elsewhere.  Let $H_{ij}$ be the matrix with $1$ in the $(i,i)$-th position and $-1$ in the $(j,j)$-th position and let $H_{1234}=\text{diag}(1,1,-1,-1)$.  Observe that $H_{i,i+1}$ span the Cartan subalgebra $\mathfrak{h}$ and the $E_{ij}$ for $i\neq j$ span the rest of the Lie algebra.  Assume without loss of generality that $i<j$.  We have the bracket relations
$$[E_{ij},E_{ji}]=H_{ij}$$
As before, the Casimir element is given by
$$\Omega\;=\;\frac{1}{2}H_{12}^{2}+\frac{1}{4}H_{1234}^{2}+\frac{1}{2}H_{34}^{2}+(\sum_{ji}E_{ij}E_{ji}+E_{ji}E_{ij})$$
Rearranging, this gives
$$\Omega\;=\;\frac{1}{2}H_{12}^{2}+\frac{1}{4}H_{1234}^{2}+\frac{1}{2}H_{34}^{2}+(\sum_{ij}2E_{ij}E_{ji}-H_{ij})$$
The Lie algebra $\mathfrak{g}$ acts on $C^{\infty}(G)$ by $$X\cdot f(g)=\frac{d}{dt}|_{t=0}f(ge^{tX})$$ The product $E_{ij}E_{ji}$ act by $0$, so Casimir is simply
$$\Omega\;=\;(\frac{1}{2}H_{12}^{2}-H_{1,2})+(\frac{1}{4}H_{1234}^{2}-H_{23})+(\frac{1}{2}H_{34}^{2}-H_{34})-H_{14}-H_{13}-H_{24}$$
As an example computation, let us see how $H_{12}$ acts on $I(s_{1},s_{2},s_{3},s_{4})$.  Note that 
$$e^{tH_{12}}\;=\;\left( \begin{array}{cccc}
e^{t} & 0&0&0 \\
0& e^{-t}&0&0\\
0 & 0 & 1&0\\
0&0&0&1\end{array} \right)$$
Therefore
\begin{align*}&\frac{d}{dt}|_{t=0}f(\left( \begin{array}{cccc}
e^{t} & 0&0&0 \\
0& e^{-t}&0&0\\
0 & 0 & 1&0\\
0&0&0&1\end{array} \right))\;=\;\frac{d}{dt}|_{t=0}\chi(\left( \begin{array}{cccc}
e^{t} & 0&0&0 \\
0& e^{-t}&0&0\\
0 & 0 & 1&0\\
0&0&0&1\end{array} \right))\\
&=\;\frac{d}{dt}|_{t=0}e^{ts_{1}}\cdot e^{-ts_{2}}\\
\end{align*}
This is just $(s_{1}-s_{2})$.  Likewise, we see that $H_{ij}$ will act on $I_{s}$ by $s_{i}-s_{j}$.  Therefore, the Casimir operator will act by
\begin{align*}&\frac{1}{2}(s_{1}-s_{2})^{2}-(s_{1}-s_{2})+\frac{1}{4}(s_{1}+s_{2}-s_{3}-s_{4})^{2}-(s_{2}-s_{3})+\frac{1}{2}(s_{3}-s_{4})^{2}-(s_{3}-s_{4})\\&-(s_{1}-s_{4})-(s_{1}-s_{3})-(s_{2}-s_{4})\end{align*}
\end{proof}
Letting $s_{1}=s+s_{f}$, $s_{2}=-s+s_{f}$, $s_{3}=s-s_{f}$, $s_{4}=-s-s_{f}$, we see that $(s_{1}-s_{2})=2s$, $(s_{2}-s_{3})=-2s+2s_{f}$, $(s_{3}-s_{4})=2s$, $(s_{1}-s_{4})=2s+2s_{f}$, $(s_{1}-s_{3})=2s_{f}$, $(s_{2}-s_{4})=2s_{f}$, and finally $(s_{1}+s_{2}-s_{3}-s_{4})=4s_{f}$.  Putting all this into the above expression for Casimir's action gives that Casimir acts by

$$\lambda_{s,f}\;=\;4s^{2}+4s_{f}^{2}-8s_{f}-4s$$
Observe that
$$\lambda_{s,f}-\lambda_{w,f}\;=\;4(s(s-1)-w(w-1))$$

\section{Friedrichs self-adjoint extensions and complex conjugation maps}
\noindent We review the result due to Friedrichs that a densely-defined, symmetric, semi-bounded operator admits a canonical self-adjoint extension with a useful characterization. We follow [Grubb], [Garrett 2011c], [Friedrichs 1935a] and [Friedrichs 1935b].\\

\noindent Let $T$ be a densely defined, symmetric, unbounded operator on a Hilbert space $V$, with domain $D$. Assume further, that $T$ is semi-bounded from below in the sense that $$||u||^{2}\;\;\leq\;\;\langle u,Tu\rangle\;\; \text{for all}\; u\in D.$$
Let $\langle x,y \rangle_{1}=\langle Tx,y \rangle$ on D. Let $V_1$ be the completion of $D$ with respect to the new inner product.  The operator $T$ remains symmetric for $\langle , \rangle_{1}$. That is, $$\langle Tx,y\rangle_{1} = \langle x,Ty\rangle_{1}$$ for $x,y\in D$. By Riesz-Fischer, for $y\in V$, the continuous linear functional $$f(x)=\langle x,y \rangle$$ can be written $$f(x)=\langle x,y{'} \rangle_{1}$$ for a unique $y{'}\in V$.  Set $$T_{\text{Fr}}^{-1}y=y'$$ That is, the inverse $T_{\text{Fr}}^{-1}$ of the Friedrichs extension $T_{\text{Fr}}$ of $T$ is an everywhere-defined map $$T_{\text{Fr}}^{-1}:V\rightarrow V_1$$ continuous for the $\langle , \rangle_{1}$ topology on $V_1$, characterized by\\
\noindent $$\langle Tx, T_{\text{Fr}}^{-1}y\rangle=\langle x,y \rangle$$
We will prove that, given $\theta\in V_{-1}$ and $T_{\theta}=T|_{\text{ker}\theta}$, the Friedrichs extension $\tilde{T}_{\theta}$ has the feature that 
$$\tilde{T}_{\theta}u\;=\;f\;\;\;\;\text{for}\; u\in V_{1}, f\in V$$
exactly when
$$T_{\theta} u\;=\;f+c\cdot\theta\;\;\;\;\text{for some}\;c\in\mathbb{C}$$

Define a \textit{conjugation map} on $V$ to be a complex-conjugate-linear automorphism $j:V\rightarrow V$ with $\langle jx,jy\rangle=\langle y,x\rangle$ and $j^2=1$.
A conjugation map is equivalent to a complex-linear isomorphism $$\Lambda:V\rightarrow V^{*}$$ of $V$ with its complex-linear dual, via Riesz-Fischer, by
$$\Lambda(y)(x)\;=\;\langle x,jy\rangle\;=\;\overline{ \langle y,jx\rangle}$$
Assume $j$ stabilizes $D$ and that $T(jx)=jTx$ for $x\in D$.  Then $j$ respects $\langle, \rangle_{1}$:
$$\langle jx,jy\rangle_{1}\;=\;\langle y,Tx\rangle \;=\;\langle y,x\rangle_{1}$$
for $x,y\in D$. Also, $j$ commutes with $T_{\text{Fr}}$:
$$\langle x,T_{\text{Fr}}^{-1}jy\rangle_{1}\;=\;\langle x,jy \rangle\;=\;\langle y,jx\rangle\;=\;\langle T_{\text{Fr}}^{-1}y,jx\rangle_{1}\;=\;\langle x,jT_{\text{Fr}}^{-1}y\rangle_{1}$$
for $x\in V_{1}$ and $y\in V$. Let $V_{-1}$ be the complex-linear dual of $V_1$. We have $V_1\subset V\subset V_{-1}$.  By design, $$T:D\rightarrow V\subset V_{-1}$$ is continuous when $V$ has the subspace topology from $V_{-1}$:
$$|Ty|_{-1}\;=\;\text{sup}_{|x|_{1}\leq 1}|\Lambda(Ty)(x)|\;=\;\text{sup}|\langle x,jTy\rangle|=|\langle x,Tjy\rangle|\leq\;\text{sup}|x_1|\cdot|y_1|\;=\;|y|_{1}$$
by Cauchy-Schwarz-Bunyakowsky.  Thus the map $T:D\rightarrow V$ extends by continuity to an everywhere-defined, continuous map $$T^{\#}:V_{1}\rightarrow V_{-1}$$ by
$$(T^{\#}y)(x) = \langle x,jy \rangle_{1}$$
Further, $T^{\#}:V_{1}\rightarrow V_{-1}$ agrees with $T_{\text{Fr}}:D_{1}\rightarrow V$ on the domain $D_{1}=BV$ of $T_{\text{Fr}}$, since 
$$(T^{\#}y)(x)\;=\;\langle x,jy \rangle_{1}\;=\langle Tx,jy \rangle\;=\;\langle Tx,T_{\text{Fr}}^{-1}T_{\text{Fr}}jy\rangle\;=\;\langle T_{\text{Fr}}^{-1}Tx,T_{\text{Fr}}jy \rangle$$
which is 
$$=\;\langle x,T_{\text{Fr}}jy\rangle\;=\;\Lambda(T_{\text{Fr}}y)(x)\;\;\;\text{for}\; x\in D\;\; \text{and}\; y\in D_1$$
This follows since $T_{\text{Fr}}$ extends $T$, and noting the density of $D$ in $V$.\\

The following were presented as heuristics in [CdV 1982/1983] and treated more formally by Garrett in [Garrett 2011a].  We give complete proofs.

\begin{thm}\; The domain of $T_{\text{Fr}}$ is $D_{1}\;=\;\{u\in V_{1}\;:\;T^{\#}u\in V\}$. \end{thm}

\begin{proof}$T^{\#}u\;=\;f\in V$ implies that
$$\langle x,ju\rangle_{1}\;=\;(T^{\#}u)(x)\;=\;\Lambda(T^{\#}u)(x)\;=\;\Lambda(f)(x)\;=\;\langle x,jf\rangle\;\;\text{for all}\; x\in V_1$$
\noindent By the characterization of the Friedrichs extension, $T_{\text{Fr}}(ju) = jf$.  Since $T_{\text{Fr}}$ commutes with $j$, we have $T_{\text{Fr}}u\;=\;f$.\end{proof}

Extend the complex conjugation $j$ to $V_{-1}$ by $(j\lambda)(x)\;=\;\overline{\lambda(jx)}$ for \;$x\in V_1$, and write
$$\langle x,\theta\rangle_{V_1\times V_{-1}}\;=\;(j\theta)(x)\;=\;\overline{\theta(jx)}\;\;(\text{for}\; x\in V_1\;\text{and}\;\theta\in V_{-1})$$
For $\theta\in V_{-1}$,$$\theta^{\perp}\;=\;\{x\in V_1:\;\langle x,\theta\rangle_{V_{1}\times V_{-1}}=0\}$$
is a closed co-dimension-one subspace of $V_1$ in the $\langle,\rangle_{1}$-topology. Assume $\theta\notin V$.  This implies density of $\theta^{\perp}$ in $V$ in the $\langle,\rangle$-topology.

\begin{thm}\; The Friedrichs extension $T_{\theta}\;=\;(T|_{\theta^{\perp}})_{\text{Fr}}$ of the restriction $T|_{\theta^{\perp}}$ of $T$ to $D\cap\theta^{\perp}$ has the property that $T_{\theta}u\;=\;f$ for $u\in V_1$ and $f\in V$ exactly when $$T^{\#}u\;=\;f+c\theta$$ for some $c\in\mathbb{C}$. Letting $D_1$ be the domain of $T_{\text{Fr}}$, the domain of $T_{\theta}$ is$$\text{domain}\;T_{\theta}\;=\;\{x\in V_{1}:\langle x,\theta\rangle_{V_{1}\times V_{-1}}\;=\;0,\;T^{\#}x\in V+\mathbb{C}\cdot\theta\}$$
\end{thm}

\begin{proof} $T^{\#}u=f+c\cdot\theta$ is equivalent to $$\langle x,ju\rangle_{1}\;=\;T^{\#}(u)(x)\;=\;(f+c\cdot\theta)(x)\;=\;\langle x,jf\rangle\;\;(\text{for all}\;x\in \theta^{\perp}).$$

\noindent This gives $\langle x,ju\rangle_{1}\;=\;\langle x,jf\rangle$.  The topology on $\theta^{\perp}$ is the restriction of the $\langle,\rangle_{1}$-topology of $V_1$, while $\theta^{\perp}$ is dense in $V$ in the $\langle,\rangle$-topology.  Thus, $ju\;=\;T_{\theta}^{-1}jf$ by the characterization of the Friedrichs extension of $T_{\theta^{\perp}}$.  Then $u\;=\;T_{\theta}^{-1}f$, since $j$ commutes with $T$.\end{proof}

Given an  everywhere-defined map $\tilde{T}^{-1}:V\rightarrow V_{1}$, characterized by 
$$\langle Tx,\tilde{T}^{-1}y\rangle=\langle x,y\rangle\;\;\;\;(\text{for}\;x\in D,y\in V)$$
we review the proof that given $\theta\in V_{-1}$ and $T_{\theta}=T|_{\text{ker}\theta}$, the Friedrichs extension $\tilde{T}_{\theta}$ has the feature that 
$$\tilde{T}_{\theta}u\;=\;f\;\;\;\;\text{for}\; u\in V_{1}, f\in V$$
exactly when
$$T_{\theta} u\;=\;f+c\cdot\theta\;\;\;\;\text{for some}\;c\in\mathbb{C}$$
Observe that $T_{\theta}u=f+c\cdot\theta$ is equivalent to
$$\langle x,u\rangle_{1}=\langle x,Tu\rangle=\langle x,f+c\cdot\theta\rangle_{V_{1}\times V_{-1}}=\langle x,f\rangle_{V_{1}\times V_{-1}}\iff\tilde{T}_{\theta}u=f$$
where the second equality follows from restricting in the first argument and extending in the second.

\noindent\section{Moment bounds assumptions}
We will need to \textit{assume} a moment bound to know that the projected distribution is in the desired Sobolev space.  This assumption is far weaker than Lindelof, but highly non-trivial. 
\begin{prop}For a degree $n$ $L$-function $L(s)$ with suitable analytic continuation and functional equation, a second-moment bound
$$\int\limits_{0}^{T} |L(\frac{1}{2}+it)|^{2}\;dt\ll T^{A}$$
implies a pointwise bound 
$$L(\sigma_{o}+it,f)\ll_{\sigma_{o},\epsilon} (1+|t|)^{\frac{A}{2}+\epsilon}\;\;\;\;(\text{for every}\;\epsilon>0)$$
\end{prop}
\begin{proof}The proof of this is a standard argument, as follows. Fix $\sigma_{o}>\frac{1}{2}$.  For $0<t_{o}\in\mathbb{R}$, let $s_o=\sigma_{o}+it_{o}$.  Let $R$ be a rectangle in $\mathbb{C}$ with vertices $\frac{1}{2}\pm iT$ and $2\pm iT$ for $T> t_{o}$.  By Cauchy's Theorem
$$L(s_{o},f)^{2}\;=\;\frac{1}{2\pi i}\int\limits_{R}\frac{e^{(s-s_{o})^{2}}}{s-s_{o}}\cdot L(s,f)^{2}\;ds$$
Since the L-function has polynomial vertical growth, we can push the top and bottom of $R$ to $\infty$, giving

$$L(s_{o})^{2}\;=\;\frac{1}{2\pi}\int\limits_{-\infty}^{\infty}\frac{e^{((\frac{1}{2}-\sigma_{o})+i(t-t_{o}))^{2}}}{(\frac{1}{2}-\sigma_{o})+i(t-t_{o})}\cdot L(\frac{1}{2}+it)^{2}\;dt + O(1)$$
The part of the integral where $|t-t_{o}|\geq t_{o}$ is visibly $\ll_{n,\sigma_{o}} e^{-t_{o}}$:
$$|e^{((\frac{1}{2}-\sigma_{o}+i(t-t_o))^{2})}|\;=\;e^{(\frac{1}{2}-\sigma_{o})^{2}-(t-t_{o})^{2}}\ll_{\sigma_{o}} e^\frac{-t_{o}^{2}}{2}\cdot e^{\frac{-(t-t_{o})^{2}}{2}}\ll e^{-t_{o}}$$
for $|t-t_{o}|\geq t_{o}$.
Squaring the convexity bound for $L(\frac{1}{2}+it)$ gives
$$|L(\frac{1}{2}+it)|^{2}\ll |t|^{\frac{n}{2}+\epsilon}\;\;\;\;\;(\text{for all}\; \epsilon>0)$$
Thus
$$\int\limits_{2t_{o}}^{\infty}\frac{e^{((\frac{1}{2}-\sigma_{o}+i(t-t_o))^{2})}}{(\frac{1}{2}-\sigma_{o})+i(t-t_{o})}\cdot L(\frac{1}{2}+it)^{2}\;dt\ll_{\sigma_{o}} e^{\frac{-t_{o}^{2}}{2}}\int\limits_{2t_{o}}^{\infty}e^{\frac{-(t-t_{o})^{2}}{2}}\cdot t^{\frac{n}{2}+\epsilon}\ll_{\epsilon} e^{-t_{o}}$$
The other half of the tail, where $t<0$, is estimated similarly.  For $0<t<2t_{o}$, use the assumed moment estimate and the trivial estimate
$$\frac{e^{((\frac{1}{2}-\sigma_{o}+i(t-t_o))^{2})}}{(\frac{1}{2}-\sigma_{o})+i(t-t_{o})}\ll_{\sigma_{o}}e^{(\frac{1}{2}-\sigma_{o})^{2}-(t-t_{o})^{2}}\ll_{\sigma_{o}} 1$$
Then $$\int\limits_{0}^{2t_{o}}\frac{e^{((\frac{1}{2}-\sigma_{o}+i(t-t_o))^{2})}}{(\frac{1}{2}-\sigma_{o})+i(t-t_{o})}\cdot L(\frac{1}{2}+it)^{2}\;dt\ll_{\sigma_{o}}\int\limits_{0}^{2t_{o}}|L(\frac{1}{2}+it)|^{2}\;dt\ll t_{o}^{A}$$
Thus,
$$L(s_{o})^{2}\;=\;\frac{1}{2\pi}\int\limits_{-\infty}^{\infty}\frac{e^{(\frac{1}{2}+it-s_{o})^{2}}}{\frac{1}{2}+it-s_{o}}\cdot L(\frac{1}{2}+it,f)^{2}\;dt + O(1)\ll_{n,\sigma_{o}}t_{o}^{A}$$
Then a standard convexity argument [Lang, p.263] gives the asserted $|t_{o}|^{\frac{A}{2}+\epsilon}$ on $\sigma_{o}=\frac{1}{2}$ for all $\epsilon>0$.
\end{proof}

\noindent\section{Local automorphic Sobolev spaces}

A notion of local automorphic Sobolev spaces $H^{s}_{\text{lafc}}$ defined in terms of global automorphic Sobolev spaces $H_{\text{gafc}}^{s}$ is necessary to discuss the meromorphic continuation of solutions $u=u_{w}$ to differential equations $(\Delta-\lambda_{w})u=\theta$ for compactly-supported automorphic distributions $\theta$.  We want a continuous embedding of global automorphic Sobolev spaces into local automorphic Sobolev spaces.  This will follow immediately from the description, below.  Second,  compactly-supported distributions $\theta\in H^{-s}_{\text{gafc}}$ should extend to continuous linear functionals in $H^{-s}_{\text{lafc}}$.  A convenient corollary is that such $\theta$ moves inside integrals appearing in a spectral decomposition/synthesis of automorphic forms lying in global automorphic Sobolev spaces.  Finally, we want automorphic test functions to be dense in the local automorphic Sobolev spaces.

The necessity of the introduction of larger spaces than global automorphic Sobolev spaces is apparent already in the simplest situations.  On $\Gamma\backslash\mathfrak{H}$, with $\Gamma=SL_{2}(\mathbb{Z})$, when $\theta\in H_{\text{gafc}}^{-1-\epsilon}$ is an automorphic Dirac $\delta^{\text{afc}}$ at $z_{0}\in\Gamma\backslash\mathfrak{H}$, the spectral expansion in $\text{Re}(w)>\frac{1}{2}$ for a solution $u_{w}$ to that differential equation yields $u_{w}\in H_{\text{gafc}}^{1-\epsilon}$, but the meromorphic continuation to $\text{Re}(w)=\frac{1}{2}$ and then to $\text{Re}(w)<\frac{1}{2}$ includes an Eisenstein series $E_{w}$ which lies in no global automorphic Sobolev space.  That $E_{w}$ lies in local automorphic Sobolev space $H^{\infty}_{\text{lafc}}$ is immediate from the smoothness of $E_{w}$ and the definition of the local spaces, below.

We describe local automorphic Sobolev spaces.  Given a global automorphic Sobolev norm $|.|_{s}$, the corresponding local automorphic Sobolev norms, indexed by automorphic test functions $\varphi$, are given by
$$f\rightarrow |f|_{s,\varphi}\;=\;|\varphi\cdot f|_{s}\;\;\;\;\;\text{for}\;f\;\text{smooth automorphic}$$
\begin{defn} The  $s$-th local automorphic Sobolev space is given by
$$H^{s}_{\text{lafc}}(X)\;=\text{quasi-completion of}\;C_{c}^{\infty}(X)\;\text{with respect to these semi-norms}$$ 
\end{defn}
 By definition, $C_{c}^{\infty}(X)$ is dense in $H^{s}_{\text{lafc}}(X)$. Continuity of the embedding of the global automorphic Sobolev spaces into the local uses integration by parts.  The Lie algebra $\mathfrak{g}$ admits a decomposition $\mathfrak{g}\;=\;\mathfrak{k}\oplus\mathfrak{s}$
where $\mathfrak{k}$ is the Lie algebra of the maximal compact subgroup $K$ and $\mathfrak{s}$ is the algebra of symmetric matrices.
Choose an orthonormal basis $\{x_{i}\}$ for $\mathfrak{s}$ with respect to the Killing form $\langle,\rangle$.  Define the gradient
$$\nabla\;=\;\sum_{i}X_{x_{i}}\otimes x_{i}$$
where $X_{x_{i}}$ is the differential operator given by $X_{x_{i}}f(g)=\frac{\partial}{\partial t}|_{t=0}f(g\cdot e^{tx_{i}})$.
Observe that in the universal enveloping algebra
$$\nabla f\cdot\nabla F\;=\;(\sum_{i}X_{x_{i}}f\otimes x_{i})\cdot(\sum_{j}X_{x_{j}}F\otimes x_{j})\;=\;\sum_{i}X_{x_{i}}f\cdot X_{x_{j}}F$$
where the product is the Killing form on $\mathfrak{s}$.
\begin{prop} For $f,F\in C_{c}^{\infty}(\Gamma\backslash G)$, we have the integration-by-parts formula
$$\int_{\Gamma\backslash G}(-\Delta f) F\;=\;\int_{\Gamma\backslash G}\nabla f\cdot\nabla F$$
\end{prop}
\begin{proof} Letting $X=\Gamma\backslash G$, consider the integral
$$\int_{X}\frac{\partial}{\partial t}f(g\cdot e^{tx_{i}})\frac{\partial}{\partial t}F(g\cdot e^{tx_{i}})\;dg$$
Let $u=\frac{\partial}{\partial t}f(g\cdot e^{tx_{i}})$ and $dv=\frac{\partial}{\partial t}F(ge^{tx_{i}})dg$.  Then $du=\frac{\partial^{2}}{\partial t_{1}\partial t_{2}}f(g\cdot e^{t_{1}x_{i}}\cdot e^{t_{2}x_{i}})$, while $v=F(g)$.  Then, using the compact support of $f$ and its derivatives, we get
$$\int_{X}\frac{\partial}{\partial t_{1}}f(g\cdot e^{t_{1}x_{i}})\frac{\partial}{\partial t_{2}}F(g\cdot e^{t_{2}x_{i}})\;dg\;=\;\int_{X}-\frac{\partial^{2}}{\partial t_{1}\partial t_{2}}f(g\cdot e^{t_{1}x_{i}}e^{t_{2}x_{i}}) F(g)\;dg$$
Taking limits as $t_{1}$ and $t_{2}$ approach $0$ gives the integration-by-parts formula
$$\int_{X} X_{x_{i}}f\cdot X_{x_{i}}F\;=\;\int_{X} (-X_{x_{i}}f)^{2}\cdot F$$
and
$$\int_{X} (-\Delta f)\cdot F\;=\;\int_{X} \nabla f\cdot\nabla F$$
\end{proof}
Now we can compare the local automorphic Sobolev $+1$-norm to the global automorphic Sobolev $+1$-norm as follows:
\begin{prop} Every local automorphic Sobolev $+1$-norm is dominated by the global automorphic Sobolev $+1$-norm.
\end{prop}
\begin{proof}
$$|f|_{H^{1}_{\text{loc}}}\;=\;|\varphi f|^{2}_{H^{1}}\;=\;\int_{X}(1-\Delta)(\varphi f)\overline{\varphi f}\;=\;\int_{X}\nabla(\varphi f)\cdot\nabla(\overline{\varphi f})+\int_{X}\varphi f\cdot\varphi f$$
This is
\begin{align*}&\int_{X}(f\nabla\varphi+\varphi\nabla f)\cdot (\overline{f}\overline{\nabla\varphi}+\overline{\varphi}\overline{\nabla f})+|\varphi f|^{2}_{L^{2}}\\&=\;\int_{X}f^{2}||\nabla\varphi||^{2}+\int_{X}( f\overline{\varphi}\overline{\nabla f}\cdot\nabla\varphi+ \varphi\overline{f}\nabla f\overline{\nabla\varphi})+|\varphi f|_{L^{2}}^{2}\end{align*}
The first and last summands are dominated by $(C_{1}+C_{2})|f|^{2}_{L^{2}}$ where $C_{1}=\text{sup}\|\varphi\|$ and $C_{2}=\text{sup}\|\nabla\varphi\|$.  For the middle term, we use Cauchy-Schwarz and a constant bigger than $2\cdot\|\varphi\|\cdot\|\nabla\varphi\|$

\begin{align*}&( f\overline{\varphi}\overline{\nabla f}\cdot\nabla\varphi +\varphi\overline{f}\nabla f\overline{\nabla\varphi})\leq \int_{X} 2\varphi |f|\|\nabla f\|\|\nabla\varphi\|\ll\int_{X}|f|\|\nabla f\|\\&\leq (\int_{X}|f|^{2})^{\frac{1}{2}}(\int_{X}\|\nabla f\|^{2})^{\frac{1}{2}}\\
&=|f|_{L^{2}}\cdot (\int_{M}-\Delta f\cdot f)^{\frac{1}{2}}\leq |f|_{L^{2}}\cdot (\int_{X}(1-\Delta)f\cdot f)^{\frac{1}{2}}\;=\;|f|_{L^{2}}\cdot |f|_{H^{1}}\leq |f|^{2}_{H^{1}}\end{align*}
That is, with an implied constant independent of $f$,
$$|\varphi f|_{H^{1}}\ll |f|_{H^{1}}$$
\end{proof}

\begin{prop} There is a continuous map 
$$H^{1}_{\text{gafc}}\rightarrow H^{1}_{\text{lafc}}$$
\end{prop}
\begin{proof} The previous result proves continuity of $H^{1}_{\text{gafc}}\rightarrow H^{1,\varphi}$ for every automorphic test function $\varphi$.  Since $H^{1}_{\text{lafc}}$ is the projective limit of the $H^{1,\varphi}$ over all automorphic test functions $\varphi$, the universal property of the projective limit guarantees that there must be a continuous map $H^{1}_{\text{gafc}}\rightarrow H^{1}_{\text{lafc}}$.
\end{proof}

\noindent\section{Main Theorem: Characterization and Sparsity of discrete spectrum}
Recall the construction of $2,2$ pseudo-Eisenstein series. Let $\phi\in C_{c}^{\infty}(\mathbb{R})$ and let $f$ be a spherical cuspform on $GL_{2}(k)\backslash GL_{2}(\mathbb{A})$ with trivial central character.  
Let $$\varphi( \left( \begin{array}{cc}
A & B\\
0& D\end{array} \right))=\phi(\Big|\frac{\text{det}A}{\text{det}D}\Big|^{2})\cdot f(A)\cdot \overline{f}(D)$$
extending by right $K$-invariance to be made spherical.
Define the $P^{2,2}$ pseudo-Eisenstein series by
$$\Psi_{\varphi}(g)\;=\;\sum_{\gamma\in P_{k}\backslash G_{k}}\varphi(\gamma g)$$
Given $g= \left( \begin{array}{cc}
A & b\\
0& D\end{array} \right)$, let $h=h(g)=|\frac{\text{det}A}{\text{det}D}|^{2}$ be the \textit{height} of $g$. The spectral decomposition for $\theta$ in a global automorphic Sobolev space $H^{-s}$ is

\begin{align*}&\widetilde{\theta}\;=\;\sum_{F_{1}\;\text{cfm}\;GL4}\langle \widetilde{\theta},F_{1}\rangle\cdot F_{1}\;+\;\frac{\langle \widetilde{\theta},1\rangle}{\langle 1,1\rangle}\;+\;\sum_{F_{2}\; \text{cfm}\; GL2}\langle \widetilde{\theta},\Upsilon_{F_{2}}\rangle\cdot \Upsilon_{F_{2}}\\&+\;\sum_{F_{3},F_{4}\;\text{cfm}\;GL2}\int_{\frac{1}{2}-i\infty}^{\frac{1}{2}+i\infty}\langle \widetilde{\theta},E_{F_{3},F_{4},s}^{2,2}\rangle\cdot E_{F_{3},F_{4},s}^{2,2}\;ds\\&
+\;\sum_{F_{5}\;\text{cfm}\;GL3}\int_{\frac{1}{2}-i\infty}^{\frac{1}{2}+i\infty}\langle \widetilde{\theta},E_{F_{5},s}^{3,1}\rangle\cdot E_{F_{5},s}^{3,1}\;ds+\;\sum_{F_{6}\;\text{cfm}\;GL2}\int_{\rho-i\infty}^{\rho+i\infty}\langle \widetilde{\theta},E_{F_{6},\lambda}^{2,1,1}\rangle\cdot E_{F_{6},\lambda}^{2,1,1}\;d\lambda\\&
+\;\int_{\rho+i\mathfrak{a}^{*}_{\text{min}}}\langle \widetilde{\theta},E_{\lambda}\rangle\cdot E_{\lambda}\;d\lambda\\
\end{align*}
where $F$ and $F^{'}$ are cuspforms on $GL(2)$ and the $\Upsilon_{F}$'s are Speh forms.
We are interested in the subspace $V$ of $L^{2}(Z_{\mathbb{A}}G_{k}\backslash G_{\mathbb{A}})$ spanned by $2,2$ pseudo-Eisenstein series with fixed cuspidal data $f$ and $\overline{f}$, where $f$ is everywhere locally spherical.   Let $D_{a,f}$ be the subspace of $V$ consisting of the $L^{2}$-closure of the span of $2,2$ pseudo-Eisenstein series  with fixed cuspidal datum $f$ and $\overline{f}$ with test function $\varphi$ supported on $h(g)<a$ and whose constant terms have support on $h(g)<a$.

Let $\Delta_{a}$ be $\Delta$ restricted to $D_{a,f}$, and let $\widetilde{\Delta}_{a}$ be the Friedrichs extension of $\Delta_{a}$ to a self-adjoint (unbounded) operator.  By construction, the domain of $\widetilde{\Delta}_{a}$ is contained in a Sobolev space $\Phi^{+1}_{a}$, defined as the completion of $D_{a,f}$ with respect to the $+1$-Sobolev norm $\langle f,f\rangle_{1}=\langle (1-\Delta)f,f\rangle_{L^{2}}$. We recall [M-W,141-143], and [Garrett 2014] the
\begin{thm} The inclusion $\Phi_{a}^{1}\rightarrow \Phi_{a}$, from $\Phi_{a}^{1}$ with its finer topology, is compact, so that the space $\Phi_{a}$ decomposes discretely.
\end{thm}

Indeed, let $L_{\eta}^{2}$ be the subspace of $L^{2}(\text{PGL}_{4}\backslash \text{PGL}_{4}(\mathbb{R})/\text{O}_{4}(\mathbb{R}))$ with all constant terms vanishing above given fixed heights, specified by a real-valued function $\eta$ on simple positive roots described below. By its construction, the resolvent of the Friedrichs extension maps continuously from $L^{2}$ to the automorphic Sobolev space $H^{1}=H^{1}(\text{PGL}_{4}(\mathbb{Z})\backslash\text{PGL}_{4}(\mathbb{R})/\text{O}_{4}(\mathbb{R}))$ with its finer topology. Letting $$H_{\eta}^{1}=H^{1}\cap L_{\eta}^{2}$$ with the topology of $H^{1}$, it suffices to show that the injection
$$H_{\eta}^{1}\rightarrow L_{\eta}^{1}$$
is compact. To prove this compactness, we show that the image of the unit ball of $H_{\eta}^{1}$ is totally bounded in $L_{\eta}^{2}$.

Let $A$ be the standard maximal torus consisting of diagonal elements of $GL_{4}$, $Z$ the center of $G$, and $K=O_{4}(\mathbb{R})$.  Let $A^{+}$ be the subgroup of $A_{\mathbb{R}}$ with positive diagonal entries, and let $Z^{+}=Z_{\mathbb{R}}\cap A^{+}$.  A standard choice of positive simple roots is
$$\Phi\;=\;\{\alpha_{i}(a)\;=\;\frac{a_{i}}{a_{i+1}}\;\:\;i=1,\dots,r-1\}$$
where $a$ is the matrix
$$a\;=\;\left( \begin{array}{cccc}
a_{1}& 0 &0&0\\
0&a_{2}&0& 0\\
0&0&a_{3}&0\\
0&0&0&a_{4}\end{array} \right)$$
Let $N^{\text{min}}$ be the unipotent radical of the standard minimal parabolic $P^{\text{min}}$ consisting of upper-triangular elements of $G$. For $g\in G_{\mathbb{R}}$, let $g=n_{g}a_{g}k_{g}$ be the corresponding Iwasawa decomposition with respect to $P^{\text{min}}$. From basic reduction theory, the quotient $Z_{\mathbb{R}}G_{\mathbb{Z}}\backslash G_{\mathbb{R}}$ is covered by the Siegel set
$$\mathfrak{S}\;=\;N_{\mathbb{Z}}^{\text{min}}\backslash N_{\mathbb{R}}^{\text{min}}\cdot Z^{+}\backslash A_{0}^{+}\cdot K\;=\;Z^{+}N_{\mathbb{Z}}^{\text{min}}\big\backslash\big\{g\in G:\alpha(a_{g})\geq\frac{\sqrt{3}}{2} ,\;\text{for all}\;\alpha\in\Phi\big\}$$
Further, there is an absolute constraint so that
$$\int_{\mathfrak{S}}|f|\ll\int_{Z_{\mathbb{R}}G_{\mathbb{Z}}\backslash G_{\mathbb{R}}}|f|$$
for all $f$.  For a non-negative real-valued function $\eta$ on the set of simple roots, let
$$X_{\eta}^{\alpha}\;=\;\{g\in\mathfrak{S}\;:\;\alpha(a_{g})\geq \eta(\alpha)\}$$
for $\alpha\in\Phi$. Let 
$$C_{\eta}\;=\;\{g\in\mathfrak{S}\;:\;\alpha(a_{g})\leq\eta(\alpha)\;\text{for all}\;\alpha\in\Phi\}$$
This is a compact set, and 
$$\mathfrak{S}\;=\;C_{\eta}\cup \bigcup_{\alpha\in\Phi}X_{\eta}^{\alpha}$$
For $\alpha\in\Phi$, let $P^{\alpha}$ be the standard maximal proper parabolic whose unipotent radical $N^{\alpha}$ has Lie algebra $\mathfrak{n}^{\alpha}$ including the $\alpha^{\text{th}}$ root space.  That is, for $\alpha(a)=\frac{a_{i}}{a_{i+1}}$, the Levi component $M^{\alpha}$ of $P^{\alpha}$ is $GL_{i}\times GL_{4-i}$. As before, let $(c_{P}f)(g)$ denote the constant term along a parabolic $P$ of a function $f$ on $G_{\mathbb{Z}}\backslash G_{\mathbb{R}}$.  For $P=P^{\alpha}$, write $c^{\alpha}=c_{P}$.  For a non-negative real-valued function $\eta$ on the set of simple roots, the space of square-integrable functions with constant terms vanishing above heights $\eta$ is
$$L_{\eta}^{2}\;=\;\{f\in L^{2}(Z_{\mathbb{R}}G_{\mathbb{Z}}\backslash G_{\mathbb{R}}/K)\;:\;c^{\alpha}f(g)\;=\;0\;\text{for}\;\alpha(a_{g})\geq \eta(\alpha),\;\text{for all}\;\alpha\in \Phi\}$$
Vanishing is meant in a distributional sense.  The global automorphic Sobolev space $H^{1}$ is the completion of $C_{c}^{\infty}(Z_{\mathbb{R}}G_{\mathbb{Z}}\backslash G_{\mathbb{R}})^{K}$ with respect to the $H^{1}$ Sobolev norm
$$|f|_{H^{1}}^{2}\;=\;\int_{Z_{\mathbb{R}}G_{\mathbb{Z}}\backslash G_{\mathbb{R}}}(1-\Delta)f\cdot\overline{f}$$
where $\Delta$ is the invariant Laplacian descended from the Casimir operator $\Omega$.  Put $H_{\eta}^{1}=H^{1}\cap L_{\eta}^{2}$.

\begin{prop} The Friedrichs self-adjoint extension $\widetilde{\Delta}_{\eta}$ of the restriction of the symmetric operator $\Delta$ to test functions in $L_{\eta}^{2}$ has compact resolvent, and thus has purely discrete spectrum
\end{prop}
\begin{proof} Let
$$A_{0}^{+}\;=\;\{a\in A\;:\;\alpha(a)\geq\frac{\sqrt{3}}{2}\;:\;\text{for all}\;\alpha\in\Phi\}$$
We grant ourselves that we can control smooth cut-off functions:
\begin{lem} Fix a positive simple roots $\alpha$.  Given $\mu\geq\eta(\alpha)+1$, there are smooth functions $\varphi_{\mu}^{\alpha}$ for $\alpha\in\Phi$ and $\varphi_{\mu}^{0}$ such that: all these functions are real-valued, taking values between $0$ and $1$, $\varphi^{0}$ is supported in $C_{\mu+1}$, and $\varphi^{\alpha}\mu$ is supported in $X_{\mu}^{\alpha}$, and $\varphi_{\mu}^{0}+\sum_{\alpha}\varphi_{\mu}^{\alpha}=1$.  Further, there is a bound $C$ uniform in $\mu\geq \eta(\alpha)+1$, such that $|f\cdot\varphi_{\mu}^{0}|_{H^{1}}\leq C\cdot |f|_{H^{1}}$, and
$$|f\cdot\varphi_{\mu}^{\alpha}|_{H^{1}}\leq C\cdot |f|_{H^{1}}$$
for all $\mu\geq \eta(\alpha)+1$.
\end{lem}
Then the key point is
\begin{cla} For $\alpha\in\Phi$,
$$\lim_{\mu\to\infty}\left(\text{sup}\frac{|f|_{L^{2}}}{|f|_{H^{1}}}\right)\;=\;0$$
where the supremum is taken over $f\in H_{\eta}^{1}$ and $\text{support}(f)\subset X_{\mu}^{\alpha}$.
\end{cla}
Temporarily grant the claim.  To prove total boundedness of $H_{\eta}^{1}\rightarrow L_{\eta}^{2}$, given $\epsilon>0$, take $\mu\geq \eta(\alpha)+1$ for all $\alpha\in\Phi$, large enough so that $f\cdot\varphi_{\mu}^{\alpha}|_{L^{2}}<\epsilon$, for all $f\in H_{\eta}^{1}$, with $|f|_{H^{1}}\leq 1$.  This covers the images $\{f\cdot \varphi_{\mu}^{\alpha}:f\in H_{\eta}^{1}\}$ with $\alpha\in\Phi$ with $\text{card}\Phi$ open balls in $L^{2}$ of radius $\epsilon$.  The remaining part $\{f\cdot\varphi_{\mu}^{0}:f\in H_{\eta}^{1}\}$ consists of smooth functions supported on the compact $C_{\mu}$.  The latter can be covered by finitely-many coordinate patches $\phi_{i}:U_{i}\rightarrow\mathbb{R}^{d}$.  Take smooth cut-off functions $\varphi$ for this covering.  The functions $(f\cdot\varphi_{i})\circ\phi_{i}^{-1}$ on $\mathbb{R}^{d}$ have support strictly inside a Euclidean box, whose opposite faces can be identified to form a flat d-torus $\mathbb{T}^{d}$.  The flat Laplacian and the Laplacian inherited from $G$ admit uniform comparison on each $\phi(U_{i})$ , so the $H^{1}(\mathbb{T}^{d})$-norm of $(f\cdot\varphi)\circ\phi_{i}^{-1}$ is uniformly bounded by the $H^{1}$-norm.  The classical Rellich lemma asserts compactness of
$$H^{1}(\mathbb{T}^{d})\rightarrow L^{2}(\mathbb{T}^{d})$$  By restriction, this gives the compactness of each $H^{1}\cdot\varphi_{i}\rightarrow L^{2}$.  A finite sum of compact maps is compact, so $H^{1}\cdot\varphi_{\mu}^{0}\rightarrow L^{2}$ is compact.  In particular, the image of the unit ball from $H^{1}$ admits a cover by finitely-many $\epsilon$-balls for any $\epsilon>0$.  Combining these finitely-many $\epsilon$-balls with the $\text{card}(\Phi)$ balls covers the image of $H_{\eta}^{1}$ in $L^{2}_{\eta}$ by finitely-many $\epsilon$-balls, proving that $H_{\eta}^{1}\rightarrow L^{2}$ is compact.

It remains to prove the claim.  Fix $\alpha=\alpha_{i}\in\Phi$, and $f\in H_{\eta}^{1}$ with support inside $X_{mu}^{\alpha}$ for $\mu\gg\eta(\alpha)$. Let $N=N^{\alpha}$, $P=P^{\alpha}$, and let $M=M^{\alpha}$ be the standard Levi component of $P$.  Use exponential coordinates
$$n_{x}\;=\;\left( \begin{array}{cc}
1_{i}& x\\
0&1_{4-i}\end{array} \right)$$
In effect, the coordinate $x$ is in the Lie algebra $\mathfrak{n}$ of $N_{\mathbb{R}}$.  Let $\Lambda\subset\mathfrak{n}$ be the lattice which exponentiates to $N_{\mathfrak{Z}}$.  Give $\eta$ the natural inner product $\langle,\rangle$ invariant under the (Adjoint) action of $M_{\mathbb{R}}\cap K$ that makes root spaces mutually orthogonal.  Fix a nontrivial character $\psi$ on $\mathbb{R}/\mathbb{Z}$.  We have the Fourier expansion
$$f(n_{x}m)\;=\;\sum_{\xi\in\Lambda'}\psi\langle x,\xi\rangle\hat{f}_{\xi}(m)$$
with $n\in N_{\mathbb{R}}$, $m\in M_{\mathbb{R}}$, and $\Lambda'$ is the dual lattice to $\Lambda$ in $\mathfrak{n}$ with respect to $\langle,\rangle$, and
$$\hat{f}_{\xi}(m)\;=\;\int_{\mathfrak{n}\backslash\Lambda}\hat{\psi}\langle x,\xi\rangle f(n_{x}m)\;dx$$
Let $\Delta^{\mathfrak{n}}$ be the flat Laplacian on $\mathfrak{n}$ associated to the inner product $\langle,\rangle$ normalized so that
$$\Delta^{\mathfrak{n}}\psi\langle x,\xi\rangle\;=\;-\langle \xi,\xi\rangle\cdot \psi\langle x,\xi\rangle$$
Let $U=M\cap N^{\text{min}}$.  Abbreviating $A_{u}=\text{Ad} u$,
$$|f|_{L^{2}}^{2}\leq \int_{\mathfrak{S}}|f|^{2}\;=\;\int_{Z^{+}\backslash A_{0}^{+}\int_{U_{\mathbb{Z}}\backslash U_{\mathbb{R}}}}\int_{A_{u}^{-1}\Lambda\backslash\mathfrak{n}}|f(un_{x}a)|^{2}dx\;du\;\frac{da}{\delta(a)}$$
with Haar measures $dx$, $du$, $da$, and where $\delta$ is the modular function of $P_{\mathbb{R}}$.  Using the Fourier expansion,
\begin{align*}&f(un_{x}a)\;=\;f(un_{x}u^{-1}\cdot u a)\;=\;\sum_{\xi\in\lambda'}\psi\langle A_{u}x,\xi\rangle\cdot\hat{f}_{\xi}(ua)\\
&=\;\sum_{\xi\in\Lambda'}\psi\langle x,A_{u}^{*}\xi\rangle\cdot\hat{f}_{\xi}(ua)\end{align*}
Then
$$-\Delta^{\mathfrak{n}}f(un_{x}a)\;=\;\sum_{\xi\in\Lambda'}\langle A_{u}^{*}\xi,A_{u}^{*}\xi\rangle\cdot\psi\langle x,A_{u}^{*}\xi\rangle\cdot\hat{f}_{\xi}(ua)$$
The compact quotient $U_{\mathbb{Z}}\backslash U_{\mathbb{R}}$ has a compact set $R$ of representatives in $U_{\mathbb{R}}$, so there is a uniform lower bound for $0\neq \xi\in\Lambda'$:
$$0<b\leq\text{inf}_{u\in R}\text{inf}_{0\neq\xi\in\Lambda'}\langle A_{u}^{*}\xi,A_{u}^{*}\xi\rangle$$
By Plancherel applied to the Fourier expansion in $x$, using the hypothesis that $\hat{f}_{0}=0$ in $X_{\mu}^{\alpha}$,
$$\int_{A_{\mu}^{-1}\Lambda\backslash\mathfrak{n}}|f(un_{x}a)|^{2}\;dx\;=\;\int_{A_{u}^{-1}\Lambda\backslash\mathfrak{n}}|f(un_{x}u^{-1}\cdot ua)|^{2}\;dx\;=\;\sum_{\xi\in\Lambda'}|\hat{f}_{\xi}(ua)|^{2}$$
$$\leq b^{-1}\sum_{\xi\in\Lambda'}\langle A_{u}^{*}\xi,A_{u}^{*}\xi\rangle\cdot |\hat{f}_{\xi}(ua)|^{2}\;=\;\sum_{\xi\in\Lambda'}-\widehat{\Delta^{\mathfrak{n}}}f_{\xi}(ua)\cdot\overline{\hat{f}}(ua)$$
$$=\;\int_{u^{-1}\Lambda u\backslash\mathfrak{n}}-\Delta^{\mathfrak{n}}f(un_{x}u^{-1}\cdot ua)\cdot\overline{f}(un_{x}u^{-1}\cdot ua)\;dx\;=\;\int_{A_{u}^{-1}\Lambda\backslash\mathfrak{n}}-\Delta^{\mathfrak{n}}f(un_{x}a)\cdot\overline{f}(un_{x}a)\;dx$$
Thus, for $f$ with $\hat{f}(0)=0$ on $\alpha(g)\geq\eta$,
$$|f|_{L^{2}}^{2}\ll\int_{Z^{+}\backslash A_{0}^{+}}\int_{U_{\mathbb{Z}}\backslash U_{\mathbb{R}}}\int_{A_{u}^{-1}\Lambda\backslash\mathfrak{n}}-\Delta^{\mathfrak{n}}f(un_{x}a)\cdot\overline{f}(un_{x}a)\;dx\;du\;\frac{da}{\delta(a)}$$
Next, we compare $\Delta^{\mathfrak{n}}$ to the invariant Laplacian $\Delta$.  Let $\mathfrak{g}$ be the Lie algebra of $G_{\mathbb{R}}$, with non-degenerate invariant pairing 
$$\langle u,v\rangle\;=\;\text{trace}(uv)$$
The Cartan involution $v\rightarrow v^{\theta}$ has $+1$ eigenspace the Lie algebra $\mathfrak{k}$ of $K$, and $-1$ eigenspace $\mathfrak{s}$, the space of symmetric matrices.

Let $\Phi^{N}$ be the set of positive roots $\beta$ whose root space $\mathfrak{g}_{\beta}$ appears in $\mathfrak{n}$.  For each $\beta\in\Phi^{N}$, take $x_{\beta}\in\mathfrak{g}_{\beta}$ such that $x_{\beta}+x_{\beta}^{\theta}\in\mathfrak{s}$, $x_{\beta}-x_{\beta}^{\theta}\in\mathfrak{k}$, and $\langle x_{\beta},x_{\beta}^{\theta}=1$: for $\beta(a)=\frac{a_{i}}{a_{j}}$ with $i<j$, $x_{\beta}$ has a single non-zero entry, at the $ij^{\text{th}}$ place.  Let
$$\Omega'\;=\;\sum_{\beta\in\Phi^{N}}(x_{\beta}x_{\beta}^{\theta}+x_{\beta}^{\theta}x_{\beta})$$
Let $\Omega''\in U\mathfrak{g}$ be the Casimir element for the Lie algebra $\mathfrak{m}$ of $M_{\mathbb{R}}$, normalized so that Casimir for $\mathfrak{g}$ is the sum $\Omega=\Omega'+\Omega''$.  We rewrite $\Omega'$ to fit the Iwasawa coordinates: for each $\beta$,
$$x_{\beta}x_{\beta}^{\theta}+x_{\beta}^{\theta}x_{\beta}\;=\;2x_{\beta}x_{\beta}^{\theta}+[x_{\beta}^{\theta},x_{\beta}]\;=\;2x_{\beta}^{2}-2x_{\beta}(x_{\beta}-x_{\beta}^{\theta})+[x_{\beta}^{\theta},x_{\beta}]\in 2x_{\beta}^{2}+[x_{\beta}^{\theta},x_{\beta}]+\mathfrak{k}$$
Therefore,
$$\Omega'\;=\;\sum_{\beta\in\Phi^{N}}2x_{\beta}^{2}+[x_{\beta}^{\theta},x_{\beta}]\;\;\;\;\text{modulo}\;\mathfrak{k}$$
The commutators $[x_{\beta}^{\theta},x_{\beta}]\in\mathfrak{m}$.  In the coordinates $un_{x}a$ with $U\mathfrak{g}$ acting on the right, $x_{\beta}\in\mathfrak{n}$ is acted on by $a$ before translating $x$, by
$$un_{x}a\cdot e^{tx_{\beta}}\;=\;un_{x}\cdot e^{t\beta(a)\cdot x_{\beta}}\cdot a\;=\;un_{x+\beta(a)x_{\beta}}a$$
That is, $x_{\beta}$ acts by $\beta(a)\cdot\frac{\partial}{\partial x_{\beta}}$.

For two symmetric operators $S,T$ on a not-necessarily-complete inner product space $V$, write $S\leq T$ when 
$$\langle Sv,v\rangle\leq\langle Tv,v\rangle$$
for all $v\in V$.  We say that a symmetric operator $T$ is non-negative when $0\leq T$.  Since $a\in A_{0}^{+}$, there is an absolute constant so that $\alpha(a)\geq\mu$ implies $\beta(a)\gg\mu$.  Thus,
$$-\Delta^{\mathfrak{n}}\;=\;-\sum_{\beta\in\Phi^{N}}\frac{\partial^{2}}{\partial x_{\beta}^{2}}\ll \frac{1}{\mu^{2}}\cdot\left(-\sum_{\beta\in\Phi^{N}}x_{\beta}^{2}\right)$$
on $C_{c}^{\infty}(X_{\mu}^{\alpha})^{K}$ with the $L^{2}$ inner product.  We claim that
$$-\sum_{\beta\in\Phi^{N}}[x_{\beta}^{\theta},x_{\beta}]-\Omega''\geq 0$$
on $C_{c}^{\infty}(X_{\mu}^{\alpha})^{K}$. From this, it would follow that
$$-\Delta^{\mathfrak{n}}\ll\frac{1}{\mu^{2}}\cdot\left(-\sum_{\beta\in\Phi^{N}}x_{\beta}^{2}\right)\leq \frac{1}{\mu^{2}}\cdot\left(-\sum_{\beta\in\Phi^{N}}x_{\beta}^{2}-\sum_{\beta\in\Phi^{N}}[x_{\beta}^{\theta},x_{\beta}]-\Omega''\right)\;=\;\frac{1}{\mu^{2}}\cdot(-\Delta)$$
Then, for $f\in H_{\eta}^{1}$ with support in $X_{\mu}^{\alpha}$ we would have
$$|f|_{L^{2}}^{2}\ll\int_{\mathfrak{S}}-\Delta^{\mathfrak{n}}f\cdot\overline{f}\ll\frac{1}{\mu^{2}}\int_{\mathfrak{S}}-\Delta f\cdot\overline{f}\ll\frac{1}{\mu^{2}}\int_{Z_{\mathbb{R}}G_{\mathbb{Z}}\backslash G_{\mathbb{R}}}-\Delta f\cdot\overline{f}\ll\frac{1}{\mu^{2}}\cdot |f|_{H^{1}}^{2}$$
Taking $\mu$ large makes this small.  Since we can do the smooth cutting-off to affect the $H^{1}$ norm only up to a uniform constant, this would complete the proof of total boundedness of the image in $L^{2}$ of the unit ball from $H_{\eta}^{1}$.

To prove the claimed nonnegativity of $T=-\sum_{\beta\in\Phi^{N}}[x_{\beta}^{\theta},x_{\beta}]-\Omega''$, exploit the Fourier expansion along $N$ and the fact that $x\in\mathfrak{n}$ does not appear in $T$: noting that the order of coordinates $n_{x}u$ differs from that above,
\begin{align*}&\int_{Z^{+}\backslash A_{0}^{+}}\int_{U_{\mathbb{Z}}\backslash U_{\mathbb{R}}}\int_{\Lambda\backslash\mathfrak{n}}Tf(n_{x}ua)\overline{f}(n_{x}ua)\;dx\;du\;\frac{da}{\delta(a)}\\
&=\;\int_{Z^{+}\backslash A_{0}^{+}}\int_{U_{\mathbb{Z}}\backslash U_{\mathbb{R}}}\int_{\Lambda\backslash\mathfrak{n}}T\left(\sum_{\xi}\psi\langle x,\xi\rangle\hat{f}(ua)\right)\sum_{\xi'}\overline{\psi}\langle x,\xi'\rangle \overline{\hat{f}}(ua)\;dx\;du\;\frac{da}{\delta(a)}
\end{align*}
Only the diagonal summands survive the integration in $x\in\mathfrak{n}$, and the exponentials cancel, so this is
$$\int_{Z^{+}\backslash A_{0}^{+}}\int_{U_{\mathbb{Z}}\backslash U_{\mathbb{R}}}\sum_{\xi}T\hat{f}_{\xi}(ua)\cdot\overline{\hat{f}}(ua)\;du\;\frac{da}{\delta(a)}$$
Let $F_{\xi}$ be a left-$N_{\mathbb{R}}$-invariant function taking the same values as $\hat{f}_{\xi}$ on $U_{\mathbb{R}}A^{+}K$, defined by
$$F_{\xi}(n_{x}uak)\;=\;\hat{f}_{\xi}(uak)$$
for $n_{x}\in N$, $u\in U$, $a\in A^{+}$, $k\in K$. Since $T$ does not involve $\mathfrak{n}$ and since $F_{\xi}$ is left $N_{\mathbb{R}}$-invariant,
$$T\hat{f}_{\xi}(ua)\;=\;TF_{\xi}(n_{x}ua)\;=\;-\Delta F_{\xi}(n_{x}ua)$$
and then
$$\int_{Z^{+}\backslash A_{0}^{+}}\int_{U_{\mathbb{Z}}\backslash U_{\mathbb{R}}}\sum_{\xi}T\hat{f}(ua)\cdot\overline{\hat{f}}_{\xi}(ua)\;du\;\frac{da}{\delta(a)}\;=\;\int_{Z^{+}\backslash A_{0}^{+}}\int_{U_{\mathbb{Z}}\backslash U_{\mathbb{R}}}\sum_{\xi}-\Delta F_{\xi}(ua)\cdot\overline{F}_{\xi}(ua)\;du\;\frac{da}{\delta(a)}$$
The individual summands are not left-$U_{\mathbb{Z}}$-invariant.  Since $\hat{f}_{\xi}(\gamma g)=\hat{f}_{A_{\gamma}^{*}\xi}(g)$ for $\gamma$ normalizing $\mathfrak{n}$, we can group $\xi\in\Lambda'$ by $U_{\mathbb{Z}}$ orbits to obtain $U_{\mathbb{Z}}$ subsums and then unwind.  Pick a representative $\omega$ for each orbit $[\omega]$, and let $U_{\omega}$ be the isotropy subgroup of $\omega$ in $U_{\mathbb{Z}}$, so
$$\int_{U_{\mathbb{Z}}\backslash U_{\mathbb{R}}}\sum_{\xi}-\Delta F_{\xi}(ua)\cdot\overline{F}_{\xi}(ua)\;du\;=\;\sum_{[\omega]}\int_{U_{\mathbb{Z}}\backslash U_{\mathbb{R}}}\sum_{\xi\in[\omega]}-\Delta F_{\xi}(ua)\cdot\overline{F}_{\xi}(ua)\;du$$
$$=\;\sum_{[\omega]}\int_{U_{\mathbb{Z}}\backslash U_{\mathbb{R}}}\sum_{\gamma\in U_{\omega}\backslash U_{\mathbb{Z}}}-\Delta F_{A_{\gamma}^{*}\omega}(ua)\cdot\overline{F}_{A_{\gamma}^{*}\omega}(ua)\;du\;=\;\sum_{\omega}\int_{U_{\omega}\backslash U_{\mathbb{R}}}-\Delta F_{\omega}(ua)\cdot\overline{F}_{\omega}(ua)\;du$$
Then
$$\int_{Z^{+}\backslash A_{0}^{+}}\int_{U_{\mathbb{Z}}\backslash U_{\mathbb{R}}}\sum_{\xi}-\Delta F_{\xi}(ua)\cdot\overline{F}_{\xi}(ua)\;du\;=\;\sum_{\omega}\int_{Z^{+}\backslash A_{0}^{+}}\int_{U_{\omega}\backslash U_{\mathbb{R}}}-\Delta F_{\omega}(ua)\cdot\overline{F}_{\omega}(ua)\;du\;\frac{da}{\delta(a)}$$
Since $-\Delta$ is a non-negative operator on functions on every quotient $Z^{+}N_{\mathbb{R}}U_{\omega}\backslash G_{\mathbb{R}}/K$ of $G_{\mathbb{R}}/K$, each double integral is non-negative, proving that $T$ is non-negative.

This completes the proof that $H_{\eta}^{1}\rightarrow L_{\eta}^{2}$ is compact, and thus, that the Friedrichs extension of the restriction of $\Delta$ to test functions in $L_{\eta}^{2}$ has purely discrete spectrum.

\end{proof}

Since the pseudo-Eisenstein series appearing in the spectral decomposition are orthogonal to all other automorphic forms appearing in the spectral expansion in every Sobolev space, we can speak of the projection $\theta$ of the period distribution $\widetilde{\theta}$ to the subspace $V$ of $L^{2}(Z_{\mathbb{A}}G_{k}\backslash G_{\mathbb{A}})$. That is,
$$\theta\;=\;\langle \widetilde{\theta}, \Upsilon_{f}\rangle\cdot \Upsilon_{f}+\frac{1}{4\pi i}\int_{\frac{1}{2}-i\infty}^{\frac{1}{2}+i\infty}\langle \widetilde{\theta}, E_{f,\overline{f},s}\rangle\cdot E_{f,\overline{f},s}$$
where $\langle,\rangle$ is the pairing of distributions with functions.
To check $\theta$ is well-defined, we must check that, for every square-integrable automorphic form $f$ not in the $L^{2}$-span of $2,2$ pseudo-Eisenstein series, we have 
$$\langle \theta,f\rangle=0$$
To this end, let us check it for $3,1$ pseudo-Eisenstein series $\Psi_{f_{1},\phi_{1}}$ with cuspidal data $f_{1}$ and test function data $\phi_{1}$.  Then
$$\langle \theta,\Psi_{f_{1},\phi_{1}}\rangle\;=\;\Bigg\langle \langle \widetilde{\theta}, \Upsilon_{f}\rangle\cdot \Upsilon_{f}+\langle \widetilde{\theta},\Psi_{f,\overline{f},\phi}^{2,2}\rangle\cdot \Psi_{f,\overline{f},\phi}^{2,2},\Psi^{3,1}_{f_{1},\phi_{1}}\Bigg\rangle$$
This is
$$\Bigg\langle \langle \widetilde{\theta}, \Upsilon_{f}\rangle\cdot \Upsilon_{f},\Psi^{3,1}_{f_{1},\phi_{1}}\Bigg\rangle +\Bigg\langle \langle \widetilde{\theta},\Psi_{f,\overline{f},\phi}^{2,2}\rangle\cdot \Psi_{f,\overline{f},\phi}^{2,2},\Psi^{3,1}_{f_{1},\phi_{1}}\Bigg\rangle\;=\;0$$
The Speh form $\Upsilon_{f}$ is a $\Delta$-eigenfunction.  Furthermore, it is orthogonal to $3,1$ pseudo-Eisenstein series in $L^{2}$.  Indeed, using the adjunction relation,
$$\langle \Upsilon_{f},\Psi^{3,1}_{\varphi_{f_{1},\phi_{1}}}\rangle\;=\;\langle c_{3,1}\Upsilon_{f}, \varphi_{f_{1},\phi_{1}}\rangle$$
Since the $3,1$ constant term of the Speh form $\Upsilon_{f}$ is zero, the above is zero.  Therefore, the Speh form $\Upsilon_{f}$ is orthogonal to $3,1$ pseudo-Eisenstein series.  Since $2,2$ pseudo-Eisenstein series are orthogonal to $3,1$ pseudo-Eisenstein series, we conclude that
$$\langle \theta,\Psi_{f_{1},\phi_{1}}\rangle\;=\;\Bigg\langle \langle \widetilde{\theta}, \Upsilon_{f}\rangle\cdot \Upsilon_{f}+\langle \widetilde{\theta},\Psi_{f,\overline{f},\phi}^{2,2}\rangle\cdot \Psi_{f,\overline{f},\phi}^{2,2},\Psi^{3,1}_{f_{1},\phi_{1}}\Bigg\rangle\;=\;0$$
We now  prove that for a $2,1,1$ pseudo-Eisenstein series $\Psi_{\varphi_{f_{2},\phi_{2},\phi_{3}}}$ with cuspidal data $f_{2}$ and test functions $\phi_{2}$ and $\phi_{3}$, that
$$\langle \theta,\Psi_{\varphi_{f_{2},\phi_{2},\phi_{3}}}\rangle\;=\;0$$
As before, this is just
$$\Bigg\langle \langle \widetilde{\theta}, \Upsilon_{f}\rangle\cdot \Upsilon_{f},\Psi_{\varphi_{f_{2},\phi_{2},\phi_{3}}}^{2,1,1}\Bigg\rangle+\Bigg\langle \langle \widetilde{\theta},\Psi_{f,\overline{f},\phi}^{2,2}\rangle\cdot \Psi_{f,\overline{f},\phi}^{2,2},\Psi^{2,1,1}_{f_{2},\phi_{2},\phi_{3}}\Bigg\rangle$$
The second term is zero, because the pseudo-Eisenstein series are orthogonal.  The first term gives zero.  Indeed
$$\langle \Upsilon_{f},\Psi^{2,1,1}_{\varphi_{f_{2},\phi_{2},\phi_{3}}}\rangle\;=\;\langle c_{2,1,1}\Upsilon_{f},\varphi_{f_{2},\phi_{2},\phi_{3}}\rangle\;=\;0$$
since the $2,1,1$ constant term of the Speh form $\Upsilon_{f}$ is zero.

Let $\Delta_{\theta}$ be $\Delta$ with domain $\text{ker}\;\theta \cap V$.  We will show that parameters for the discrete spectrum $\lambda_{s,f}=s_{f}(s_{f}-2)+s(s-1)$ (if any) of the Friedrichs extension $\widetilde{\Delta}_{\theta}$ are contained in the zero-set of the $L$-function appearing in the period.

To legitimize applying the distribution $\theta$ to cuspidal-data Eisenstein series $E_{f,\overline{f},s}$ requires discussion of local automorphic Sobolev spaces.  Recall that $\theta$ is in the $-1$ global automorphic Sobolev space, so is in the $-1$ local automorphic Sobolev space. As $E_{f,\overline{f},s}$ is in the $+1$ local automorphic Sobolev space, we can apply $\theta$ to it. 
\begin{thm} For $\text{Re}(w) = \frac{1}{2}$, if the equation $(\Delta-\lambda_{w,f})u\;=\;\theta$ has a solution $u\in V$, then $\theta E_{f,\overline{f},w}\;=\;0$.  Conversely, if $\theta E_{f,\overline{f},w}\;=\;0$ for $\text{Re}(w)=\frac{1}{2}$, then there is a solution to that equation in $V$, and the solution is unique with spectral expansion
$$u\;\;=\;\;\frac{\theta(\Upsilon_{f})\cdot \Upsilon_{f}}{(\lambda_{\Upsilon_{f}}-\lambda_{w})}+\frac{1}{4\pi i}\int\limits_{(\frac{1}{2})}\frac{\theta E_{f,\overline{f},1-s}}{\lambda_{s,f}-\lambda_{w,f}}\cdot E_{f,\overline{f},s}\;ds$$
convergent in $V^{+1}$
\end{thm}

\begin{proof} The condition $\theta\in V_{-1}$ is that
$$\int\limits_{\mathbb{R}}\frac{|\theta E_{f,\overline{f},1-s}|^{2}}{1+t^{2}}\;dt\;<\;\infty$$
Thus, $u\in V_{+1}$, and $u$ has a spectral expansion of the form 
$$u\;\;=\;\;{A_{f}\cdot \Upsilon_{f}}\;+\;\frac{1}{4\pi i}\int\limits_{(\frac{1}{2})}A_{s}\cdot E_{f,\overline{f},1-s}\;ds$$
with $t\rightarrow A_{\frac{1}{2}+it}$ in $L^{2}(\mathbb{R})$.  The distribution $\theta$ has spectral expansion in $V_{-1}$,
$$\theta\;\;=\;\;{\theta(\Upsilon_{f})}\cdot \Upsilon_{f}\;+\;\frac{1}{4\pi i}\int\limits_{(\frac{1}{2})}\theta E_{f,\overline{f},1-s}\cdot E_{f,\overline{f},s}\;ds$$
We describe the vector-valued \textit{weak} integrals of [Gelfand 1936] and [Pettis 1938] and summarize the key results. We follow [Bourbaki 1963].\begin{defn} For $X,\mu$ a measure space and $V$ a locally convex, quasi-complete topological vector space, a Gelfand-Pettis (or weak) integral is a vector-valued integral $C_{c}^{0}(X,V)\rightarrow V$ denoted $f\rightarrow I_{f}$ such that for all $\alpha\in V^{*}$, we have
$$\alpha(I_{f})\;=\;\int_{X}\alpha\circ f \;d\mu$$
where the latter is the usual scalar-valued Lebesgue integral.\end{defn}

\begin{prop}Hilbert, Banach, Frechet, and LF spaces together with their weak duals are locally convex, quasi-complete topological vector spaces. \end{prop} 
\begin{prop}Gelfand-Pettis integrals exist and are unique.\end{prop}
\begin{prop}Any continuous linear operator between locally convex, quasi-complete topological vector spaces $T:V\rightarrow W$ commutes with the Gelfand-Pettis integral:
$$T(I_{f})\;=\;I_{Tf}$$\end{prop}
Note that $E_{f,\overline{f},s}$ lies in a local automorphic Sobolev space. By the Gelfand-Pettis theory, if $T:V\rightarrow W$ is a continuous linear map of locally convex topological vector spaces, where convex hulls of compact sets in $V$ have compact closures and if $f$ is a continuous, compactly-supported $V$-valued function on a finite measure space $X$, then the $W$-valued function $T\circ f$ has a Gelfand-Pettis integral, and
$$T\left(\int_{X}f\right)\;=\;\int_{X}T\circ f$$
Let $V=H_{\text{lafc}}^{1}(X)$.  Note that $V$ is a locally convex, quasi-complete topological vector space since it is the completion of $C_{c}^{\infty}(X)$ with respect to a family of semi-norms.  Given a compactly-supported distribution $\theta\in H^{-1}_{\text{gafc}}(X)$, $\theta$ extends to a continuous linear functional $\theta\in H^{-1}_{\text{lafc}}(X)$, by section 7.  Since $\theta$ is a continuous mapping $\theta:H^{-1}_{\text{lafc}}(X)\rightarrow\mathbb{C}$, given a continuous, compactly-supported $H^{1}_{\text{lafc}}(X)$-valued function $f$,
$$\theta\int_{X} f\;=\;\int_{X}\theta\circ f$$
Gelfand-Pettis theory allows us  to move $\theta$ inside the integral.
Thus $$(\lambda_{\Upsilon_{f}}-\lambda_{w})A_{f}\;=\;\theta(\Upsilon_{f})$$ and 
$$(\lambda_{s,f}-\lambda_{w,f})\cdot A_{s}\;=\;\theta E_{f,\overline{f},1-s}$$
The latter equality holds at least in the sense of locally integrable functions.  Letting $w\;=\;\frac{1}{2}+i\tau$, by Cauchy-Schwarz-Bunyakowsky, for any $\epsilon>0$,
$$\int\limits_{\tau-\epsilon}^{\tau+\epsilon}|\theta E_{f,\overline{f}\frac{1}{2}-it}|^{2}\;dt\;\;=\;\;\int\limits_{\tau-\epsilon}^{\tau+\epsilon}|(\lambda_{\frac{1}{2}+it,f}-\lambda_{\frac{1}{2}+i\tau,f})A_{\frac{1}{2}+it}|^{2}\;dt$$
Using $s=\frac{1}{2}+it$ and rewriting the difference of eigenvalues gives us equality of the above with
$$\int_{\tau-\epsilon}^{\tau+\epsilon}|(t-\tau)(t-1+\tau)A_{\frac{1}{2}+it}|^{2}dt \leq\int_{\tau-\epsilon}^{\tau+\epsilon}|t-\tau|^{2}\;dt\cdot\int_{\tau-\epsilon}^{\tau+\epsilon}|(t-i+\tau)A_{\frac{1}{2}+it}|^{2}dt\ll\epsilon^{3}$$
The function $$t\rightarrow \theta E_{f,\overline{f},\frac{1}{2}+it}$$ is continuous, in fact $$s\rightarrow \theta E_{f,\overline{f},s}$$ is meromorphic, since $\theta$ is compactly supported (see [Grothendieck 1954] and [Garrett 2011 e]), so $$\theta E_{f,\overline{f},1-w}\;=\;0$$
Conversely, when $\theta E_{1-w}\;=\;0$, the function $$t\rightarrow \frac{\theta E_{f,\overline{f},\frac{1}{2}-it}}{(\lambda_{\frac{1}{2}+it}-\lambda_{w})}$$ is continuous and square-integrable, assuring $H^{1}$-convergence of the integral 
$$u\;\;=\;\;{\frac{\theta(\Upsilon_{f})\cdot \Upsilon_{f}}{\lambda_{\Upsilon_{f}}-\lambda_{w,f}}}\;+\;\frac{1}{4\pi i}\int\limits_{(\frac{1}{2})}\frac{\theta E_{f,\overline{f},1-s}\cdot E_{f,\overline{f},s}}{(\lambda_{s,f}-\lambda_{w,f})}\;ds$$
this spectral expansion produces a solution of the differential equation.  Any solution in $V^{+1}$ admits such an expansion, and the coefficients are uniquely determined, giving uniqueness.
\end{proof}

Let $X_{a}=\{A,D\in GL_{2} : |\frac{\text{det}A}{\text{det}D}|^{2}=a\}$.  Let $H$ be the subgroup of $GL_{2}\times GL_{2}$ consisting of pairs $(B,C)$ so that $|\text{det}B\cdot\text{det} C|=1$.  The group $H$ acts simply transitively on $X_{a}$, so $X_{a}$ has an $H$-invariant measure. Fix $GL_{2}$ cuspforms $f_{1}$ and $f_{2}$ and define

$$\eta_{a}F\;=\;\int_{Z_{\mathbb{R}}H_{k}\backslash X_{a}}c_{P}(F(a))\cdot f_{1}(A)\cdot f_{2}(D)\;dx$$

\begin{prop} Take $\text{Re}(w)=\frac{1}{2}$.  For $a\gg 1$ such that the support of $\widetilde{\theta}$ is below $h=a$, the constant term $c_{P}u$ of a solution $u\in V^{+1}$ to $(\Delta-\lambda_{w,f})u=\theta$ vanishes for height $h\geq a$.
\end{prop}

\begin{proof} Let $\eta_{a,f_{1}\otimes f_{2}}$ be the functional above.  This functional is in $H^{-\frac{1}{2}-\epsilon}$ for all $\epsilon>0$.  Thus, for $u\in H^{+1}$,
$$\eta_{a,f_{1}\otimes f_{2}}u\;=\;\eta_{a,f_{1}\otimes f_{2}}\left(\frac{\theta(\Upsilon_{f})\cdot \Upsilon_{f}}{(\lambda_{\Upsilon_{f}}-\lambda_{w})\langle 1,1\rangle}+\frac{1}{4\pi i}\int_{(\frac{1}{2})}\frac{\theta E_{f,\overline{f},1-s}}{\lambda_{s}-\lambda_{w}}\cdot E_{f,\overline{f},s}\;ds\right)$$

We can  break up the integral into two tails and a truncated finite part.  The truncated finite part is a continuous, compactly-supported integral of functions in a local automorphic Sobolev space, so Gelfand-Pettis theory allows us to move compactly-supported distributions inside the integral.  The tails are spectral expansions of functions in $H^{+1}$, and since $H^{+1}$ embeds into a local automorphic Sobolev space, the Gelfand-Pettis theory applies there also, allowing us to move the distribution inside the integral.  

$$\frac{\theta(\Upsilon_{f})\cdot \eta_{a,f_{1}\otimes f_{2}}(\Upsilon_{f})}{(\lambda_{\Upsilon_{f}}-\lambda_{w,f})}+\frac{1}{4\pi i}\int_{(\frac{1}{2})}\frac{\theta E_{f,\overline{f},1-s}\cdot\eta_{a,f_{1}\otimes f_{2}}E_{f,\overline{f},s}}{\lambda_{s,f}-\lambda_{w,f}}\;ds$$
This is
$$\theta\left(\frac{\eta_{a,f_{1}\otimes f_{2}}(\Upsilon_{f})\cdot\Upsilon_{f}}{(\lambda_{\Upsilon_{f}}-\lambda_{w,f})}+\frac{1}{4\pi i}\int_{(\frac{1}{2})}\frac{\eta_{a,f_{1}\otimes f_{2}}E_{f,\overline{f},s}}{(\lambda_{s,f}-\lambda_{w,f})}\cdot E_{f,\overline{f},1-s}\;ds\right)$$
which is

$$\theta\left(\frac{\eta_{a,f_{1}\otimes f_{2}}(\Upsilon_{f})\cdot\Upsilon_{f}}{(\lambda_{\Upsilon_{f}}-\lambda_{w,f})}+\frac{1}{4\pi i}\int_{(\frac{1}{2})}\frac{C(a^{1-s}+c_{1-s}a^{s})}{(\lambda_{s,f}-\lambda_{w,f})}\cdot E_{f,\overline{f},1-s}\;ds\right)$$
where $$C=\int_{Z_{\mathbb{R}}H_{k}\backslash X_{a}}f(A)\cdot\overline{f}(D)\cdot f_{1}(A)\cdot f_{2}(D) dx$$
Since $\theta$ has compact support below $h=a$, the last integral need be evaluated only for $h\leq a$.  Using the functional equation $$c_{1-s}E_{f,\overline{f},s}=E_{f,\overline{f},1-s}$$
we see
$$\int_{(\frac{1}{2})}\frac{c_{1-s}a^{s}}{(\lambda_{s,f}-\lambda_{w,f})}\cdot E_{f,\overline{f},s}\;ds\;=\;\int_{(\frac{1}{2})}\frac{a^{1-s}}{(\lambda_{s,f}-\lambda_{w,f})}\cdot E_{f,\overline{f},s}\;ds$$
by changing variables. Thus, for $g$ with $h(g)\leq a$, the integral can be evaluated by residues of vector-valued holomorphic functions as in [Grothendieck] and [Garrett 2011 e].

$$\theta\left(\frac{\eta_{a,f_{1}\otimes f_{2}}(\Upsilon_{f})\cdot\Upsilon_{f}}{(\lambda_{\Upsilon_{f}}-\lambda_{w,f})}+\frac{1}{4\pi i}\int_{(\frac{1}{2})}\frac{C(a^{1-s}+c_{1-s}a^{s})}{(\lambda_{s,f}-\lambda_{w,f})}\cdot E_{f,\overline{f},1-s}\;ds\right)$$

$$=\theta\left(\frac{\eta_{a,f_{1}\otimes f_{2}}(\Upsilon_{f})\cdot\Upsilon_{f}}{(\lambda_{\Upsilon_{f}}-\lambda_{w,f})}+\frac{1}{2\pi i}\int_{(\frac{1}{2})}\frac{C(a^{1-s})}{(\lambda_{s,f}-\lambda_{w,f})}\cdot E_{f,\overline{f},s}\;ds\right)$$
Consider the integral
$$\int_{(\frac{1}{2})}\frac{a^{1-s}\theta E_{f,\overline{f},s}}{(\lambda_{s,f}-\lambda_{w,f})}\;ds$$
With $s=\alpha+iT$, consider a rectangle with vertices $\frac{1}{2}\pm iT$ and $T\pm iT$.  Let $\gamma_{1}$ be the line segment from $\frac{1}{2}+iT$ to $T+iT$.  Let $\gamma_{2}$ be the line segment from $T+iT$ to $T-iT$, and let $\gamma_{3}$ be the line segment from $T-iT$ to $\frac{1}{2}-iT$.  We invoke our assumed subconvexity bound $\theta E_{f,\overline{f},s}\ll |s|^{1-\epsilon}$.  Then we get an estimate
$$\big|\int_{\gamma_{1}}\frac{a^{1-s}\cdot \theta E_{f,\overline{f},s}}{\lambda_{s,f}-\lambda_{w,f}}\;ds\big|\ll \frac{a^{1-s}\cdot |s|^{1-\epsilon}}{|\lambda_{s,f}-\lambda_{w,f}|}\cdot  (T-\frac{1}{2})$$
since $\gamma_{1}$ has length $T-\frac{1}{2}$.  Then,
$$\frac{a^{1-s}\cdot |s|^{1-\epsilon}}{|\lambda_{s,f}-\lambda_{w,f}|}\cdot  (T-\frac{1}{2})\leq \frac{a^{1-s}\cdot |s|^{1-\epsilon}}{|\lambda_{s,f}-\lambda_{w,f}|}\cdot (|s|-\frac{1}{2})\rightarrow 0$$
as $T\rightarrow\infty$, since the denominator is a degree $2$ polynomial in $s$, while the numerator is a polynomial of degree $2-\epsilon$. Likewise, for the curve $\gamma_{2}$, we get an estimate
$$\big|\int_{\gamma_{2}}\frac{a^{1-s}\cdot \theta E_{f,\overline{f},s}}{\lambda_{s,f}-\lambda_{w,f}}\;ds\big|\ll \frac{a^{1-s}\cdot |s|^{1-\epsilon}}{|\lambda_{s,f}-\lambda_{w,f}|}\cdot  (2T)$$
since $\gamma_{1}$ has length $2T$.  Then,
$$\frac{a^{1-s}\cdot |s|^{1-\epsilon}}{|\lambda_{s,f}-\lambda_{w,f}|}\cdot  (T-\frac{1}{2})\leq \frac{a^{1-s}\cdot |s|^{1-\epsilon}}{|\lambda_{s,f}-\lambda_{w,f}|}\cdot (2|s|)\rightarrow 0$$
as $T\rightarrow\infty$, since the denominator is a degree $2$ polynomial in $s$, while the numerator is a polynomial of degree $2-\epsilon$. A similar argument shows that the integrals along $\gamma_{2}$ and $\gamma_{3}$ go to $0$ as $T\rightarrow 0$.  Therefore, the original integral
$$\int_{(\frac{1}{2})}\frac{a^{1-s}\theta E_{f,\overline{f},s}}{(\lambda_{s,f}-\lambda_{w,f})}\;ds\;=\;-2\pi i(\text{sum of residues in the right half-plane})$$
This implies
$$\frac{1}{2\pi i}\int_{(\frac{1}{2})}\frac{a^{1-s}\cdot C\cdot \theta E_{f,\overline{f},s}}{(\lambda_{s,f}-\lambda_{w,f})}\;ds\;=\;-\text{(sum of residues in the right half-plane)}$$
The Eisenstein series $E_{f,\overline{f},s}$ has a simple pole at $s=1$ ([MW] and [Garrett 2011 f]), with residue $$\frac{\eta_{a,f_{1}\otimes f_{2}}(\Upsilon_{f})\cdot \Upsilon_{f}}{(\lambda_{\Upsilon_{f}}-\lambda_{w,f})}$$ Therefore $\theta E_{f,\overline{f},s}$ has residue at $s=1$ given by 
$$\theta \left(\frac{\eta_{a,f_{1}\otimes f_{2}}(\Upsilon_{f})\cdot \Upsilon_{f}}{(\lambda_{\Upsilon_{f}}-\lambda_{w,f})}\right)$$
Thus,
$$\frac{1}{2\pi i}\int_{(\frac{1}{2})}\frac{a^{1-s}\cdot C\cdot \theta E_{f,\overline{f},s}}{(\lambda_{s,f}-\lambda_{w,f})}\;ds\;=\;-\theta\left(\frac{\eta_{a,f_{1}\otimes f_{2}}(\Upsilon_{f})\cdot \Upsilon_{f}}{(\lambda_{\Upsilon_{f}}-\lambda_{w,f})}\right)+\frac{a^{1-w}}{1-2w}\cdot C\cdot \theta E_{f,\overline{f},1-w}$$
Returning to the original equation,
$$\theta\left(\frac{\eta_{a,f_{1}\otimes f_{2}}(\Upsilon_{f})\cdot \Upsilon_{f}}{(\lambda_{\Upsilon_{f}}-\lambda_{w,f})}+\frac{1}{2\pi i}\int_{(\frac{1}{2})}\frac{C(a^{1-s})}{(\lambda_{s,f}-\lambda_{w,f})}\cdot \theta E_{f,\overline{f},s}\;ds\right)\;=\;\frac{a^{1-w}}{1-2w}\cdot C\cdot \theta E_{1-w,f,\overline{f}}$$
Since $\theta E_{1-w,f,\overline{f}}=0$, we are done.

\end{proof}

Recall that $\Phi_{a}$ decomposes  discretely, with (square-integrable) eigenfunctions consisting of truncated Eisenstein series $\wedge^{a}E_{s_{j},f,\overline{f}}$ of Eisenstein series for $s_{j}$ such that
$$a^{s}\cdot f(A)\cdot\overline{f}(D)+a^{1-s}\cdot c_{s}\cdot\overline{f}(A)\cdot f(D)=0$$
where $(A,D)\in X_{a}$, and finitely-many other eigenfunctions. In fact, these truncations are in $H^{\frac{3}{2}-\epsilon}$ for every $\epsilon>0$, since they are solutions to the differential equation $(\Delta-\lambda_{w,f})u=\eta_{a,f_{1}\otimes f_{2}}$.  There are finitely-many other eigenfunctions in addition to these truncated Eisenstein series.

Let $S$ denote the operator $S=1-\tilde{\Delta}_{a}$ with dense domain in $\Phi^{+1}_{a}$ as before.  Then $S$ is an unbounded, symmetric, densely-defined operator. We have the continuous injections

$$\Phi^{+1}_{a}\rightarrow \Phi_{a}\rightarrow \Phi^{-1}_{a}$$
Then $S$ extends by continuity to $S^{\#}:\Phi^{1}_{a}\rightarrow \Phi^{-1}_{a}$.  Since we have the natural inclusion 
$$j:\Phi^{1}_{a}\rightarrow H^{+1}$$
taking adjoints produces an inclusion
$$j^{*}:H^{-1}\rightarrow \Phi^{-1}_{a}$$
Let $j_{\theta}^{*}$ denote the image of $\theta$ under this mapping.Then we can solve the differential equation
$$(S^{\#}-\lambda_{w})u\;=\;j_{\theta}^{*}$$
because $j^{*}_{\theta}\in \Phi^{-1}_{a}$.

\begin{prop} Take $a\gg 1$ such that the (compact) support of $\theta$ is below height $a$.  If necessary, adjust $a$ so that $\theta E_{s_{j}}\neq 0$ for any $s_{j}$ such that 
$$a^{s_{j}}\cdot f(A)\cdot\overline{f}(D)+a^{1-s_{j}}\cdot c_{s_{j}}\cdot\overline{f}(A)\cdot f(D)=0$$
where $(A,D)\in X_{a}$.  For $w$ not among the $s_{j}$, the equation $(S^{\#}-\lambda_{w,f})v=j^{*}_{\theta}$ has a unique solution $v_{w}\in V\cap \Phi_{a}$, this solution lies in $H^{+1}$, and has spectral expansion

$$v_{w}\;=\;\sum_{j}\frac{\theta E_{f,\overline{f},1-s_{j}}}{\lambda_{s_{j},f}-\lambda_{w,f}}\cdot \frac{\wedge^{a}E_{f,\overline{f},s_{j}}}{||\wedge^{a}E_{f,\overline{f},s_{j}}||^{2}}$$

\end{prop}

\begin{proof} As before, any solution is in $H^{+1}$, since $\theta\in H^{-1}$.  The solution $v\in V\cap \Phi_{a}$ has an expansion in terms of the orthogonal bases $\wedge^{a}E_{s_{j},f,\overline{f}}$,

$$v_{w}\;=\;\sum_{j}A_{j}\frac{\wedge^{a}E_{s_{j},f,\overline{f}}}{||\wedge^{a}E_{s_{j},f,\overline{f}}||}\;\;\;\;\;\;\text{convergent in}\;H^{+1}$$
Thus,
$$j^{*}_{\theta}\;=\;(S^{\#}-\lambda_{w,f})v_{w}\;=\;\sum_{j}(\lambda_{s_{j},f}-\lambda_{w,f})A_{j}\frac{\wedge^{a}E_{f,\overline{f},s_{j}}}{||\wedge^{a}E_{f,\overline{f},s_{j}}||}$$
Indeed, since the compact support of $\widetilde{\theta}$ is below $h=a$, the projection $\theta$ to $V$ is in the $H^{-1}$ completion of $V\cap \Phi_{a}$.  Therefore, the expansion of $j^{*}_{\theta}$ in terms of truncated Eisenstein series must be
$$j^{*}_{\theta}\;=\;\sum_{j}\frac{\theta E_{f,\overline{f},s_{j}}\cdot\wedge^{a}E_{f,\overline{f},s_{j}}}{||\wedge^{a}E_{f,\overline{f},s_{j}}||^{2}}$$
noting that $\theta E_{f,\overline{f},s_{j}}=\theta\wedge^{a}E_{f,\overline{f},s_{j}}$.  Thus, the coefficients $A_{j}$ are uniquely determined, also giving uniqueness.

\end{proof}

\begin{prop} Solutions $w$ to the equation $\theta v_{w}=0$ all lie on $(\frac{1}{2}+i\mathbb{R})\cup [0,1]$, and there is exactly one such between each pair $s_{j},s_{j+1}$ of adjacent solutions of $$\big|\frac{\text{det} A}{\text{det} D}\big|^{s}+ \big|\frac{\text{det} A}{\text{det} D}\big|^{1-s}\cdot\frac{\Lambda(2s-1,\pi\otimes \pi^{'})}{\Lambda(2s,\pi\otimes \pi^{'})}\;=\;0.$$
\end{prop}

\begin{proof} Using the expansion of $v_{w}$ in $H^{+1}$ in terms of the truncated Eisenstein series, and that of $\theta\in H^{-1}$ in those terms,
$$\theta v_{w}\;=\;\sum_{j}\frac{|\theta E_{1-s_{j},f,\overline{f}}|^{2}}{(\lambda_{s_{j},f}-\lambda_{w,f})\cdot \|\wedge^{a}E_{s_{j},f,\overline{f}}\|^{2}}$$
Since every $\lambda_{s_{j},f}$ is real, for $\lambda_{w,f}\notin\mathbb{R}$, the imaginary part of $\theta v_{w}$ is easily seen to be nonzero, thus $\theta v_{w}\neq 0$.  Thus, any solution lies in $(\frac{1}{2}+i\mathbb{R})\cup\mathbb{R}$.  For $\lambda_{w}>0$, all the (infinitely-many) summands are nonnegative real, so the sum can not be $0$.  Therefore $w\in(\frac{1}{2}+i\mathbb{R})\cup [0,1]$.

Take $\text{Re}(w)=\frac{1}{2}$ with $\lambda_{s_{j+1},f}<\lambda_{w,f}<\lambda_{s_{j},f}$.  Note that $\theta v_{w}\in\mathbb{R}$ for such $w$.  For $w$ on the vertical line segment between $s_{j}$ and $s_{j+1}$, all summands but the $j^{th}$ and $(j+1)^{th}$ are bounded.  As $w\rightarrow s_{j}$, $0<\lambda_{s_{j},f}-\lambda_{w,f}\rightarrow 0^{+}$ and $\lambda_{s_{j+1},f}-\lambda_{w,f}$ is bounded. As $w\rightarrow s_{j+1}$, $0>\lambda_{s_{j+1},f}-\lambda_{w}\rightarrow 0^{-}$ and $\lambda_{s_{j}}-\lambda_{w}$ is bounded.  Since $w\rightarrow v_{w}$ is a holomorphic $H^{+1}$-valued function, $\theta v_{w}$ is continuous.  By the intermediate value theorem, there is at least one $w$ between $s_{j}$ and $s_{j+1}$ with $\theta v_{w}=0$.

To see that there is at most one $w$ giving $\theta v_{w}=0$ between each adjacent pair $s_{j},s_{j+1}$ again use holomorphy of $w\rightarrow v_{w}$, and take the derivative in $w$:
$$\frac{\partial}{\partial w}\theta v_{w}\;=\;\sum_{j}\frac{|\theta E_{1-s_{j},f}|^{2}\cdot (2w-1)}{(\lambda_{s_{j},f}-\lambda_{w,f})^{2}\cdot \|\wedge^{a}E_{s_{j},f}\|^{2}}$$
Everything is positive real except the purely imaginary $2w-1$, because, in fact, the height $a$ was adjusted so that no $\theta E_{1-s_{j},f}$ vanishes.  That is, away from poles, the derivative is non-vanishing, so all zeros are simple.
Returning to the proof of the theorem: suppose $u\in V$ such that $(S^{\#}-\lambda_{w})u=j_{\theta}^{*}$ with $\text{Re}(w)=\frac{1}{2}$.  For $u$ to be an eigenfunction for $\widetilde{\Delta}_{\theta}$ requires $\theta u=0$ by the nature of the Friedrichs extension.

From above, $\eta_{a} u$ vanishes above a height $a$ depending on the compact support of $\tilde{\theta}$.  Thus, $u\in V\cap \Phi_{a}$, so $u$ must be the solution $v_{w}$ expressed as a linear combination of truncated Eisenstein series, and $\theta v_{w}=0$.  Since there is at most one $w$ giving $\theta v_{w}=0$ between any two adjacent roots $s_{j}$ of $$\big|\frac{\text{det} A}{\text{det} D}\big|^{s}+ \big|\frac{\text{det} A}{\text{det} D}\big|^{1-s}\cdot\frac{\Lambda(2s-1,\pi\otimes \pi^{'})}{\Lambda(2s,\pi\otimes \pi^{'})}\;=\;0$$  giving the constraint.

\end{proof}

\section{L-function background}
Recall that the $2,2$ constant term of the $2,2$ Eisenstein series with fixed cuspidal, everywhere-spherical data $f$ and $\overline{f}$ at height $h=a$ is
$$a^{s}+c_{s}a^{1-s}$$
where
$$c_{s}\;=\;\frac{\Lambda(2(1-s),f\otimes\overline{f})}{\Lambda(2s,f\otimes\overline{f})}$$
A standard argument principle computation shows that the number of zeros of $a^{s}+c_{s}a^{1-s}$ with imaginary parts between $0$ and $T>0$ is
$$N(T)\;=\;\frac{T}{\pi}\text{log}(\frac{T}{2\pi e}+T\text{log}\;a+O(\text{log}\;T))$$
All zeros of $a^{s}+c_{s}a^{1-s}$ are on $\text{Re}(s)=\frac{1}{2}$ for $a\geq 1$.  Recall ([Iwaniec-Kowalski, p.115]) that
$$\text{log}\;L(1+iu,f\otimes\overline{f})-\text{log}\;L(1+it,f\otimes\overline{f})\;=\;O(\frac{\text{log}\;t}{\text{log}\;\text{log}\;t})\cdot (u-t)$$
for $u\geq t$.

\begin{lem} The gaps between consecutive zeros of $a^{s}+c_{s}a^{1-s}$ at height greater than or equal to $T$ are
$$\frac{\pi}{\text{log}\;T}+O(\frac{1}{\text{log}\;\text{log}\;T})$$
\end{lem}

\begin{proof} The condition for the vanishing of $a^{s}+c_{s}a^{1-s}$ can be rewritten as
$$\frac{\Lambda(2s,f\otimes\overline{f})}{\Lambda(2(1-s),f\otimes\overline{f})}\;=\;-1$$
where
$$\Lambda(s,f\otimes f)\;=\;\frac{\pi^{1-s}}{2}\cdot \Gamma(\frac{s+\mu-\nu}{2})\Gamma(\frac{s-\mu+\nu}{2})\Gamma(\frac{s-\mu-\nu}{2})\Gamma(\frac{s+\mu+\nu}{2})\cdot L(s,f\otimes\overline{f})$$
where $\mu$ is the parameter for the principal series $I_{\mu}$ generated by $f$, while $\nu$ is the parameter for the principal series generated by $\overline{f}$.  Therefore, with $s$ on the critical line, we have
\begin{align*}&-1\;=\\
&\frac{\Gamma(\frac{1+2it+\mu-\nu}{2})\Gamma(\frac{1+2it-\mu+\nu}{2})\Gamma(\frac{1+2it-\mu-\nu}{2})\Gamma(\frac{1+2it+\mu+\nu}{2})}{\Gamma(\frac{1-2it+\mu-\nu}{2})\Gamma(\frac{1-2it-\mu+\nu}{2})\Gamma(\frac{1-2it-\mu-\nu}{2})\Gamma(\frac{1-2it+\mu+\nu}{2})} \pi^{1-2it}\frac{L(1+2it,f\otimes\overline{f})}{L(1-2it,f\otimes\overline{f})}\end{align*}

All the factors on the right-hand side are of absolute value $1$. The count of zeros as $t=\text{Im}(s)$ moves from $0$ to $T$ is the number of times the right-hand side assumes the value $-1$.  Regularity is entailed by upper and lower bounds for the derivative of the logarithm of that right-hand side, for large $t$.  Observe that
$$\frac{d}{dt}\text{Im}\;\text{log}\;\frac{\Gamma(a+it)}{\Gamma(a-it)}\;=\;2\frac{d}{dt}\text{Im}\;\text{log}\;\Gamma(a+it)$$
From the Stirling asymptotic,
$$\text{log}\;\Gamma(s)\;=\;(s-\frac{1}{2})\text{log}\;s-s+\frac{1}{2}\text{log}\;2\pi+O_{\delta}(\frac{1}{s})$$
in $\text{Re}(s)\geq \delta>0$. From this, we have
$$\text{log}\;\Gamma(a+it)\;=\;it\text{log}\;(a+it)-(a+it)+\frac{1}{2}\text{log}\;2\pi+O_{\delta}(\frac{1}{a+it})$$
$$=\;it\big(i(\pi+O(\frac{1}{t}))+\text{log}\;t+O(\frac{1}{t^{2}})\big)-(a+it)+\frac{1}{2\pi}\text{log}\;2\pi+O_{\delta}(\frac{1}{a+it})$$
Therefore,
$$\text{Im}\;\text{log}\;\Gamma(a+it)\;=\;t\text{log}\;t-t+O(\frac{1}{t})$$
Consider, for $0<\delta\ll t$,
$$\text{Im}\;\text{log}\;\Gamma(a+i(t+\delta))-\text{Im}\;\text{log}\;\Gamma(a+it)\;=\;\big((t+\delta)\text{log}\;(t+\delta)-(t+\delta)\big)-(t\text{log}\;t-t)+O(\frac{1}{t})$$
Which is
$$=\delta\text{log}\;t-(t+\delta)\frac{\delta}{t}-\delta+O_{\delta}(\frac{1}{t})=\delta\text{log}\;t-2\delta+O_{\delta}(\frac{1}{t})$$
In particular, for $0<\delta\leq \frac{1}{\text{log}\;t}$,
$$\text{Im}\;\text{log}\;\Gamma(a+i(t+\delta))-\text{Im}\;\text{log}\;\Gamma(a+it)\;=\;\delta\text{log}\;t+O(\frac{1}{\text{log}\;t})$$
Let
$$f(t)\;=\;\frac{\Gamma(\frac{1+2it+\mu-\nu}{2})\Gamma(\frac{1+2it-\mu+\nu}{2})\Gamma(\frac{1+2it-\mu-\nu}{2})\Gamma(\frac{1+2it+\mu+\nu}{2})}{\Gamma(\frac{1-2it+\mu-\nu}{2})\Gamma(\frac{1-2it-\mu+\nu}{2})\Gamma(\frac{1-2it-\mu-\nu}{2})\Gamma(\frac{1-2it+\mu+\nu}{2})}$$
Then using the calculation above,
$$\text{Im}\;\text{log}\;f(t+\delta)-\text{Im}\;\text{log}\;f(t)\;=\;4\delta\text{log}\;t+O(\frac{1}{\text{log}\;t})$$
The result on $L(1+it,f\otimes\overline{f})$ quoted above gives
$$\text{log}\;L(1+2i(t+\delta),f\otimes\overline{f})-\text{log}\;L(1+2it,f\otimes\overline{f})\;=\;O(\frac{\text{log}\;t}{\text{log}\;\text{log}\;t})$$
Therefore,
$$\text{Im}\;\text{log}\;\Lambda(1+2i(t+\delta),f\otimes\overline{f})-\text{Im}\;\text{log}\;\Lambda(1+2it,f\otimes\overline{f})\;=\;4\delta\text{log}\;t+O(\frac{\text{log}\;t}{\text{log}\;\text{log}\;t})\cdot\delta$$
The presence of the $4$ being due to the four factors of $\Gamma$ appearing.  Thus, if $t$ gives a $0$ of the constant term, the next $t'=t+\delta$ giving a zero of the constant term must be such that
$$4\delta\text{log}\;t+O(\frac{\text{log}\;t}{\text{log}\;\text{log}\;t})\cdot\delta\geq 2\pi$$
On the other hand, when that inequality is satisfied, then the unit circle will have been traversed, and a zero of the constant term occurs.

\end{proof}

Since periods of automorphic forms produce $L$-functions, it is anticipated that $\theta E_{s}$ will produce a self-adjoint, degree $4$ $L$-function, with a corresponding pair-correlation conjecture.  That is, given $\epsilon>0$, there are many pairs of zeros of $\theta E_{s}$ within $\epsilon$ of each other.  The previous section exhibits the zeros $w$ of $\theta E_{s}$ as paramaters of the discrete spectrum of $\widetilde{\Delta}_{\theta}$.  Since parameters of the discrete spectra interlace with the zeros $s_{j}$ of $a^{s}+c_{s}a^{1-s}$, and these are regularly spaced by the argument above, the discrete spectrum is presumably sparse.

\section{Appendix I: Harmonic Analysis on $GL_{3}$}
Given a parabolic $P$ in $G=GL_{3}$ and a function $f$ on $Z_{\mathbb{A}}G_{k}\backslash G_{\mathbb{A}}$, the constant term of $f$ along $P$ is
$$c_{P}f(g)\;=\;\int_{N_{k}\backslash N_{\mathbb{A}}}f(ng)\;dn$$
where $N$ is the unipotent radical of $P$.  An automorphic form satisfies the Gelfand condition if, for all maximal parabolics $P$, the constant term along $P$ is zero.  If such a function is also $\mathfrak{z}$-finite, and $K$-finite, it is called a cuspform.  Since the right $G_{\mathbb{A}}$-action commutes with taking constant terms, the space of functions meeting Gelfand's condition is $G_{\mathbb{A}}$-stable, so is a sub-representation of $L^{2}(Z_{\mathbb{A}}G_{k}\backslash G_{\mathbb{A}})$.  
Godement, Selberg, and Piatetski-Shapiro showed that integral operators on this space are compact.  Specifically, for $\varphi\in C_{c}^{\infty}(G)$, the operator $f\rightarrow \varphi\cdot f$ gives a compact operator from $L^{2}_{\text{cfm}}(Z_{\mathbb{A}}G_{k}\backslash G_{\mathbb{A}})$ to itself.  Here,
$$(\varphi\cdot f)(y)\;=\;\int_{Z_{\mathbb{A}}G_{k}\backslash G_{\mathbb{A}}}\varphi(x)\cdot f(yx)\;dx$$
By the spectral theorem for compact operators, this sub-representation decomposes into a direct sum of irreducibles, each appearing with finite multiplicity.  To decompose the remainder of $L^{2}$ demands an understanding of the continuous spectrum, consisting of pseudo-Eisenstein series.  We classify non-cuspidal automorphic forms according to their cuspidal support, the smallest parabolic on which they have a nonzero constant term.  In $GL_{3}$, there are three conjugacy classes of proper parabolic subgroups.  We will consider the standard parabolic subgroups $P^{3}=GL_{3}$, $P^{2,1}$ and $P^{1,2}$ the maximal parabolics, and $P^{1,1,1}$ the minimal parabolic.

Given the $2,1$ parabolic, define a smooth, compactly-supported function $\varphi$ by
$$\varphi(\left( \begin{array}{cc}
A & * \\
0& d\end{array} \right))=\varphi_{\phi_{1},f_{1}}(\left( \begin{array}{cc}
A & * \\
0& d\end{array} \right))=\phi(\frac{\text{det}A}{d^{2}})\cdot f_{1}(A)$$
where $f_{1}$ is a $GL_{2}$-cuspform and $\phi$ is a compactly-supported smooth function. The pseudo-Eisenstein series attached to $\varphi$ is the function 
$$\Psi_{\varphi}^{2,1}(g)\;=\;\sum_{P_{k}\backslash G_{k}}\varphi(\gamma g)$$
Likewise, given the $2,1$ parabolic, define a function $\psi$ by
$$\psi(\left( \begin{array}{cc}
a & * \\
0& D\end{array} \right))=\psi_{\phi_{2},f_{2}}(\left( \begin{array}{cc}
a & * \\
0& D\end{array} \right))=\phi(\frac{a^{2}}{\text{det}D})\cdot f_{2}(D)$$
again $\phi_{2}$ is a compactly-supported smooth function and $f_{2}$ is a cuspform on $GL_{2}$.
Finally, given the $1,1,1$ parabolic, define a function $\psi$ by
$$\psi(\left( \begin{array}{ccc}
a & *& * \\
0& b & * \\ 
0 & 0 & c \end{array}\right))=\psi_{g_{1},g_{2}}(\left( \begin{array}{ccc}
a & *& * \\
0& b & * \\ 
0 & 0 & c \end{array}\right))=g_{1}(\frac{a}{b})\cdot g_{2}(\frac{b}{c})$$
where $g_{1}$ and $g_{2}$ are compactly-supported smooth functions. Then,
$$\Psi_{\psi}(g)\;=\;\sum_{\gamma\in P_{k}\backslash G_{k}}\psi(\gamma\cdot g)$$
Next, we exhibit the spaces spanned by non-associate pseudo-Eisenstein series as the orthogonal complement to $L^{2}$ cuspforms.

\begin{prop} For any square-integrable automorphic form $f$ and any pseudo-Eisenstein series $\Psi_{\varphi}^{P}$, with $P$ a parabolic subgroup
$$\langle f,\Psi_{\varphi}^{P}\rangle_{Z_{\mathbb{A}}G_{k}\backslash G_{\mathbb{A}}}\;=\;\langle c_{P}f,\varphi\rangle_{Z_{\mathbb{A}}N_{\mathbb{A}}^{P}M_{k}^{P}\backslash G_{\mathbb{A}}}$$
\end{prop}

\begin{proof}  The proof involves a standard unwinding argument.  Observe that
$$\langle f,\Psi_{\varphi}^{P}\rangle_{Z_{\mathbb{A}}G_{k}\backslash G_{\mathbb{A}}}\;=\;\int\limits_{Z_{\mathbb{A}}G_{k}\backslash G_{\mathbb{A}}}f(g)\cdot\overline{\Psi_{\varphi}^{P}(g)}\;dg\;=\int\limits_{Z_{\mathbb{A}}G_{k}\backslash G_{\mathbb{A}}}f(g)(\sum_{\gamma\in P_{k}\backslash G_{k}}\overline{\varphi(\gamma\cdot g)})\;dg$$
This is
\begin{align*}&=\int\limits_{Z_{\mathbb{A}}P_{k}\backslash G_{\mathbb{A}}}f(g)\overline{\varphi(g)}\;dg=\int\limits_{Z_{\mathbb{A}}N_{k}M_{k}\backslash G_{\mathbb{A}}}f(g)\overline{\varphi(g)}\;dg\;\\&=\;\int\limits_{Z_{\mathbb{A}}N_{\mathbb{A}}M_{k}\backslash G_{\mathbb{A}}}\int\limits_{N_{k}\backslash N_{\mathbb{A}}}f(ng)\overline{\varphi(ng)}\;dn\;dg\\
&=\int\limits_{Z_{\mathbb{A}}N_{\mathbb{A}}M_{k}\backslash G_{\mathbb{A}}}(\int\limits_{N_{k}\backslash N_{\mathbb{A}}}f(ng)\;dn)\overline{\varphi(g)}\;dg\\
&=\langle c_{P}f,\varphi\rangle_{Z_{\mathbb{A}}N_{\mathbb{A}}^{P}M_{k}^{P}\backslash G_{\mathbb{A}}}
\end{align*}
\end{proof}

The space spanned by $P^{1,2}$ pseudo-Eisenstein series is the same as the space spanned by $P^{2,1}$ pseudo-Eisenstein series.  More generally, pseudo-Eisenstein series of associate parabolics span the same space.  The $L^{2}$ decomposition is that $L^{2}(Z_{\mathbb{A}}G_{k}\backslash G_{\mathbb{A}})$ decomposes as the direct sum of cuspforms together with the spaces spanned by the minimal parabolic pseudo-Eisesntein series and $2,1$ pseudo-Eisenstein series with cuspidal data.
Following the $GL_{2}$ case, we will decompose the pseudo-Eisenstein series into genuine Eisenstein series.  There are several kinds of Eisenstein series in $GL_{3}$.  For a parabolic $P$, the $P$-Eisenstein series is
$$E_{\lambda}\;=\;\sum_{\gamma\in P_{k}\backslash G_{k}}f_{\lambda}(\gamma g)$$
where $f_{\lambda}$ is a spherical vector in a representation $\lambda$ of $M^{P}$, extended to a $P$-representation by left $N$-invariance, and induced up to $G$.  One of the chief ingredients in the spectral decomposition for $GL_{2}$ pseudo-Eisenstein series was that the Levi component was a product of copies of $GL_{1}$, allowing us to reduce to the spectral theory for $GL_{1}$.  Unfortunately, this is no longer true for non-minimal parabolic pseudo-Eisenstein series, because the Levi component contains a copy of $GL_{2}$.
Therefore, we will first decompose the minimal parabolic pseudo-Eisenstein series.  To this end, we need the functional equation of the Eisenstein series.  Because of the increase in dimension, there is more than one functional equation.  The symmetries of the Eisenstein series can be described in terms of the action of the Weyl group $W$ on the standard maximal torus $A$, on its Lie algebra $\mathfrak{a}$, and the dual space $i\mathfrak{a}^{*}$.  We describe the constant term and the functional equations of the Eisenstein series and use them in the spectral decomposition.
For $GL_{n}$ the standard maximal torus $A$ is the product of $m$ copies of $GL_{1}$, and representations of $A$ are products of representations of $GL_{1}$; in the unramified case, these representations are just $y\rightarrow y^{s_{i}}$, for complex $s_{i}$.  The Weyl group $W$ can be identified with the group of permutation matrices in $GL_{n}$.  It acts on $A$ by permuting the copies of $GL_{1}$, and it acts on the dual in the canonical way, permuting the $s_{i}$, in the unramified case.
We give a preliminary sketch of the constant term and functional equation of the Eisenstein series, with details to be filled in later.  The constant term of the Eisenstein series (along the minimal parabolic) has the form
$$c_{P}(E_{\lambda})\;=\;\sum_{w\in W}c_{w}(\lambda)\cdot w_{\lambda}$$
where $w_{\lambda}$ is the image of $\lambda$ under the action of $w$ and $c_{w}(\lambda)$ is a constant depending on $w$ and $\lambda$ with $c_{1}(\lambda)=1$.  The Eisenstein series has functional equations
$$c_{w}(\lambda)\cdot E_{\lambda}=E_{w_{\lambda}}\;\;\;\text{for all}\;w\in W$$
We start the decomposition of $\Psi_{\varphi}$ by using the spectral expansion of its data $\varphi$.  Recall that $\varphi$ is left $N_{\mathbb{A}}$-invariant, so it is essentially a function on the Levi component, which is a product of copies of $k^{\times}\backslash\mathbb{J}$.  By Fujisaki's lemma, this is the  product of a ray with a compact abelian group.  We assume that the compact abelian group is trivial.  So spectrally decomposing $\varphi$ is a higher-dimensional version of Mellin inversion.
$$\varphi\;=\;\int \langle \varphi,\lambda \rangle \cdot \lambda\;d\lambda$$
Winding up gives
$$\Psi_{\varphi}(g)\;=\;\int_{i\mathfrak{a}^{*}}\langle\varphi,\lambda\rangle\cdot E_{\lambda}(g)\;d\lambda$$
In order for this decomposition to be valid, the parameters of $\lambda$ must have $\text{Re}(s_{i})\gg 1$.  However, in order to use the symmetries of the functional equations, we need the parameters to be on the critical line. In moving the contours, we pick up some residues, which are mercifully constants.  Breaking up the dual space according to Weyl chambers and changing variables,

$$\Psi_{\varphi}(g)-(\text{residues})\;=\;\sum_{w\in W}\int_{\text{1st Weyl chamber}}\langle\varphi, w_{\lambda}\rangle \cdot E_{w_{\lambda}}(g)\;d\lambda$$
Now using the functional equations,
$$\Psi_{\varphi}(g)-(\text{residues})\;=\;\sum_{w\in W}\int_{(1)}\langle\varphi, w_{\lambda}\rangle\cdot c_{w}(\lambda)\cdot E_{\lambda}(g)\;d\lambda$$
This is
$$\int_{(1)}\sum_{w\in W}\langle \varphi, c_{w}(\lambda)w_{\lambda}\rangle\cdot E_{\lambda}(g)\;d\lambda$$
We recognize the constant term of the Eisenstein series and apply the adjointness relation
$$\sum_{w\in W}\langle\varphi, c_{w}(\lambda)w_{\lambda}\rangle\;=\;\langle\varphi,c_{P}E_{\lambda}\rangle\;=\;\langle\Psi_{\varphi},E_{\lambda}\rangle$$
So we have
$$\Psi_{\varphi}(g)\;=\;\sum_{(1)}\langle\Psi_{\varphi},E_{\lambda}\rangle\cdot E_{\lambda}(g)\;d\lambda\;+\;\text{residues}$$
Our next goal is to show that the remaining automorphic forms, namely those with cuspidal support $P^{1,2}$ or $P^{2,1}$ can be written as superpositions of genuine $P^{2,1}$ Eisenstein series.  To do this it suffices to decompose $P^{2,1}$ and $P^{1,2}$ pseudo-Eisenstein series with cuspidal support.  Let $P=P^{1,2}$ and $Q=P^{2,1}$. We start by looking more carefully at pseudo-Eisenstein series with cuspidal data.  The data for a $P$ pseudo-Eisenstein series is smooth, compactly-supported, and left $Z_{\mathbb{A}}M_{k}^{P}N_{\mathbb{A}}^{P}$-invariant.  Assume the data is spherical.  Then the function is determined by its behavior on $Z_{\mathbb{A}}M_{k}^{P}\backslash M_{\mathbb{A}}^{P}$. In contrast to the minimal parabolic case, this is not a product of copies of $GL_{1}$, so we can not use the $GL_{1}$ spectral theory of Mellin inversion to establish the decomposition.  Instead the quotient is isomorphic to $GL_{2}(k)\backslash G_{\mathbb{A}}$, so we will use the spectral theory for $GL_{2}$.  If $\eta$ is the data for a $P^{2,1}$ pseudo-Eisenstein series $\Psi_{\eta}$, we can write $\eta$ as a tensor product $\eta\;=\;f\times \nu$ on
$$Z_{GL_{2}(\mathbb{A})}GL_{2}(k)\backslash GL_{2}(\mathbb{A})\cdot Z_{GL_{2}(k)}\backslash Z_{GL_{2}(\mathbb{A})}$$
Saying that the data is cuspidal means that $f$ is a cuspform.  Similarly the data $\varphi=\varphi_{F,s}$ for a $P^{2,1}$-Eisenstein series is the tensor product of a $GL_{2}$ cusp form $F$ and a character $\lambda_{s}=|.|^{s}$ on $GL_{1}$. We show that $\Psi_{f,\nu}$ is the superposition of Eisenstein series $E_{F,s}$ where $F$ ranges over an orthonormal basis of cuspforms and $s$ is on a vertical line.\\

Using the spectral expansions of $f$ and $\nu$,
$$\eta\;=\;f\otimes\nu\;=\;\big(\sum_{\text{cfms} F}\langle f,F\rangle\cdot F\big)\cdot\big(\int_{s}\langle\nu,\lambda_{s}\rangle\cdot\lambda_{s}\;ds\big)\;=\;\sum_{\text{cfms} F}\int_{s}\langle \eta_{f,\nu},\varphi_{F,s}\rangle\cdot\varphi_{F,s}\;ds$$
So the pseudo-Eisenstein series can be re-expressed as a superposition of Eisenstein series
$$\Psi_{f,\nu}(g)\;=\;\sum_{\text{cfms} F}\int_{s}\langle\eta_{f,\nu},\varphi_{F,s}\rangle\cdot E_{F,s}(g)\;ds$$
In fact the coefficient $\langle\eta,\varphi\rangle_{GL_{2}}$ is the same as the pairing $\langle \Psi_{\eta},E_{\varphi}\rangle_{GL_{3}}$, since
$$\langle \Psi_{\eta},E_{\varphi}\rangle\;=\;\langle c_{P}(\Psi_{\eta}),\varphi\rangle\;=\;\langle \eta,\nu\rangle$$
So the spectral expansion is 
$$\Psi_{f,\nu}\;=\;\sum_{\text{cfms} F}\int_{s}\langle \Psi_{f,\nu}, E_{F,s}\rangle\cdot E_{F,s}(g)\;ds$$
So far, we have not had to shift the line of integration to the critical line $\frac{1}{2}+i\mathbb{R}$.\\
It now remains to show that pseudo-Eisenstein series for the associate parabolic $Q=P^{1,2}$ can also be decomposed into superpositions of $P$-Eisenstein series.  For maximal parabolic pseudo-Eisenstein series, the functional equation does not relate the Eisenstein series to itself but rather to the Eisenstein series of the associate parabolic.  We will use this functional equation to obtain the decomposition of associate parabolic pseudo-Eisenstein series.  The functional equation is

$$E_{F,s}^{Q}\;=\;b_{F,s}\cdot E_{F,1-s}^{P}$$
where $b_{f,s}$ is a meromorphic function that appears in the computation of the constant term along $P$ of the $Q$-Eisenstein series.\\

We consider a $Q$-pseudo-Eisenstein series $\Psi_{f,\nu}^{Q}$ with cuspidal data.  By the same arguments used above to obtain the decomposition of $P$-pseudo-Eisenstein series, we can decompose $\Psi_{f,\nu}^{P}$ into a superposition of $Q$-Eisenstein series
$$\Psi_{f,\nu}^{Q}(g)\;=\;\sum_{\text{cfms} F}\int_{s}\langle \eta_{f,\nu},\varphi_{F,s}\rangle\cdot E_{F,s}^{Q}(g)$$
Now using the functional equation,
$$\Psi_{f,\nu}^{Q}(g)\;=\;\sum_{\text{cfms} F}\int_{s}\langle \Psi_{f,\nu}^{Q},b_{F,s}\cdot E_{F,1-s}^{P}\rangle\cdot b_{F,s}\cdot E_{F,1-s}^{P}$$
$$=\;\sum_{\text{cfms} F}\int_{s}\langle \Psi_{f,\nu}^{Q},E_{F,1-s}^{P}\rangle\cdot |b_{F,s}|^{2}\cdot E_{F,1-s}^{P}$$
So we have a decomposition of $Q$-pseudo-Eisenstein series (with cuspidal data) into $P$-Eisenstein series (with cuspidal data).  In order to use the functional equation we moved some contours, but there are no poles, so no residues are acquired.\\

We have described the spectral decomposition of $L^{2}(Z_{\mathbb{A}}G_{k}\backslash G_{\mathbb{A}})$ as the direct sum/integral of irreducibles.  Any automorphic form $\xi$ can be written as
$$\xi\;=\;\sum_{GL_{3}\;\text{cfms}\;f}\langle\xi,f\rangle\cdot f + \sum_{GL_{2}\;\text{cfms}\;F}\int_{s}\langle\xi,E_{F,s}^{2,1}\rangle\cdot E_{F,s}^{2,1} + \int_{(1)}\langle\xi,E_{\lambda}^{1,1,1}\rangle\cdot E_{\lambda}^{1,1,1}\;d\lambda + \frac{\langle\xi,1\rangle}{\langle 1,1\rangle}$$
This converges in $L^{2}$.

\newpage\noindent \textbf{\large{Bibliography}}\\

\noindent [Bourbaki 1963] N. Bourbaki, \textit{Topological Vector Spaces, ch. 1-5}, Springer-Verlag, 1987.\\

\noindent [CdV 1982,83] Y.~Colin de Verdi\`{e}re, \textit{Pseudo-laplaciens, I, II,}\; Ann. Inst. Fourier (Grenoble) \textbf{32}\; (1982) no. 3, 275-286, \textbf{33}\; no. 2, 87--113.\\

\noindent [CdV 1981] Y. Colin de Verdi\`{e}re, \textit{Une nouvelle demonstration du prolongement meromorphe series d'Eisenstein}, C. R. Acad. Sci. Paris Ser. I Math. \textbf{293} (1981), no. 7, 361-363.\\

\noindent [DeCelles 2011a] A. DeCelles, \textit{Fundamental solution for $(\Delta-\lambda_ {z})^{\nu}$ on a symmetric space $G/K$}, arXiv:1104.4313 [math.RT].\\

\noindent [DeCelles 2011b] A. DeCelles, \textit{Automorphic partial differential equations and spectral theory with applications to number theory}, Ph.D thesis, University of Minnesota, 2011.\\

\noindent [DeCelles 2012] A. DeCelles {\it An exact formula relating lattice points
in symmetric spaces to the automorphic spectrum}, Illinois
J. Math. {\bf 56} (2012), 805-823.\\

\noindent [Fadeev 1967] L.~D.~Faddeev, \textit{Expansion in eigenfunctions of the Laplace operator on the fundamental domain of a discrete group on the Lobacevskii plane}, Trudy Moskov. math 0-ba \textbf{17}, 323--350 (1967).\\

\noindent [Faddeev-Pavlov 1972] L.~Faddeev, B. S. Pavlov, \textit{Scattering theory and automorphic functions,} Seminar Steklov Math. Inst \textbf{27} (1972), 161--193.\\

\noindent [Feigon-Lapid-Offen 2012] B. Feigon, E. Lapid, O. Offen, \textit{On representations distinguished by unitary groups}, Publ. Math. Inst. Hautes Etudes Sci. (2012), 185-323.\\

\noindent [Garrett 2010] P. Garrett, \textit{Examples in automorphic spectral theory}\\
\noindent http://www.math.umn.edu/~{}\text{garrett/m/v/durham.pdf}\\

\noindent [Garrett 2011 a] P. Garrett, \textit{Colin de Verdi\`{e}re's meromorphic continuation of Eisenstein series}\\
\noindent http://www.math.umn.edu/\~{}\text{garrett/m/v/cdv\_eis.pdf}\\

\noindent [Garrett 2011 b] P. Garrett, \textit{Pseudo-cuspforms, pseudo-Laplaciens}\\
\noindent www.math.umn.edu/\~{}\text{garrett/m/v/pseudo-cuspforms.pdf}\\

\noindent [Garrett 2011 c] P. Garrett, \textit{Unbounded operators, Friedrichs' extension theorem}\\
\noindent www.math.umn.edu/\~{}\text{garrett/m/v/friedrichs.pdf}\\

\noindent [Garrett 2011 d] P. Garrett, \textit{Vector-Valued Integrals}\\
\noindent www.math.umn.edu/\~{}\text{garrett/m/fun/Notes/$07$\_vv\_integrals.pdf}\\

\noindent [Garrett 2011 e] P. Garrett \textit{Holomorphic vector-valued functions}\\
\noindent http://www.math.umn.edu/\~{}\text{garrett/m/mfms/notes\_c/cont\_afc\_spec.pdf}\\

\noindent [Garrett 2011 f] P. Garrett \textit{Slightly non-trivial example of Maass-Selberg relations}\\
\noindent http://www.math.umn.edu/\~{}\text{garrett/m/mfms/notes\_c/cont\_afc\_spec.pdf}\\

\noindent [Garrett 2012] P. Garrett, \textit{Most continuous automorphic spectrum for} $GL_n$\\
\noindent www.math.umn.edu/\~{}\text{garrett/m/v/gln\_cont\_spec.pdf}\\

\noindent [Garrett 2014] P. Garrett, \textit{Discrete decomposition of pseudo-cuspforms on $GL_{n}$}\\
\noindent www.math.umn.edu/\~{}\text{garrett/m/v/discreteness\_pseudo\_GLn.pdf}\\

\noindent [Gelfand 1936], I.M. Gelfand, \textit{Sur un lemme de la theorie des espaces lineaires}, Comm. Inst. Sci. Math de Kharkoff, no. 4, \textbf{13} (1936),35-40.\\

\noindent [Grothendieck 1952], A. Grothendieck, \textit{Sur cetains espaces de fonctions holomorphes I, II, III}, J. Reine Angew. Math. \textbf{192} (1953), 35-64 and 77-95.\\

\noindent[Grothendieck 1955], A. Grothendieck, \textit{Produits tensoriels topologiques et espaces nucleaires}, Mem. Am. Math. Soc. \textbf{16}, 1955.\\

\noindent [Grubb 2009] G.~Grubb, \textit{Distributions and operators}, Springer-Verlag, 2009.\\

\noindent [Haas 1977] H. Haas, \textit{Numerische Berechnung der Eigenwerte der Differentialgleichung $y^{2}\Delta u+\lambda u=0$ fur ein unendliches Gebiet im $\mathbb{R}^{2}$}, Diplomarbeit, Universit{\"a}t Heidelberg (1977) 155.pp.\\

\noindent [Harish-Chandra 1968] Harish--Chandra, \textit{Automorphic Forms on semi-simple Lie Groups}, Lecture Notes in Mathematics, no. 62, Springer-Verlag, Berlin, Heidelberg, New York, 1968.\\

\noindent [Hejhal 1981] D. Hejhal, \textit{Some observations concerning eigenvalues of the Laplacian and Dirichlet L-series} in \textit{Recent Progress in Analytic Number Theory}, ed. H. Halberstam and C. Hooley, vol. 2, Academic Press,  NY, 1981, 95--110.\\

\noindent [Hejhal 1976] D. Hejhal \textit{The Selberg trace formula for $SL_{2}(\mathbb{R})$ I}, Lecture Notes In Math. \textbf{548}, Springer-Verlag, Berlin, 1976.\\

\noindent [Hejhal 1983] D. Hejhal \textit{The Selberg trace formula for $SL_{2}(\mathbb{R})$ II}, Lecture Notes In Math. \textbf{1001}, Springer-Verlag, Berlin, 1983.\\

\noindent [Hejhal 1990] D. Hejhal \textit{On a result of G. P\'olya concerning the Riemann $\xi$-function}, J. d'Analyse Mathematique 55 (1990), pp. 59-95.\\

\noindent [Hejhal 1994] D. Hejhal \textit{On the triple correlation of zeros of the zeta function} Internat. Math. Res. Notices 7, pp. 293-302, 1994.\\

\noindent [Jacquet 1983] H. Jacquet, \textit{On the residual spectrum of $GL(n)$}, in \textit{Lie Group Representations, II}, Lecture notes in Math. 1041, Springer-Verlag, 1983.\\

\noindent [Jacquet-Lapid-Rogowski 1999] H. Jacquet, E. Lapid, J. Rogowski, \textit{Periods of automorphic forms}, J. Amer. Math. Soc. \textbf{12} (1999), no. 1, 173-240.\\

\noindent [Lang 1970] S. Lang, \textit{Algebraic number theory}, Addison-Wesley, 1970.\\

\noindent [Lapid-Offen 2007] E. Lapid, O. Offen, \textit{Compact unitary periods}, Compos. Math. \textbf{143} 3 (2007), no. 2, 323-338.\\

\noindent [Langlands 1976] R.~ P.~Langlands, \textit{On the Functional Equations satisfied by Eisenstein series}, Lecture Notes in Mathematics no. 544, Springer-Verlag, New York, 1976.\\

\noindent [Lax-Phillips 1976] P. Lax, R. Phillips, \textit{Scattering theory for automorphic functions}, Annals of Math. Studies, Princeton, 1976.\\

\noindent [Maass 1949] H. Maass, \textit{Uber eine neue Art von nichtanalytischen automorphen Funktionen}, Math. Ann. \textbf{121} (1949), 141--183.\\

\noindent [Moeglin--Waldspurger 1989] C. Moeglin, J. L. Waldspurger, \textit{Le spectre residuel de $GL(n)$}, with appendix \textit{ Poles
des fonctions $L$ de pairs pour $GL(n)$}, Ann. Sci. Ecole Norm. Sup. \textbf{22} (1989), 605--674.\\

\noindent [Moeglin--Waldspurger 1995] C. Moeglin, J. L. Waldspurger, \textit{Spectral decompositions and Eisenstein series}, Cambridge Univ. Press, Cambridge, 1995.\\

\noindent [Pettis] B. J. Pettis, \textit{On integration in vectorspaces}, Trans. AMS \textbf{44}, 1938, 277-304.\\

\noindent [P\'olya] G. P\'olya, \textit{Bemerkung {\"u}ber die Integraldarstellung der Riemannschen $\xi$-Funktion} Acta Math. \textbf{48} (1926), 305-317.\\

\noindent[Rudin 1991] W. Rudin, \textit{Functional Analysis}, second edition, McGraw-Hill, 1991.\\

\noindent[Rudnick-Sarnak 1994] Z. Rudnick, P. Sarnak, \textit{The n-level correlations of zeros of the zeta function}, C.R. Acad. Sci. Paris \textbf{319}, 1027-1032, 1994.\\

\noindent [Shahidi 2010] F. Shahidi, \textit{Eisenstein series and automorphic L-functions}, AMS Colloquium Publ, \textbf{58}, AMS, 2010.\\

\end{document}